\documentclass[11pt, twoside,leqno]{article}
\usepackage{mathrsfs}
\usepackage{amssymb}
\usepackage{amsmath}
\usepackage{amsthm}
\usepackage{amsfonts}
\usepackage{color}
\usepackage{latexsym}
\usepackage{txfonts}
\usepackage{indentfirst}
\usepackage{soul}
\usepackage[normalem]{ulem}
\usepackage{enumitem}

\def\red{\color{red}}

\allowdisplaybreaks

\def\rr{\mathbb{R}}
\def\rn{\mathbb{R}^{n}}
\def\zz{\mathbb{Z}}
\def\zn{\mathbb{Z}^{n}}

\def\nn{\mathbb{N}}

\def\cf{\mathcal {F}}

\def\cj{\mathcal {J}}

\def\cl{\mathcal {L}}

\def\cq{\mathcal {Q}}

\def\cs{\mathcal {S}}

\def\vp{{\varepsilon}}
\def\vf{{\varphi}}

\def\lp{{L^{p}(\rn)}}

\def\fz{\infty}

\def\supp{{\mathop\mathrm{\,supp\,}}}

\def\loc{{\mathop\mathrm{\,loc\,}}}

\def\sgn{{\mathop\mathrm{\,sgn\,}}}

\def\iif{{\mathop\mathrm{\,\Pi_1\,}}}
\def\iis{{\mathop\mathrm{\,\Pi_2\,}}}
\def\iit{{\mathop\mathrm{\,\Pi_3\,}}}

\def\ls{\lesssim}

\def\wz{\widetilde}
\def\r{\right}
\def\lf{\left}
\def\f{\frac}
\def\rar{\rightarrow}

\def\dsup{\displaystyle\sup}

\def\gfz{\genfrac{}{}{0pt}{}}

\newtheorem{theorem}{Theorem}[section]
\newtheorem{lemma}[theorem]{Lemma}
\newtheorem{corollary}[theorem]{Corollary}
\newtheorem{proposition}[theorem]{Proposition}
\theoremstyle{definition}
\newtheorem{remark}[theorem]{Remark}

\newtheorem{definition}[theorem]{Definition}
\renewcommand{\appendix}{\par
   \setcounter{section}{0}%
   \setcounter{subsection}{0}%
   \setcounter{subsubsection}{0}%
   \gdef\thesection{\@Alph\c@section}%
   \gdef\thesubsection{\@Alph\c@section.\@arabic\c@subsection}%
   \gdef\theHsection{\@Alph\c@section.}%
   \gdef\theHsubsection{\@Alph\c@section.\@arabic\c@subsection}%
   \csname appendixmore\endcsname
 }

\numberwithin{equation}{section}

\begin{document}

\arraycolsep=1pt

\title{\bf\Large Pointwise Multipliers for Besov Spaces $B^{0,b}_{p,\infty}(\mathbb{R}^n)$
with Only Logarithmic Smoothness
\footnotetext{\hspace{-0.35cm} 2020 {\it Mathematics Subject Classification}.
Primary 46E35; Secondary 42B35.
\endgraf {\it Key words and phrases.} pointwise multiplier, Besov space, paramultiplication, logarithmic smoothness, Dini continuous function.
\endgraf This project is partially supported by the National Key Research
and Development Program of China
(Grant No.\ 2020YFA0712900)
and the National
Natural Science Foundation of China (Grant Nos.\ 11971058, 12071197  and 12122102).}}
\author{Ziwei Li, Winfried Sickel\footnote{Corresponding author,
E-mail: \texttt{winfried.sickel@uni-jena.de}/{\red October 21,
2022}/Final version.},\ \ Dachun Yang
 and Wen Yuan}
\date{}
\maketitle

\vspace{-0.8cm}

\begin{center}
\begin{minipage}{13cm}
{\small{\textbf{Abstract}}\quad In this article, we establish a characterization of the
set $M(B^{0,b}_{p,\infty}(\mathbb{R}^n))$ of all pointwise multipliers of Besov spaces $B^{0,b}_{p,\infty}(\mathbb{R}^n)$ with
only logarithmic smoothness $b\in\mathbb{R}$ in the special cases
 $p=1$ and $p=\infty$.
As applications of these two characterizations,
we clarify whether or not the  three concrete examples, namely
characteristic functions of open sets,
continuous functions defined by differences, and the functions $e^{ik\cdot x}$
 with $k\in\mathbb{Z}^n$ and  $x\in\mathbb{R}^n$, are pointwise multipliers of $B^{0,b}_{1,\infty}(\mathbb{R}^n)$ and $B^{0,b}_{\infty,\infty}(\mathbb{R}^n)$, respectively; furthermore,  we obtain the explicit estimates of
$\|e^{ik \cdot x}\|_{M(B^{0,b}_{1,\infty}(\mathbb{R}^n))}$  and
$\|e^{ik \cdot x}\|_{M(B^{0,b}_{\infty,\infty}(\mathbb{R}^n))}$.
In the  case that $p\in(1,\infty)$, we give some  sufficient conditions and some necessary conditions of
the pointwise multipliers of $B^{0,b}_{p,\infty}(\mathbb{R}^n)$ and a
complete characterization of $M(B^{0,b}_{p,\infty}(\mathbb{R}^n))$ is still  open.
However, via a different method, we are still able to accurately  calculate
$\|e^{ik \cdot x}\|_{M(B^{0,b}_{p,\infty}(\mathbb{R}^n))}$, $k\in\mathbb{Z}^n$, in this situation.
The novelty of this article is that most of the proofs are constructive and these  constructions strongly depend on the logarithmic structure of Besov
spaces under consideration.}
\end{minipage}
\end{center}

\vspace{0.2cm}




\section{Introduction\label{s1}}


The understanding of pointwise multipliers for a given function space
belongs to the key problems in the theory of function spaces and has found a great variety of applications in many fields such as partial differential equations.
Fundamental work has been done  by Maz'ya and Shaposhnikova \cite{MS85,MS09}.
They described the set of all pointwise multipliers for
Sobolev, Besov, and Bessel potential spaces for a wide range of parameters.
The monograph \cite{MS09} is also a good source for applications.
Pointwise multipliers in the more general context of  Besov and Triebel--Lizorkin spaces
were studied in some details in  \cite{Ne92,Tr03,Tr06,RS96}.

In recent years various modifications of  Besov and Triebel--Lizorkin spaces
have been studied in the literature. The multiplier problem has been discussed, for instance, in
\cite{YSY05} (Besov-type and Triebel--Lizorkin-type spaces),  \cite{Y17}
(Haj{\l}asz-type spaces), \cite{NS17a,NS17b} (spaces with dominated mixed smoothness),
\cite{Am} (anisotropic  spaces), and \cite{BLYY21} (Zygmund classes).

Here, in this article, we are interested in a rather mild modification of Besov spaces, namely
Besov spaces with logarithmic smoothness.
These spaces  have been investigated in the
literature for more than 50 years.
A good source concerning the history is given by the introduction
from the article by Farkas and Leopold \cite{FL06}. It is connected with the names Ul'yanov, Dzhafarov, Gol'dman,
Kalyabin, Lizorkin, Netrusov, Merucci, Cobos, Fernandez, Triebel, Haroske, and Moura.
All these authors considered Besov spaces with
generalized smoothness, where the spaces  $B^{s,b}_{p,q}(\rn)$ with logarithmic smoothness are just a special case.
Quite different are the following references, where some fixed smoothness was studied.
Ciesielski, Kerkyacharian, and Roynette \cite{CKR93} had to consider
$B^{1/2,1/2}_{\infty,\infty}(\rr)$ to describe the regularity of the Brownian motion; see also Kamont  \cite{Ka96} for the regularity of the Brownian sheet, and
Kempka et al. \cite{KSV} for a further refinement.
Vybiral \cite{Vy05} showed that the trace mapping,
$$f(x,y) \mapsto f(x,x),\ \forall\,(x,y)\in \rr^2,$$
is a surjection of the space $S^{1/p}_{p,q}B(\rr^2)$ with dominating mixed smoothness
onto $B^{1/p,1/q -1}_{p,q}(\rr)$ if $p\in[1,\fz)$ and $q\in[1,\fz]$.
Of particular interests for us  are the more recent articles by Cobos, Dominguez, Haroske, Triebel, and Tikhonov \cite{CD14,CD15,CD15b,CD16,CDT16,D17,DT}.
They mainly concentrated on the spaces  $B^{0,b}_{p,q}(\rn)$.
Besides a good overview, they highlighted the differences between two important approaches in introducing function spaces, namely the Fourier analytic way as used in the present article and the way by using differences.

In this article, we mainly work on the characterization of the
space $M(B^{0,b}_{p,\infty}(\rn))$ of all pointwise multipliers of the Besov space
$B^{0,b}_{p,\infty}(\rn)$ with only logarithmic smoothness.
To our own surprise, there are only very few related results in the literature.
As indicated, we will deal here with the case $q=\infty$ and $s=0$ only.
The parameter $b \in \rr$, describing the logarithmic smoothness, is given for free
in $\rr$.

Of course, the case  $q<\infty$ would be interest as well. However, we had the feeling that
$q<\infty$ is more complicated. So we postponed the study of this case for future
investigations.
In contrast to this,  the case of both positive classical smoothness and an additional logarithmic smoothness seems to be less interesting, because many
results can be derived  parallel to the classical situation for $s>0$.

For the moment, exactly two characterizations of the pointwise multiplier spaces
$M(B^{0,b}_{p,\infty}(\rn))$
have been known. One is trivial and  given by  $M(B^0_{2,2}(\rn)) = M(L_{2}(\rn)) = L_\infty (\rn)$.
The other one is non-trivial and concerns the characterization of
$M(B^0_{\infty,\infty}(\rn))$; see \cite{KS02}.
We will not only extend this characterization to  logarithmic smoothness,
additionally we will give a characterization of $M(B^{0,b}_{1,\infty}(\rn))$.
For the cases inbetween, which means the spaces $B^{0,b}_{p,\infty}(\rn)$,
with $p\in(1,\fz)$ and $b\in\rr$,
we can only  derive some sufficient and some necessary conditions.
Regrettably, the characterization of $M(B^{0,b}_{p,\infty}(\rn))$, when $p\in(1,\infty)$ and  $b \in \rr$, remains an open problem.
This is, in some sense, parallel to the classical  case.
If  $s >0$ and $q=p$,  the  characterizations of $M(B^{s}_{1,1}(\rn))$ and
of $M(B^{s}_{\infty,\infty}(\rn))$
are easier to get than those of $M(B^{s}_{p,p}(\rn))$ with $p\in(1,\infty)$.
To our delight, our results will be good enough for calculating
$\| e^{ik\cdot }\|_{M(B^{0,b}_{p,\infty}(\rn))}$, $k\in\zz^n$, for any given
$b\in \rr$ and $p\in[1,\infty]$.
Just as expected, these norms are of  logarithmic growth for $k$
tending to infinity.
To make our results more transparent for the reader, we also consider both
continuous functions of minimal smoothness and characteristic functions of open sets
as pointwise multipliers.

As a common phenomenon, the road from the case $s=0,\ b<0$ to $s<0$, $b=0$
or from the case $s=0,\ b>0$ to  $s>0$, $b=0$ is complicated.
One typical example to see this is the characteristic function of the upper half-space.
Here is a summary for this case:
\begin{itemize}
 \item $\mathbf{1}_{\rn_+} \in M(B^{s}_{\fz,\fz}(\rn))$ for any $s \in (-1,0)$;
 \item $\mathbf{1}_{\rn_+} \notin M(B^{0,b}_{\fz,\fz}(\rn))$ for any $b\in \rr$;
 \item $\mathbf{1}_{\rn_+} \in M(B^{0}_{p,\fz}(\rn))$ for any $p\in(1,\fz)$;
 \item $\mathbf{1}_{\rn_+} \in M(B^{0,b}_{p,\fz}(\rn))$ for any $ b\in\rr$
 and $p\in(1,\fz)$;
  \item $\mathbf{1}_{\rn_+} \in M(B^{s,b}_{p,\fz}(\rn))$ for any $ b\in\rr$,
$p\in(1,\fz)$, and $s\in(1/p-1,1/p)$;
 \item $\mathbf{1}_{\rn_+} \notin M(B^{0,b}_{1,\fz}(\rn))$ for any $ b\in\rr$;
\item $\mathbf{1}_{\rn_+} \in M(B^{s,b}_{1,\fz}(\rn))$ for any $ b\in\rr$ and $s\in (0,1)$.
\end{itemize}

The article will be organized as follows.
In Section \ref{section2} we collect the definitions and notation we will use.
A few more essentially known properties of the spaces
$B^{0,b}_{p,\infty}(\rn)$ and $M(B^{0,b}_{p,\infty}(\rn))$
will be recalled there as well.
Section \ref{section3} contains the main characterizations and some specific applications of these characterizations.
Finally, in Section \ref{section4}, we concentrate on the proofs of the
aforementioned  results.
Within this section, Subsections \ref{sec-suff-p} and \ref{sec-nece-p}
are vital, where we derive the sufficient and the necessary conditions, respectively, for $M(B^{0,b}_{p,\fz}(\rn))$, with $p\in[1,\fz]$
and $b\in \rr$; Subsections \ref{section3} and \ref{section4} contain the proofs of the main characterizations and applications .
It is worth mentioning that most of the proofs are constructive
and strongly depend on the construction of the auxiliary functions,
which in turn strongly depends on the logarithmic structure
of Besov spaces under consideration
[see, for instance, \eqref{eq-construct-g}].
We also use some technical tools such as
the Peetre--Fefferman--Stein type maximal function,
the Littlewood--Paley characterization of $L^p(\rn)$, and the duality.


\section{Preliminaries}
\label{section2}

This section is divided into three parts. We first recall some notation conventions in Subsection \ref{sec-notation}
and then, in Subsection \ref{sec-HoTri},
recall some basic knowledge about Besov spaces with logarithmic smoothness.
Subsection \ref{sec-MB} is devoted to studying some key properties on the
multiplier space $M(B^{s,b}_{p,q}(\rn))$.


\subsection{Notation and Concepts}\label{sec-notation}

As usual, $\rr$ denotes the real numbers, $\nn$ the natural numbers (the positive integers),  $\zz$ the integers, $\zz_+:=\nn\cup\{0\}$,
$\rn$ the $n$-dimensional  Euclidean space, $\zz^n$ the space of vectors
in $\rn$ with integer components, and $\zz_+^n$ with nonnegative integer components.
We use $\mathbf{0}$ to denote the zero element of $\rn$.

Let $\cs(\rn)$ be the collection of all Schwartz functions on $\rn$,
whose topology is determined by a family of norms,
$\{\|\cdot\|_{\cs_{k,m}(\rn)}\}_{k,m\in\zz_{+}}$,
where, for any $k,m\in\zz_{+}$ and $\vf\in\cs(\rn)$,
$$
\|\vf\|_{\cs_{k,m}(\rn)}:=\sup_{\alpha\in\zz_{+}^{n},|\alpha|\leq k}\sup_{x\in\rn}
\lf(1+|x|\r)^{m}\lf|\partial^{\alpha}\vf(x)\r|
$$
with the multi-index
$\alpha:=(\alpha_{1},\ldots,\alpha_{n})\in\zz_{+}^{n},
\ |\alpha|:=\alpha_{1}+\cdots+\alpha_{n}$, and
$\partial^{\alpha}:=(\f{\partial}{\partial x_{1}})^{\alpha_{1}}
\cdots(\f{\partial}{\partial x_{n}})^{\alpha_{n}}$.
Also let $\cs'(\rn)$ be the space of all tempered distributions on $\rn$
equipped with the weak-$*$ topology.
We use the symbol $\cf$ to denote the Fourier transform  and by  $\cf^{-1}$  its inverse; in
particular, for any $f\in L^1(\rn)$ and $\xi\in\rn$,
$$\cf f(\xi):=(2\pi)^{-\f n 2}\int_{\rn}f(x)e^{-ix\cdot \xi}\,dx.$$

For any $j\in\zz$ and $\nu\in\zn$, we
define the dyadic cube by  $Q_{j,\nu}:=2^{-j}(\nu+[0,1)^{n})$.
Let $x_{Q}$ denote the \emph{lower left-corner} $2^{-j}\nu$ of $Q:=Q_{j,\nu}$
and $l(Q)$ its edge length $2^{-j}$.
For any $\lambda\in(0,\fz)$, we denote by $\lambda Q$ the cube with the same center as $Q$ and edges
parallel to the edges of $Q$, but
with $\lambda$-times edge length.
By $\cq$ we denote the collection of all dyadic cubes in $\rn$ and
by $\cq_{j}$ the collection of all dyadic cubes with edge length $2^{-j}$.

For any  $p\in(0,\infty]$, we denote by $\|\cdot\|_{L^p(\rn)}$ the norm of $L^p(\rn)$.
For any given $p,q\in(0,\infty]$ and $s,b\in\rr$, we
define $\ell^{q}_{s,b}(L^{p})$ as the collection of all the sequences $\{u_k\}_{k\in\zz_+}$  of functions on $\rn$ such that
\begin{equation*}\label{eq-lq(Lp)}
\|\{u_k\}_{k\in\zz_+}\|_{\ell^{q}_{s,b}(L^{p})}
:=\lf\{\sum_{k\in\zz_+}\lf[2^{ks}(1+k)^b\lf\|u_k\r\|_{L^p(\rn)}\r]^{q}\r\}^{\f 1 q}<\fz.
\end{equation*}
When $s=0=b$, we will simply write  $\ell^{q}(L^{p})$.

We denote by $L^1_{\loc}(\rn)$ the collection of all locally integrable functions on $\rn$.
For any $u\in L^1_{\loc}(\rn)$ and any measurable set $E\subset \rn$, let
$$
\fint_{E}u(x)\,dx:=\f{1}{|E|}\int_{E}u(x)\,dx.
$$
The characteristic function over $E$ is denoted by $\mathbf{1}_E$.

Frequently we will  use the following inequality : for any given $r\in(0,1]$,
\begin{equation}\label{eq-triangle}
\lf(\sum_{k\in\zz}|a_k|\r)^r\leq\sum_{k\in\zz}|a_k|^r,
\quad\forall\,\{a_k\}_{k\in\zz}\subset\rr.
\end{equation}

Throughout this article, for any $q\in[1,\fz]$, we denote by $q'$ its conjugate index,
that is, $1/q+1/q'=1$.
The {symbols}  $C, C_1, \ldots $ denote    positive constants
which depend only on the fixed parameters $n,s,b,p$, and,
probably, also on some auxiliary functions
unless others are stated. Their  values  may vary from line to line.
Sometimes we use the symbol ``$ \ls $'' (``$\gtrsim $'')
instead of ``$ \le $'' (``$ \ge $''). The meaning of $A \ls B$ (resp., $A\gtrsim B$) is
given by: there exists a positive constant $C$ such that
 $A \le CB$ (resp., $A \ge CB$).
 Finally,  the symbol $A \sim B$ will be used as an abbreviation of
$A \ls B \ls A$.


\subsection{Besov Spaces with Logarithmic Smoothness}\label{sec-HoTri}


Let $\phi_0\in\cs(\rn)$ be a radial and real-valued function such that
\begin{equation}\label{eq-phi0}
0\leq\phi_0\leq1,\ \ \phi_0\equiv1\ \ \text{on}\ \ \{x\in\rn:\ |x|\leq1\},\ \ \text{and}\ \
\phi_0\equiv0\ \  \text{on}\ \ \{x\in\rn:\ |x|\geq3/2\}.
\end{equation}
Define $\{\phi_k\}_{k\in\nn}$ by setting, for any $x\in\rn$,
\begin{equation}\label{eq-phi1}
\phi_1(x):=\phi_0(x/2)-\phi_0(x)
\end{equation}
and, for any $k\geq2$,
\begin{equation}\label{eq-phik}
\phi_k(x):=\phi_1(2^{-k+1}x).
\end{equation}
Clearly, $\sum_{k=0}^\fz \phi_k=1$ on $\rn$.
Furthermore, for any $k\in\zz_+$ and $f\in\cs'(\rn)$, we define
\begin{equation}\label{eq-S_k}
S_kf:=\cf^{-1}(\phi_k\cf f)=\varphi_k\ast f
\end{equation}
and
\begin{equation}\label{eq-S^k}
S^kf:=\sum_{j=0}^{k}S_{j}f,
\end{equation}
where $\varphi_k:= \cf^{-1}{\phi_k}$ is the inverse Fourier transform of $\phi_k$.
For convenience, throughout this article,
for any integer $j<0$ and $f\in\cs'(\rn)$, we let $S_{j}f:=0$.
It is easily observed that, for any $f\in\cs'(\rn)$,
\begin{equation*}\label{eq-supp-1}
\supp \cf(S_0f)\subset \{x\in\rn:\ |x|\leq 3/2\},
\end{equation*}
for any $k\geq1$,
\begin{equation}\label{eq-supp-2}
\supp \cf(S_kf)\subset \lf\{x\in\rn:\ 2^{k-1}\leq|x|\leq 3\cdot2^{k-1}\r\},
\end{equation}
\begin{equation*}\label{eq-supp-3}
\supp \cf(S^kf)\subset \lf\{x\in\rn:\ |x|\leq 3\cdot2^{k-1}\r\},
\end{equation*}
and, for any $k\geq2$,
\begin{equation}\label{eq-supp-4}
\supp \cf\lf(\lf[S^{k-2}f\r]S_kg\r)\subset \lf\{x\in\rn:\ 2^{k-3}\leq|x|\leq 2^{k+1}\r\},
\end{equation}
\begin{equation}\label{eq-supp-5}
\supp \cf\lf(\lf[S_{k}f\r]S^{k-2}g\r)\subset \lf\{x\in\rn:\ 2^{k-3}\leq|x|\leq 2^{k+1}\r\},
\end{equation}
and
\begin{equation}\label{eq-supp-6}
\supp \cf\lf(\sum_{j=k-1}^{k+1}\lf[S_{j}f\r]S_{k}g\r)\subset \lf\{x\in\rn:\ |x|\leq 5\cdot2^{k}\r\}.
\end{equation}

\begin{definition}
Let $s,b\in\rr$,  and $p,q\in(0,\fz]$. The
\emph{Besov space with logarithmic smoothness} $B^{s,b}_{p,q}(\rn)$
is defined as the collection of all the $f\in\cs'(\rn)$ such that
\begin{equation*}\label{eq-def-B^sb_pq}
\|f\|_{B^{s,b}_{p,q}(\rn)}:=\lf\{\sum_{k\in \zz_+}\lf[2^{ks}(1+k)^b
\lf\|S_kf\r\|_{\lp}\r]^{q}\r\}^{\f 1 q}<\fz
\end{equation*}
with the usual modification made if $q=\fz$.
In case $b=0$ we simply write $B^{s}_{p,q}(\rn)$ instead of $B^{s,0}_{p,q}(\rn)$.
\end{definition}

A few times we will need Triebel--Lizorkin spaces with $p=\infty$.

\begin{definition}
Let $s,b\in\rr$ and $q\in(0,\infty]$. The
\emph{Triebel--Lizorkin space with logarithmic smoothness} $F^{s,b}_{\infty,q}(\rn)$
is defined as the collection of all the $f\in\cs'(\rn)$ such that
\begin{equation*}\label{eq-def-F^sb_pq}
\|f\|_{F^{s,b}_{\infty,q}(\rn)}:=\lf\{\sup_{k \in \zz_+} \sup_{\nu \in \zz^n}
2^{kn} \int_{Q_{k,\nu}}
\sum_{j=k}^\infty 2^{jsq}(1+j)^{bq}
|S_jf (x)|^q \,dx\r\}^{\f 1 q},
\end{equation*}
with the usual modification made if $q=\fz$, is finite.
\end{definition}

The following lemma, in the case  $b=0$,  originates from
\cite{Ya86} (see also \cite[Proposition 2.3.2/1(ii)]{RS96}).
The extension to the case of logarithmic  smoothness is almost obvious.
So we omit further details.

\begin{lemma}\label{lem-I1-1}
Let $p,q\in(0,\fz]$, $s,b\in\rr$, and
$\{u_k\}_{k\in\zz_+}$ be a sequence of functions such that
\begin{enumerate}
\item[\rm(i)] for any $k\in\zz_+$, $u_k\in[\cs'(\rn)\cap L^p(\rn)]$;
\item[\rm(ii)] $\supp \, \cf {u_0}\subset\{\xi\in \rn:\ |\xi|\leq2\}$;
\item[\rm(iii)] for any $k\in\nn$, $\supp \, \cf {u_k}\subset\{\xi\in \rn:\ 2^{k-1}\leq|\xi|\leq 2^{k+\sigma}\}$,
\end{enumerate}
where $\sigma$ is some fixed positive integer.
If $\{2^{js}(1+j)^b u_j\}_{j= 0}^\infty\in \ell^q(L^p)$, then $\sum_{k = 0}^\infty u_k\in B^{s,b}_{p,q}(\rn)$ and, furthermore,
there exists a positive constant $C=C_{(n,p,q,s,b,\sigma)}$ such that
$$
\lf\|\sum_{k=0}^\infty  u_k\r\|_{B^{s,b}_{p,q}(\rn)}
\leq C\lf\|\lf\{2^{js}(1+j)^bu_j\r\}_{j= 0}^\infty \r\|_{\ell^q(L^p)}.
$$
\end{lemma}

Sometimes we  will use the following simple facts.

\begin{lemma}\label{lem-sum-log}
\begin{enumerate}
\item[\rm(i)] Let $b\in(1,\fz)$ and $k\in\zz_+$. Then there exists a positive constant $C$ such that
\begin{equation*}
\f{1}{b-1}(k+1)^{1-b}\le\sum_{j = k}^\infty \f{1}{(1+j)^b}
\leq C(k+1)^{1-b}.
\end{equation*}

\item[\rm(ii)] Let $b\in(-1,\fz)$ and $k\in\zz_+$. Then there exists a positive constant $C$ such that
\begin{equation*}\label{eq-sumlog2}
\f{1}{b+1}(k+1)^{b+1}\le\sum_{j=0}^k (1+j)^b\leq
C(k+1)^{b+1}.
\end{equation*}
\end{enumerate}
\end{lemma}

Let $a\in(0,\fz)$, $j\in\zz_+$, and $f\in\cs'(\rn)$.
Recall that the \emph{maximal function of Peetre--Fefferman--Stein type} is defined by setting
\begin{equation}\label{eq-Peetre}
S_j^{*,a}f(x):=\sup_{y\in\rn}\f{|S_jf(x-y)|}{(1+2^j|y|)^a}, \qquad \forall\,x\in\rn\, ,
\end{equation}
where $S_j$ is defined in \eqref{eq-S_k} (see, for instance, \cite[Definition 2.3.6/2]{Tr83}).
We remark that sometimes  $S_j^{*,a}f$ is also defined in \eqref{eq-Peetre}
with $(1+2^j|y|)^a$ replaced by $1+|2^jy|^a$, but  these two definitions are obviously pointwise equivalent.

Clearly, for any $j\in\zz_+$ and $x\in\rn$, $|S_jf(x)|\le S_j^{*,a}f(x)$.
The following lemma (see \cite[2.3.6/(22)]{Tr83}) means that,
for any $p\in(0,\fz]$ and $j\in\zz_+$, $S_j^{*,a}f$ is bounded by $S_jf$ in $L^p(\rn)$.

\begin{lemma}\label{lem-peeter-Lp}
Let $p\in(0,\fz]$ and $a\in(n/p,\fz)$. Then there exists a positive constant $C$ such that,
for any $f\in\cs'(\rn)$ and $j\in\zz_+$,
\begin{align*}
\lf\|S_j^{*,a}f\r\|_{L^p(\rn)}\le C\lf\|S_jf\r\|_{L^p(\rn)}.
\end{align*}
\end{lemma}


\subsection{The Multiplier Space $M(B^{s,b}_{p,q}(\rn))$}\label{sec-MB}


We shall use the following definition of the product.
For any given $f,g\in\cs'(\rn)$, we define
\begin{equation}\label{eq-def-product}
fg:=\lim_{j\rar\fz}\left(S^jf\right)S^jg
\end{equation}
if the limit on the right-hand side
exists in $\cs'(\rn)$.
Thus, at least formally,  $fg$ has a decomposition
\begin{align}\label{eq-decompose}
fg&=\sum_{k = 2}^\infty\left(S^{k-2}f\right)S_kg+
\sum_{k = 0}^\infty \sum_{i=-1}^{1}\left(S_{k+i}f\right)S_kg+\sum_{k=2}^\infty \left(S_kf\right)S^{k-2}g\\
&=:\iif(f,g)+\iis(f,g)+\iit(f,g)\nonumber.
\end{align}
The bilinear mappings
 $\iif(f,g),\ \iis(f,g),$ and $\iit(f,g)$ are called \emph{paraproducts}.

 \begin{definition}
Let $s,b\in\rr$,   and $p,q\in(0,\fz]$.
The pointwise multiplier space $M(B^{s,b}_{p,q}(\rn))$ of $B^{s,b}_{p,q}(\rn)$ is defined by setting
\begin{equation*}\label{eq-Mspace}
M(B^{s,b}_{p,q}(\rn)):=
\lf\{f\in\cs'(\rn):\ fg\in B^{s,b}_{p,q}(\rn) ~ \mbox{for any}~\,g\in B^{s,b}_{p,q}(\rn)\r\},
\end{equation*}
equipped with the quasi-norm
\begin{align}\label{eq-Mnorm}
\lf\|f\r\|_{M(B^{s,b}_{p,q}(\rn))}:&=\sup_{g\in B^{s,b}_{p,q}(\rn),\,g\neq \mathbf{0}}
\f{\lf\|fg\r\|_{B^{s,b}_{p,q}(\rn)}}{\lf\|g\r\|_{B^{s,b}_{p,q}(\rn)}}\\
&=\sup_{g\in B^{s,b}_{p,q}(\rn),\,g\neq\mathbf{0}}
\f{\lf\|T_f g\r\|_{B^{s,b}_{p,q}(\rn)}}{\lf\|g\r\|_{B^{s,b}_{p,q}(\rn)}}\nonumber,
\end{align}
where $\mathbf{0}$ denotes the zero element of $B^{s,b}_{p,q}(\rn)$.
\end{definition}

Here and thereafter, we denote the operator $g\mapsto fg$ by $T_f$
and, for any given quasi-Banach space $X$, we use $\cl(X)$ to denote the set of all bounded linear operators on $X$.
The following lemma contains  some basic properties of $M(B^{s,b}_{p,q}(\rn))$.

\begin{lemma}\label{lem-M(B)}
Let $p,q\in(0,\fz]$ and $s,b\in\rr$.
\begin{enumerate}
\item[\rm(i)]
If $f\in M(B^{s,b}_{p,q}(\rn))$, then $T_f\in\cl(B^{s,b}_{p,q}(\rn))$.

\item[\rm(ii)]
There exists a positive constant $C$ such that,
for any $f\in M(B^{s,b}_{p,q}(\rn))$,
$$
\|f\|_{L^\fz(\rn)}\le C\lf\|f\r\|_{M(B^{s,b}_{p,q}(\rn))}.
$$

\item[\rm(iii)] Let $p,q\in[1,\fz]$.
If $f\in M(B^{s,b}_{p,q}(\rn))$ and $h\in L^1(\rn)$,
then $h\ast f\in M(B^{s,b}_{p,q}(\rn))$ and, moreover,
$$
\lf\|h\ast f\r\|_{M(B^{s,b}_{p,q}(\rn))}\le\|h\|_{L^1(\rn)}\lf\|f\r\|_{M(B^{s,b}_{p,q}(\rn))}.
$$

\item[\rm(iv)] Let $p,q\in[1,\fz]$. Then
$M(B^{s,b}_{p,q}(\rn))$ is a Banach space.
\end{enumerate}
\end{lemma}

For the proofs of parts (i)-(iii) of Lemma \ref{lem-M(B)} in the case $b=0$,
we refer to \cite[Lemma 9]{KS02} and
\cite[Theorem 4.3.2 and Lemma 4.6.3/1]{RS96}.
The general case follows from a similar way.
Thus, we need to prove (iv) only.
To  this end, we shall need a further preparation.

\begin{lemma}\label{mult}
Let $p_0 ,p_1,q_0,q_1 \in (0,\infty]$ and $s_0, b_0, s_1,b_1 \in \rr$.
Let $f \in B^{s_0,b_0}_{p_0,q_0} (\rn)$ be locally integrable, $g \in B^{s_1,b_1}_{p_1,q_1} (\rn)$,
and, for some $N\ge 6$,
$$
 \lim_{j \to \infty}\left(S^{j+N}f\right)S^jg
$$
exists in $\cs '(\rn)$.
Then $ \lim_{j \to \infty}  f S^jg$ exists in $\cs'(\rn)$ and
\[
  \lim_{j \to \infty}   f  S^jg =  \lim_{j \to \infty}  (S^{j+N} f)S^jg.
\]
\end{lemma}

\begin{proof}
By both the assumption and  \cite[Theorem 1.4.1]{Tr83}, we know that, for any given $p_2\in[\max\{1,p_1\},\infty]$ and any
$j\in\nn$, $S^jg \in L^{p_2} (\rn)$, which, together with the fact $f\in L_\loc^1(\rn)$, further implies that $f   S^jg\in L_\loc^1(\rn)$
as well and, for any $\varrho \in \cs (\rn)$,
$$
 \lf(f  S^jg\r) (\overline{\varrho})  = \int_{\rn}  f(x)  S^jg (x)
 \overline{\varrho (x)}\, dx .
$$
Let $\wz{\phi}\in\cs(\rn)$ be such that
$\wz{\phi} = 1$ on $\{x\in\rn:\ 1/4 \le |x|\le 4\}$ and
$$\supp \wz{\phi} \subset \lf\{\xi \in \rn:\ 1/8 \le |\xi| \le 8\r\}.$$
For any given $u \in \cs'(\rn)$ and $m\in\nn$, we let
$\wz{S}_m u := \cf^{-1} [\wz{\phi}(2^{-m+1}\, \cdot \, )\cf u]$.
Then we have, at least formally, for any $N\in\nn$,
\begin{eqnarray}\label{ws-30}
 \lf(\lf[f-S^{j+N}f\r]  S^jg\r) (\overline{\varrho}) & = &
 \sum_{m=j+N+1}^\infty  \int_{\rn}  S_m f(x)
 \overline{\overline{S^jg (x)}\,   \varrho (x)}\, dx
 \\
&=&  \sum_{m=j+N+1}^\infty
\int_{\rn}  S_m f(x)  \overline{\wz{S}_m
\left(\varrho  \overline{S^jg}\right)(x)}\,  dx.\nonumber
\end{eqnarray}
Now, we take $N\ge 6$.
Let $k,j\in\nn$ and $m \ge j+N +1$. Since
$$
 \supp \cf \lf(\left[S_k\varrho\right] \overline{S^jg}\r)\subset\lf\{\xi\in\rn:
 \ |\xi| \le 3\cdot 2^{k-1} + 3\cdot 2^{j-1}\r\}
$$
and, for any $k\le m-7$,
$$
 \lf\{\xi\in\rn:
 \  |\xi| \le 3\lf(2^{k-1} + 2^{j-1}\r)\r\} \cap
 \lf\{\xi\in\rn:
 \ 2^{m-4} < |\xi| <  2^{m+2} \r\} = \emptyset,
$$
then it follows that
$$
 \wz{S}_m
\lf(\varrho \overline{S^jg}\r)(x) = \sum_{k=m-6}^\infty \wz{S}_m
\lf(\lf[S_k\varrho\r]\overline{S^jg}\r)(x).
$$
Inserting this into \eqref{ws-30} and using the H\"older  inequality, we find that
\begin{align}\label{ws-31}
&\lf|\lf(\lf[f-S^{j+N}f\r]  S^jg\r) (\overline{\varrho})\r|\\
&\quad \le
\sum_{m=j+N+1}^\infty \sum_{k=m-6}^\infty \lf\| S_m f \r\|_{L^\infty (\rn)}
\int_{\rn} \lf|\wz{S}_m
\lf(\lf[S_k\varrho\r]\overline{S^jg}\r)(x)\r|\, dx
\nonumber\\
&\quad\ls
\sum_{m=j+N+1}^\infty \sum_{k=m-6}^\infty \lf\| S_m f \r\|_{L^\infty (\rn)}
\int_{\rn} \lf|S_k\varrho (x)\r| \lf|S^jg(x)\r|\, dx
\nonumber\\
&\quad\ls
\sum_{m=j+N+1}^\infty \sum_{k=m-6}^\infty \lf\|  S_m f \r\|_{L^\infty (\rn)}
\lf\| S^jg \r\|_{L^{p_2}(\rn)}  \lf\|  S_k\varrho \r\|_{L^{p_2'}(\rn)}\nonumber
\end{align}
for any $p_2 \in [1,\infty]$. We choose  $p_2 \ge  p_1$. From the famous Nikol'skij inequality (see, for instance,  \cite[(1.3.2/5)]{Tr83}), we infer that
\begin{align*}
\lf\| S^jg\r\|_{L^{p_2}(\rn)} & \le \sum_{l=0}^j\lf\|  S_l g \r\|_{L^{p_2}(\rn)}\\
&\ls\sum_{l=0}^j 2^{ln(\f{1}{p_1}-\f{1}{p_2})-ls_1}(1+l)^{-b_1}\lf[2^{ls_1}(1+l)^{b_1}\lf\|  S_l g\r\|_{L^{p_1}(\rn)}\r]
\\
&\ls  \| g\|_{B^{s_1,b_1}_{p_1,\infty}(\rn)}
\begin{cases}
\ 1 &\  \mbox{if}\  \gamma >0, \ b_1 \in \rr,
\\
\ 1 &\  \mbox{if}\  \gamma =0, \ b_1 >1,
\\
\ \ln (1+j) &\  \mbox{if}\  \gamma = 0, \ b_1=1,
\\
\ (1+j)^{1-b_1} &\ \mbox{if}\  \gamma = 0, \ b_1< 1,
\\
\ 2^{-j(s_1 -  \f{n}{p_1} +\f{n}{p_2})}(1+j)^{-b_1}
&\ \mbox{if}\  \gamma <0,\ b_1 \in \rr,
\end{cases}
\end{align*}
where $\gamma=\gamma(s_1,p_1,p_2):=s_1 - n(1/p_1 - 1/p_2)$;
for simplicity, we shall denote by $C_{(j)}$ the number on the right-hand side in the last step.
Applying this, \eqref{ws-31}, and the Nikol'skij inequality again,
we conclude that
\begin{align}\label{newlab}
&\lf|\lf(\lf[f-S^{j+N}f\r]  S^jg\r) (\overline{\varrho})\r|
\\
&\quad\ls  C_{(j)} \|  g \|_{B^{s_1,b_1}_{p_1,\infty}(\rn)}
\sum_{m=j+N+1}^\infty \|  S_m f \|_{L^\infty (\rn)}
\sum_{k=m-6}^\infty
  \|  S_k\varrho \|_{L^{p_2'}(\rn)}\nonumber
\\
&\quad\ls
C_{(j)} \|  g \|_{B^{s_1,b_1}_{p_1,q_1}(\rn)}  \|  \varrho  \|_{B^t_{p_2',1}(\rn)}
\sum_{m=j+N+1}^\infty 2^{\f{mn}{p_0}-mt}\lf\|S_m f \r\|_{L^{p_0} (\rn)} \nonumber
\\
&\quad\ls
C_{(j)} \|  g \|_{B^{s_1,b_1}_{p_1,q_1}(\rn)}  \| \varrho\|_{B^t_{p_2',1}(\rn)}
\| f \|_{B^{s_0,b_0}_{p_0,\infty}(\rn)}\nonumber\\
&\qquad\times\sum_{m=j+N+1}^\infty  2^{m(\f{n}{p_0}-t-s_0)}(1+m)^{-b_0} \nonumber
\\
&\quad\ls
C_{(j)} 2^{j(\f{n}{p_0}-t-s_0)}(1+j)^{-b_0}
\| g\|_{B^{s_1,b_1}_{p_1,q_1}(\rn)} \|  \varrho \|_{B^t_{p_2',1}(\rn)}
\| f \|_{B^{s_0,b_0}_{p_0,\infty}(\rn)},\nonumber
\end{align}
as long as $t$ is chosen sufficiently large.
Note that $\cs (\rn)$ is included in any Besov space $B^t_{p,q}(\rn)$, which means that
we can choose a $t$ as large as we want.
Clearly, for an appropriate $t$ we have
$$
\lim_{j \rar \infty} C_{(j)} 2^{- j (t+s_0 -  \f{n}{p_0})} (1+j)^{-b_0} =0,
$$
which proves that, for any $\varrho \in \cs (\rn)$,
$$
\lim_{j \rar \infty}\lf|\lf(\lf[f-S^{j+N}f\r]  S^jg\r) (\overline{\varrho})\r|=0.
$$
This finishes the proof of Lemma \ref{mult}.
\end{proof}

\begin{lemma}\label{multmod}
Let $p_0 ,p_1,q_0,q_1 \in (0,\infty]$ and $s_0, b_0, s_1,b_1 \in \rr$.
Suppose  $f \in B^{s_0,b_0}_{p_0,q_0} (\rn)$  and  $g \in B^{s_1,b_1}_{p_1,q_1} (\rn)$.
Then, for any $\ell\in \nn$,
\begin{equation*}
  \lim_{j \to \infty}\lf(S_{j+\ell} f\r)S^jg= 0.
\end{equation*}
\end{lemma}

\begin{proof}
 Let  ${\varrho} \in \cs (\rn)$.
Applying the Nikol'skij inequality with $q=\fz$
(see, for instance, \cite[(1.3.2/5)]{Tr83}) and \eqref{eq-supp-2},
 we find that, for any $j\in\nn$,
 \begin{align*}
\lf\| S^jg\r\|_{L^\fz(\rn)}
& \le \sum_{m=0}^j \lf\| S_m g\r\|_{L^\fz(\rn)}
\ls\sum_{m=0}^j 2^{\f{mn}{p_1}}\lf\| S_m g\r\|_{L^{p_1}(\rn)}
\\
&\ls \sum_{m=0}^j 2^{m(\f{n}{p_1} -s_1)} (1+m)^{-b_1} \|  g\|_{B^{s_1,b_1}_{p_1,\infty}(\rn)}
 \end{align*}
and, for any given $\alpha\in\zz^n_+$,
 \begin{align}\label{eq-Niko-alpha}
 \lf\| D^\alpha S^jg\r\|_{L^\fz(\rn)}&\ls
   \sum_{m=0}^j 2^{m(|\alpha|+\f{n}{p_1} -s_1)} (1+m)^{-b_1} \| g\|_{B^{s_1,b_1}_{p_1,\infty}(\rn)}\\
   &\ls 2^{j(|\alpha|+\f{n}{p_1} -s_1)} (1+j)^{-b_1}\| g\|_{B^{s_1,b_1}_{p_1,\infty}(\rn)}, \nonumber
\end{align}
where we took $|\alpha|$ big enough and the hidden positive constants depend on $|\alpha|$.
Similarly, for any given $l\in\nn$ and any $j\in\nn$, we also have
 \begin{align}\label{eq-Niko-l}
 \| S_{j+\ell} f\|_{L^\fz(\rn)}  &\ls
2^{\f{(j+\ell)n}{p_0}} \| S_{j+\ell} f\|_{L^{p_0}(\rn)}\\
&\ls  2^{(j+\ell)(\f{n}{p_0} -s_0)} (1+j+\ell)^{-b_0} \|  f\|_{B^{s_0,b_0}_{p_0,\infty}(\rn)}\nonumber .
\end{align}
Let $\wz{\phi}\in\cs(\rn)$ and $\wz{S}_m,\ m\in\nn,$ be
defined the same as in the proof of Lemma \ref{mult}.
By checking the supports of the Fourier transform of the integrands, we find that, for any $j\in\nn$,
\begin{align*}
\lf(\lf[S_{j+\ell} f\r] S^jg \r)(  \overline{\varrho} )
&=\int_{\rn} S_{j+\ell} f (x)\overline{\varrho (x) \overline{S^jg(x)} }\, dx
\\
&= \int_{\rn} S_{j+\ell} f (x) \overline{\wz{S}_{j+\ell}\lf(\varrho
\overline{S^jg}\r)(x)}\,dx,
\end{align*}
which, combined with \eqref{eq-Niko-l}, implies that
\begin{align}\label{eq-star6}
&\lf|\lf(\lf[S_{j+\ell} f\r]S^jg \r)(\overline{\varrho})\r|\\
& \quad\le \| S_{j+\ell} f\|_{L^\fz(\rn)}
\int_{\rn} \lf| \wz{S}_{j+\ell} \lf(\varrho \overline{S^jg}\r)(x)\r|\, dx \nonumber\\
& \quad\ls  2^{(j+\ell)(\f{n}{p_0} -s_0-t_1)} (1+j+\ell)^{-b_0} \|  f\|_{B^{s_0,b_0}_{p_0,\infty}(\rn)}
\lf\|\varrho \overline{S^jg}\r\|_{B^{t_1}_{1,\infty}(\rn)},\nonumber
\end{align}
where $t_1$ is at our disposal.
Let $\mathfrak{C}^s(\rn)$ denote the \emph{Zygmund space} defined the same as in \cite[(2.2.2/6)]{Tr83}.
From both \cite[Theorem 2.8.2]{Tr83} (see also \cite[(4.7.1/10)]{RS96})
and \eqref{eq-Niko-alpha},
we deduce that, for any $j\in\nn$,
\begin{align*}
\lf\|\varrho \, \overline{S^jg}\r\|_{B^{t_1}_{1,\infty}(\rn)}
&\ls
\lf\| \overline{S^jg}\r\|_{\mathfrak{C}^{t_2}(\rn)}
\|\varrho \|_{B^{t_1}_{1,\infty}(\rn)}\\
&\ls2^{j(t_2+\f{n}{p_1} -s_1)} (1+j)^{-b_1}\| g\|_{B^{s_1,b_1}_{p_1,\infty}(\rn)}
\lf\|\varrho \r\|_{B^{t_1}_{1,\infty}(\rn)}\nonumber,
\end{align*}
where $t_2$ is some positive integer.
By inserting this into \eqref{eq-star6}, we conclude that, if $t_1$ is chosen large enough, then
there exists some $\vp\in(0,\fz)$ such that
$$
 \lf|\lf(\lf[S_{j+\ell} f\r]S^jg \r)\lf(  \overline{\varrho} \r)\r|\ls
 2^{-(j+\ell)\varepsilon}
 \lf\|\varrho \r\|_{B^{t_1}_{1,\infty}(\rn)} \|  f\|_{B^{s_0,b_0}_{p_0,\infty}(\rn) }\| g\|_{B^{s_1,b_1}_{p_1,\infty}(\rn)}.
$$
Notice that the implicit positive constant is independent of $j$.
This finishes the proof of Lemma \ref{multmod} by letting $j\rar\fz$.
\end{proof}

As a consequence of the preceding two lemmas, we obtain the following corollary.

\begin{corollary}\label{multmodco}
Let $p_0 ,p_1,q_0,q_1 \in (0,\infty]$ and $s_0, b_0, s_1,b_1 \in \rr$.
Let $f \in B^{s_0,b_0}_{p_0,q_0} (\rn)$ be locally integrable and $g \in B^{s_1,b_1}_{p_1,q_1} (\rn)$ such that $f g $ exists in $\cs'(\rn)$.
Then $ \lim_{j \to \infty} f  S^jg$ exists in $\cs'(\rn)$ and, moreover,
\begin{equation*}
f g =
  \lim_{j \rar \infty}   f  S^jg .
\end{equation*}
\end{corollary}

Corollary \ref{multmodco} also follows easily  from the following inequality: for any $N\ge 6$,
\begin{align*}
\lf|fg- f S^jg\r|
&\le\lf|fg-\lf(S^jf\r)S^jg\r|+\sum_{l=1}^{N}\lf|\lf(S_{j+l} f\r)S^jg\r|\\
&\quad+\lf|\lf(S^{j+N} f\r) S^jg-f S^jg\r|;
\end{align*}
we omit the details of its proof.

\begin{remark}\label{remark-lem+}
It is essentially obvious that  Lemmas \ref{mult} and \ref{multmod} and Corollary \ref{multmodco} remain true
after replacing the assumption $f\in B^{s_0,b_0}_{p_0,q_0}(\rn)$ by $f\in L^\fz(\rn)$.
This follows from $L^\fz(\rn) \hookrightarrow B^{0}_{\infty,\infty}(\rn)$.
\end{remark}

Now, we continue the proof of Lemma \ref{lem-M(B)}(iv).

\begin{proof}[Proof of Lemma \ref{lem-M(B)}(iv)]
Let $\{f_l\}_{l\in\nn}$ be a Cauchy sequence in $M(B^{s,b}_{p,q}(\rn))$. On the one hand,
(ii) of Lemma \ref{lem-M(B)} implies that there exists an $f\in L^\infty (\rn)$ such that
$f=\lim_{l\rar \fz}f_l$ in $L^\infty (\rn)$.
On the other hand, (i) of Lemma \ref{lem-M(B)} implies the existence of a linear continuous operator
$T \in \cl (B^{s,b}_{p,q}(\rn))$ such that
$$
 T := \lim_{l\rar\fz} T_{f_l} \qquad \mbox{in the sense of}\quad \cl (B^{s,b}_{p,q}(\rn)),
$$
where $T_{f_l}$ is the same as in \eqref{eq-Mnorm}.
Now, to prove the completeness of $M(B^{s,b}_{p,q}(\rn))$,
it suffices to show that $T=T_f$.

Let $g\in B^{s,b}_{p,q}(\rn)$. For any given $j\in\nn$ and any $l\in\nn$, we have
\begin{align}\label{eq-Banach-tri}
\lf[\lf(S^{j}f\r)S^j g - T g\,  \r] & =
\lf[\lf(S^{j}f\r)S^j g -\lf(S^{j}f_l\r)S^j g\r]+
\lf[\lf(S^{j}f_l\r)S^j g - f_l S^j g \r]\\
&\quad + \lf[ f_l  S^j g -  f_l   g \r] +\lf[ f_l  g - T g\r].\nonumber
\end{align}
Let $\vp$ be sufficiently small and $\varrho \in \cs (\rn)$.
By  the H\"{o}lder inequality and $f_l\rar f$ in $L^{\fz}(\rn)$ as $l\rar\fz$,
we find that there exists an $l_{1 (j,\varepsilon,\varrho)}\in\nn$ such that,
for any $l \ge l _{1(j,\varepsilon,\varrho)}$,
\begin{align}\label{eq-Banach-tri-1}
 &\lf|\lf(\lf[S^{j}f\r]S^j g\r) (\varrho)  -  \lf(\lf[S^{j}f_l\r]S^j g\r)(\varrho)\r|
 \\
& \quad\le
\lf\| S^{j}f- S^{j}f_l \r\|_{L^\infty (\rn)}  \lf\|S^j g\r\|_{L^\fz(\rn)} \lf\|\varrho\r\|_{L^1(\rn)}
\nonumber
\\
&  \quad\le
\lf\| f- f_l \r\|_{L^\infty (\rn)} \lf\|S^j g\r\|_{L^\fz(\rn)} \lf\|\varrho\r\|_{L^1(\rn)} < \varepsilon/3.
\nonumber
\end{align}
Applying the previous two estimates \eqref{newlab} and \eqref{eq-star6} with $f$ therein replaced by $f_l$,
and following an argument similar to that used in the proof of Corollary \ref{multmodco},
we conclude that, for $j$ large enough,
\begin{align}\label{eq-Banach-tri-2}
\lf|\lf(\lf[S^jf_l\r] S^j g\r)\lf(\varrho\r)-\lf(f_l  S^j g\r)\lf(\varrho\r)\r|<\vp/3;
\end{align}
Corollary  \ref{multmodco} also leads to
\begin{align}\label{eq-Banach-tri-3}
\lim_{j\to \infty}\, f_l S^j g =  f_l  g
\end{align}
in $\cs'(\rn)$.
In addition, since
$$
 \lf\| f_l  g - T g\r\|_{B^{s,b}_{p,q}(\rn)}
 \le \lf\| T_{f_l}  - T  \r\|_{M(B^{s,b}_{p,q}(\rn))} \|g \|_{B^s_{p,q}(\rn)},
$$
then, from the definition of $T$ and the simple embedding
$B^{s,b}_{p,q}(\rn)\hookrightarrow\cs'(\rn)$, it follows that
$\lim_{l\rar \infty}\, (f_l   g - T g) = 0$
in $\cs'(\rn)$. Thus, there exists an $l_{2 (j,\varepsilon,\varrho)}$ such that,
for any $l \ge l _{2(j,\varepsilon,\varrho)}$,
$$
\lf|\lf(f_l g\r)(\varrho)-\lf(Tg\r)(\varrho)\r|<\vp/3.
$$
Inserting this, \eqref{eq-Banach-tri-1}, \eqref{eq-Banach-tri-2},
and \eqref{eq-Banach-tri-3} into \eqref{eq-Banach-tri},
we obtain, for any $l\ge\max\{l _{1(j,\varepsilon,\varrho)},l _{2(j,\varepsilon,\varrho)}\}$,
\begin{align*}
&\lim_{j\rar\fz} \lf| \lf(\lf[S^{j}f\r]  S^j g \r)(\varrho) - T g (\varrho)\r|\\
&\quad \le
\lim_{j\rar\fz}\lf| \lf(f_l  S^j g\r)(\varrho) -  \lf(f_l g\r)(\varrho) \r|
+\vp
 =\vp.
\end{align*}
Thus, from the arbitrariness of $\vp$, it follows that
$$
\lim_{j\to \infty}\lf(S^{j}f\r)S^j g =  T g
$$
in the sense of $\cs'(\rn)$.
This, together with \eqref{eq-def-product}, means $f g = T g$ in $\cs'(\rn)$,
which completes the proof of Lemma \ref{lem-M(B)}(iv).
\end{proof}

Finally, we concentrate on the duality for multiplier spaces.
As mentioned above, we use the convention that $p'$ and $q'$ are defined
via the relations $1/p+1/p'=1$ and $1/q+1/q'=1$, respectively.
As a preparation, we first deal with $\wz{B}^{s,b}_{p,q}(\rn)$, the closure of $\cs(\rn)$ in $B^{s,b}_{p,q}(\rn)$.

\begin{lemma}\label{lem-duality}
Let $s,b\in\rr$.
\begin{enumerate}
\item[\rm(i)] If $p,q\in[1,\fz)$, then $\cs(\rn)$ is dense in $B^{s,b}_{p,q}(\rn)$.
\item[\rm(ii)] Let $p,q\in[1,\fz]$. Then
$$
\lf(\wz{B}^{s,b}_{p,q}(\rn)\r)'=B^{-s,-b}_{p',q'}(\rn).
$$
where $(\wz{B}^{s,b}_{p,q}(\rn))'$ denotes the dual space of $\wz{B}^{s,b}_{p,q}(\rn)$.
\end{enumerate}
\end{lemma}

Lemma \ref{lem-duality}(i)  is contained in \cite[p.\,13]{FL06}.
Under the restrictions $p\in(1,\fz)$ and $q\in[1,\fz)$, Lemma \ref{lem-duality}(ii) can be found in \cite[Theorem 3.1.10]{FL06}.
The remaining cases in (ii) can be easily proved via a slight modification of the previous cases.

\begin{proposition}\label{prop-M-duality}
Let $s,b\in\rr$.
\begin{itemize}
 \item [\rm(i)]
Let  $p,q\in(1, \infty)$. Then
$$
M\lf(B^{s,b}_{p,q}(\rn)\r)=M\lf(B^{-s,-b}_{p',q'}(\rn)\r)\, .
$$
\item[\rm(ii)]
Let $p\in(1, \infty]$ and  $q=\infty$. Then
\[
M\lf(\wz{B}^{s,b}_{p,\infty}(\rn)\r) \hookrightarrow M\lf({B}^{s,b}_{p,\infty}(\rn)\r)\, .
\]
\end{itemize}
\end{proposition}

\begin{proof}
To prove (i),
let $X$ be a Banach space, $f$ be a function such that
the related multiplier operator $T:=T_f\in\cl (X)$.
We denote by $X^*$ the dual Banach space of $X$, $\overline{f}$ the conjugate function of $f$,
and $T^*:=T_{\overline{f}}$ the adjoint operator of $T$.
Then $T^* \in \cl (X^*)$ and $\|T\|_{\cl(X)} = \|T^*\|_{\cl(X^*)}$.
In addition, we shall use the trivial  fact that $f \in M({B}^{s,b}_{p,q}(\rn))$
if and only if $\overline{f} \in M({B}^{s,b}_{p,q}(\rn))$.
Thus, the previous arguments and  Lemma \ref{lem-duality} imply
$$
M \lf(B^{s,b}_{p,q}(\rn)\r) \hookrightarrow
 M \lf(B^{-s,-b}_{p',q'}(\rn)\r)  \hookrightarrow
 M \lf(B^{s,b}_{p,q}(\rn)\r).
$$
This proves (i).

Now, we prove  (ii).
Let $f\in M(\wz{B}^{s,b}_{p,\infty}(\rn))$.
By  Lemma \ref{lem-duality}(ii) and the previous argument, we have
$f\in M(B^{-s,-b}_{p',1}(\rn))$.
This, combined with Lemma \ref{lem-duality} and $p'<\infty$, implies
$f\in M(B^{s,b}_{p,\infty}(\rn))$,
which completes the proof of Proposition \ref{prop-M-duality}.
 \end{proof}

\begin{remark}
Al least in the case that $p=\infty$, $b=0$, and $s\in(0,\infty)$, the embedding in
Proposition \ref{prop-M-duality}(ii) is proper. This  has been observed
 by Triebel in \cite[Remark 10]{Tr03}.
Indeed, when $s\in(0,\fz)$, we know that
$M(B^{s}_{\infty,\infty}(\rn))= B^{s}_{\infty,\infty}(\rn)$
in the sense of equivalent norms; see, for instance, \cite[Theorem 1.7]{NS18}.
Now, we choose
$f(x):= \phi_0 (x)\, |x|^s$, $\forall\,x\in \rn$, where $\phi_0$ is the same as in \eqref{eq-phi0}.
By \cite[Remark 2.8.4/2]{Tr83} or \cite[Lemma 2.3.1/1(ii)]{RS96},
we know that $f\in B^{s}_{\infty,\infty}(\rn)$
and hence $f\in M(B^{s}_{\infty,\infty}(\rn))$.
To show that $f\notin M(\wz{B}^{s}_{\infty,\infty}(\rn))$,
we select this function $\phi_0\in \wz{B}^{s}_{\infty,\infty}(\rn)$ as the test function.
The behavior of $f\phi_0$ at the origin implies that $f\phi_0$ cannot be approximated in
$B^{s}_{\infty,\infty}(\rn)$ by Schwartz functions.
 More precisely,
this follows from the characterization of $B^{s}_{\infty,\infty}(\rn)$ by differences
(see, for instance, \cite[Theorem 2.5.12]{Tr83}) and the following observation: for any $g\in \cs (\rn)$,
there exists a suitable $h_0\in(0,\fz)$ depending on $g$ such that,
for any $h\in\rn\setminus\{\mathbf{0}\}$ with $|h|<h_0$,
$$
\frac{|(\Delta_h^M f)\phi_0 (\mathbf{0}) - \Delta_h^M g (\mathbf{0})|}{|h|^{s}}
= \frac{|C_s \, |h|^{s} - \Delta_h^M g (\mathbf{0})|}{|h|^{s}}\ge \frac{C_s}{2}
$$
holds true for some appropriate positive constant $C_s$, where $M$ is a positive integer strictly larger than $s$.
Here and thereafter, $\Delta_h^m, \ m\in\nn,$ is the well-known difference operator,
that is, for any function $u$ on $\rn$,
\begin{equation}\label{eq-def-difference}
\Delta_h^1 u(\cdot):=u(\cdot+h)-u(\cdot),\
\Delta_h^m u:=\Delta_h^1\lf(\Delta_h^{m-1} u\r),\ m\in\lf\{2,3,\ldots\r\}.
\end{equation}
Thus, we conclude that
$\|f\phi_0-g\|_{B^{s}_{\infty,\infty}(\rn)}$ cannot be arbitrarily small and hence
$f\phi_0\notin\wz{B}^{s}_{\infty,\infty}(\rn)$.
Altogether, we obtain $f \in M(B^{s}_{\infty,\infty}(\rn)) \setminus M(\wz{B}^{s}_{\infty,\infty}(\rn))$.
Thus, the above claim holds true.
\end{remark}


\section{Pointwise Multiplier Space of $B^{0,b}_{p,\infty}(\rn)$}
\label{section3}


 In this section, we give our main result, the characterization of  the pointwise multiplier space of $B^{0,b}_{p,\fz}(\rn)$,
as well as  three typical examples of these pointwise multipliers, including characteristic functions of open sets,
classes of continuous functions defined by differences, and exponential functions.
We divide the description of $M(B^{0,b}_{p,\infty}(\rn))$   into three cases: $p=1$, $p=\fz$, and $p\in(1,\fz)$, which
are presented,
respectively, in Subsections \ref{subsec1}, \ref{sec-p=fz}, and \ref{sec-1<p<fz}.
All the proofs of these  results
are given in the forthcoming Section \ref{section4}. It should be pointed out that some of these proofs, especially those for necessary parts, are rather
constructive and complicated.

Let us start with one remark.
It is well known that the space $B^{0}_{p,\infty}(\rn)$
contains singular distributions. In addition, it is also known that $\cs (\rn)$
is not dense in $B^{0}_{p,\infty}(\rn)$.
This makes clear  that  we cannot simply deal with the usual pointwise definition
of a product.

Only very few results are known  about characterizations of
$M(B^0_{p,q} (\rn))$. We refer to Koch and Sickel \cite{KS02},
where  $M(B^0_{\infty,1} (\rn))$ and $M(B^0_{\infty,\infty} (\rn))$ are investigated;
see the following two subsections for further details.
There is one more  result which is at least related; see \cite[Remark 13.2,~p.\,136]{FJ90}.
It concerns the Triebel-Lizorkin spaces $F^0_{p,q}(\rn)$ and states that, if $p\in[1,\fz)$, $q\in[1,\fz]$, and $q \neq 2$,
\[
 \limsup_{k \to \infty}\left\| e^{i2^kx_1} \right\|_{M(F^0_{p,q}(\rn))} = \infty.
\]

This implies that
$\|\cdot\|_{M(F^0_{p,q}(\rn))}$ and $\|\cdot\|_{L^\infty (\rn)}$
are not equivalent.
The  simple argument, used by Frazier and Jawerth, carries over to
$M(B^0_{p,\infty}(\rn))$.

\begin{lemma}\label{fj1}
Let $p\in[1,\infty)$. Then  $M(B^{0}_{p,\infty}(\rn))$ is a proper subset of $L^\infty(\rn)$.
\end{lemma}

\begin{proof}
Recall that $M(B^{0}_{p,\infty}(\rn))$ is a Banach space which is continuously embedded
into $L^\infty (\rn)$; see Lemma \ref{lem-M(B)}.
If we assume that $M(B^{0}_{p,\infty}(\rn))$ coincides with $L^\infty (\rn)$ as sets,
then, by \cite[Corollary 2.12(c)]{Ru91}, the norms $\|\cdot\|_{M(B^{0}_{p,\infty}(\rn))}$
and $\|\cdot \|_{L^\infty(\rn)}$ must be equivalent.
Thus, to prove this lemma,
it will be sufficient to show that the norms are not equivalent.

To this end, let $\{\phi_j\}_{j\in\zz_+}$ be the smooth dyadic decomposition in
\eqref{eq-phi0}-\eqref{eq-phik}. Recall that
$$\phi_1 (2,0,\, \ldots \, ,0)= 1\ \ \mbox{and}\ \
\varphi_k := (2\pi)^{-\f n 2}\, \cf^{-1} \phi_k,\quad \forall\ k \in \zz_+;$$ see \eqref{eq-S_k}.
Now, we define
$\Theta (\xi) := e^{-i 2 \xi_1} \varphi_1 (\xi)$ for any $\xi:=(\xi_1,\ldots,\xi_n) \in \rn$.
Clearly, $\cf\Theta (0) = 1$. Thus, $\{2^{jn} \, \Theta (2^j\, \cdot\, )\}_{j\in\nn}$ defines an
approximation of the unity, which implies that, for any given $f\in L^p (\rn)$,
\begin{align*}
\lf\|f\r\|_{L^p(\rn)} & =  \lim_{j\to \infty} \lf\| f \ast 2^{jn}  \Theta (2^j \cdot )\r\|_{L^p(\rn)}
\le \sup_{j\ge 1} \lf\| f \ast \lf[2^{jn}  e^{-i (2^{j+1} \cdot)} \varphi_1 (2^j \cdot )\r] \r\|_{L^p(\rn)}
\\
& \le  \sup_{j\ge 1} \lf\|  \lf[f(\cdot)  e^{i (2^{j+1}\cdot)}\r]\ast \varphi_{j+1} \r\|_{L^p(\rn)} \\
& \le \lf\|f(\cdot)  e^{i (2^{j+1}\cdot)} \r\|_{B^0_{p,\infty}(\rn)}.
\end{align*}
Next, let us assume that
the family $\{e^{i 2^{j+1}\cdot}\}_{j\in\nn}$ induces a sequence
of uniformly bounded pointwise multipliers in $M(B^{0}_{p,\infty} (\rn))$.
Then, by the previous estimates, we find that,
for any $f \in \cs(\rn)$,
\[
\|f\|_{L^p(\rn)}
\le   \lf\| e^{i (2^{j+1}\cdot)} \r\|_{M(B^{0}_{p,\infty}(\rn))}
\lf\| f\r\|_{B^{0}_{p,\infty}(\rn)}
\ls\lf\| f\r\|_{B^{0}_{p,\infty}(\rn)},
\]
which further implies that
\begin{equation}\label{ebdd-1}
\wz{B}^{0}_{p,\infty}(\rn) \hookrightarrow L^p (\rn)=F^0_{p,2}(\rn).
\end{equation}
If $q=\fz$, then, by \eqref{ebdd-1}  and Lemma \ref{lem-duality}(ii), one has
\begin{equation}\label{ebdd-2}
L^{p'}(\rn) \hookrightarrow {B}^{0}_{p',1}(\rn).
\end{equation}
To show that this embedding is not true,
we divide our considerations  into two cases.
When $p\in(1,\fz)$, we employ
\[
L^{p'}(\rn)= F^0_{p',2} \hookrightarrow {B}^{0}_{p',q}(\rn)
\quad \Longleftrightarrow \quad q\ge \max (p',2);
\]
see \cite[Theorem 3.1.1(i)]{ST95}. When $p=1$, we know that
$B^0_{\infty,1}(\rn)$ is contained in $C^0(\rn)$ (the space of all bounded and uniformly continuous functions on $\rn$);
see, for instance, \cite[Remark 2.7.1/2]{Tr83} or \cite[Theorem 3.3.1(ii)]{ST95}.
This disproves \eqref{ebdd-2} for all $p \in [1,\infty)$.

Altogether, our assumption that $M(B^{0}_{p,\fz} (\rn))$ and $L^\infty (\rn)$ coincide as sets is wrong and, therefore, the embedding from
$M(B^{0}_{p,\fz} (\rn))$ into $L^\infty (\rn)$ must be proper.
This finishes the proof of Lemma \ref{fj1}.
\end{proof}

\begin{remark}
 \rm
 The above used arguments can also partly extend to the more general situation of
 $M(B^{0,b}_{p,q}(\rn))$ with $ p,q \in [1,\infty]$ and $b \in \rr$.
 However, we will not do that here.
 Later on, as a trivial consequence of Theorems \ref{expo3}, \ref{expo4}, \ref{expo5}, and
 \ref{expo7}, we shall obtain
 \[
 \limsup_{|k| \to \infty}\lf\| e^{i2^kx}\r \|_{M(B^{0,b}_{p,\infty}(\rn))} = \infty
 \]
for all $b\in \rr$ and $p \in [1,\infty]$. Thus, under these
restrictions, we conclude that $M(B^{0,b}_{p,\infty}(\rn))$ is a proper subset of
$L^\infty (\rn)$.
\end{remark}


\subsection{Pointwise Multiplier Space of $B^{0,b}_{1,\infty}(\rn)$}
\label{subsec1}


In this subsection, we concentrate on the pointwise multiplier space of
$B^{0,b}_{1,\infty}(\rn)$. The main result is as follows.

\begin{theorem}\label{p=1}
Let $b\in\rr$. A function $f$ is a pointwise multiplier of $B^{0,b}_{1,\fz}(\rn)$
if and only if $f\in L^{\fz}(\rn)$ and $\|f\|^{(1)}_{2,b}+\|f\|^{(1)}_{3,b}<\fz$,
where
$$
\|f\|^{(1)}_{2,b}:=
\sup_{l\in\zz_+}\sum_{k = l}^\infty \lf(\f{1+l}{1+k}\r)^b
\lf\|S_kf\r\|_{L^{\fz}(\rn)}
$$
and
$$
\|f\|^{(1)}_{3,b}:=\sup_{k\ge2}\sum_{l=0}^{k-2}\lf(\f{1+k}{1+l}\r)^{b}
\sup_{\gfz{P\in\cq}{l(P)=2^{-l}}}\fint_P\lf|S_kf(y)\r|\,dy.
$$
Furthermore, $\| f\|_{M(B^{0,b}_{1,\fz}(\rn))}$ is equivalent to
$\lf\|f\r\|_{L^{\fz}(\rn)} + \|f\|^{(1)}_{2,b}+\|f\|^{(1)}_{3,b}$.
\end{theorem}

\begin{remark}
Of some particular interest is the case $b=0$.
Obviously we have
\[
\|f\|_{B^0_{\infty,1}(\rn)} = \sup_{l\in\zz_+}\sum_{k = l}^\infty \lf\|S_kf\r\|_{L^{\fz}(\rn)} .
\]
Thus, the characterization of $M(B^0_{1,\infty}(\rn))$
reads as follows: $M(B^0_{1,\infty}(\rn))$ coincides with the
collection of all the functions $f\in B^{0}_{\infty,1}(\rn)$ such that
$$
 \sup_{k\ge2} \sum_{l=0}^{k-2}
\sup_{\gfz{P\in\cq}{l(P)=2^{-l}}}\fint_P\lf|S_kf(y)\r|\,dy <\fz.
$$
Furthermore,
\[
\lf\|f\r\|_{B^{0}_{\infty,1}(\rn)} + \sup_{k\ge2}\, \sum_{l=0}^{k-2}
\sup_{\gfz{P\in\cq}{l(P)=2^{-l}}}\fint_P\lf|S_kf(y)\r|\,dy
\]
is equivalent to $\| f\|_{M(B^{0}_{1,\fz}(\rn))}$.
\end{remark}

There are  much  simpler sufficient conditions.

\begin{corollary}\label{approx}
\begin{enumerate}
\item[{\rm (i)}]
Let $b \in(1,\fz)$.
Then  $ B^{0,b}_{\infty,\infty}(\rn) \hookrightarrow
M(B^{0,b}_{1,\fz}(\rn))$.

\item[{\rm (ii)}]
Let $b =1$.
Then a function $f\in L^\infty (\rn)$, satisfying
$$
\sup_{k\in\nn} (1+k)  \ln (1+k) \lf\|S_kf\r\|_{L^{\fz}(\rn)}<\infty,
$$
belongs to $M(B^{0,b}_{1,\fz}(\rn))$.

\item[{\rm (iii)}]
Let $b\in(0,1)$.
Then
$$
\lf[L^\infty (\rn) \cap B^{0,1}_{\infty,\infty}(\rn)\r] \hookrightarrow
M\lf(B^{0,b}_{1,\fz}(\rn)\r).
$$

\item[{\rm (iv)}]
Let $b =0$.
Then
$$
\lf[B^0_{\infty,1} (\rn) \cap B^{0,1}_{\infty,\infty}(\rn)\r] \hookrightarrow
M\lf(B^{0}_{1,\fz}(\rn)\r).
$$

\item[{\rm (v)}]
Let $b \in(-\fz,0)$. For any given $\alpha \in( |b|+1,\fz)$, a function
$f\in L^\infty (\rn)$, satisfying
$$
\sup_{k\in\zz_+} (1+k)^\alpha   \lf\|S_kf\r\|_{L^{\fz}(\rn)}<\infty,
$$
belongs to $M(B^{0,b}_{1,\fz}(\rn))$.
\end{enumerate}
\end{corollary}

\begin{remark}\label{netrusov}
 \rm
Let $ s\in(0,n)$.
Recall that Netrusov \cite[Theorem 3]{Ne92} has proved the following characterization
of $M(B^s_{1,\infty}(\rn))$. That is,
$M(B^s_{1,\infty}(\rn))$ coincides with the collection of all the $f \in L^\infty (\rn)$ such that there exists a representation
$$
f= \sum_{j=0}^\infty f_j \qquad \mbox{(convergence \ \  in \ $\cs'(\rn)$)}
$$
satisfying that
\begin{equation}\label{ws-03}
\sup_{i \in \zz_+}  2^{is}  \int_{2^{-i-1}}^1
\lf[\sup_{x \in \rn} \int_{B(x,t)}  \lf|f_i(y)\r|\,  dy\r]  t^{s-n}\,  \frac{dt}{t}
<\infty,
\end{equation}
where, for any $j\in\zz_+$, $f_j\in\cs'(\rn)\cap L^1_{\loc}(\rn)$,
$$
\supp \cf f_0 \subset B(\mathbf{0},2)\, ,
$$
and, for any $j\in\nn$,
$$
\supp \cf f_j \subset B(\mathbf{0},2^{j+1})\setminus B(\mathbf{0},2^{j-1}), \quad \forall\ j \in \nn.
$$
Observe that, using
$$
 \sup_{x \in \rn} \int_{B(x,t)}  |f_i(y)|\,  dy  \sim
 \sup_{x \in \rn} \int_{B(x,t/2)}  |f_i(y)|\,  dy
$$
with the implicit positive constants independent of $t$, $i$, and $f_i$,
we can rewrite \eqref{ws-03} as
$$
\sup_{i \in \zz_+}  2^{is}  \sum_{l=0}^i   2^{-ls}
\sup_{x \in \rn} \fint_{B(x,2^{-l})}  |f_i(y)|\, dy<\fz.
$$
By choosing $f_i := S_i f$,
this implies that a function $f \in L^{\infty}(\rn)$ belongs to $M(B^s_{1,\infty}(\rn))$
if
\begin{equation}\label{eq-lefthand}
 \sup_{i \in \zz_+}  2^{is} \sum_{l=0}^i   2^{-ls}
\sup_{x \in \rn} \fint_{B(x,2^{-l})}  |S_if(y)|\, dy < \infty.
\end{equation}
Observe that, when $s=0$, \eqref{eq-lefthand} becomes
$$
\sup_{k\ge2} \sum_{l=0}^{k-2}
\sup_{\gfz{P\in\cq}{l(P)=2^{-l}}}\fint_P\lf|S_kf(y)\r|\,dy.
$$
Comparing this with our  characterization of $M(B^{0,b}_{1,\infty}(\rn))$
in Theorem \ref{p=1}, we are missing the second term $\|f\|^{(1)}_{2,b}$.
This indicates that  the passage from $s\in(0,\fz)$ to $s=0$ is more complicated than
probably expected.
\end{remark}

To increase transparency, as applications of the characterization of $M(B^{0,b}_{1,\fz}(\rn))$ in Theorem \ref{p=1},
below we will treat three different types of concrete examples:
\begin{enumerate}
\item[\rm(i)] Characteristic functions of open sets.
\item[\rm(ii)] Classes of continuous functions defined by differences.
\item[\rm(iii)] The functions $ e^{ikx}$ with $x \in \rn$ and $k \in \zn$.
\end{enumerate}


\subsubsection*{Characteristic functions of open sets}


An interesting class of pointwise multipliers is given by characteristic functions
which can reflect some information of truncated functions.
Since the early sixties, the problem whether or not
the characteristic function of $\rn_+$ is a pointwise multiplier for function spaces
with fractional order of smoothness has been discussed.
We mention here Strichartz, Lions, Magenes, Shamir, Triebel, and Franke;
see \cite{S67,LM72,Tr83,F86,RS96,Tr03,Tr06}.
This question is also of interest for our logarithmic Besov spaces.

Let $E$ be a measurable set in $\rn$.
By $\mathbf{1}_{E}$ we denote the corresponding characteristic function.
We will say that $E$ is \emph{nontrivial} if both $E$ and  $\rn \setminus E$
are of  positive measure.

\begin{theorem}\label{rn+}
Let  $E\subset\rn$ be a nontrivial  measurable set.
Then, for any $b\in(-\fz,0]$, $\mathbf{1}_{E}\notin M(B^{0,b}_{1,\fz}(\rn))$.
\end{theorem}

\begin{remark}
The proof given below (see Subsection \ref{subsec-pf.eg.1}) yields even a sharper result, namely
\[
 M(B^{0,b}_{1,\fz}(\rn)) \hookrightarrow B^0_{\infty,1}(\rn) .
\]
\end{remark}

Probably Theorem \ref{rn+} can extend to $b\in\rr$, but here we have only a partial result.
The characteristic function of a cube has, in a sense, the maximal regularity
within the set of all characteristic functions.
For this particular characteristic function, we are able to extend
Theorem \ref{rn+} from $b\in(-\fz,0]$ to $b\in\rr$.

\begin{theorem}\label{rn+-}
Let $b\in \rr$.
\begin{enumerate}
\item[\rm(i)] Let $\mathbf{1}_n$ be the characteristic function of the cube
$(-1,1)^n$.
Then $\mathbf{1}_n\notin M(B^{0,b}_{1,\fz}(\rn))$.

\item[\rm(ii)] $\mathbf{1}_{\rn_+} \notin M(B^{0,b}_{1,\fz}(\rn))$.
\end{enumerate}
\end{theorem}

\begin{remark}
 Recall that, for any $s\in(0,1)$, $\mathbf{1}_{\rn_+} \in M(B^{s}_{1,\fz}(\rn))$;
 see \cite[Theorem 2.8.7]{Tr83} or \cite[Theorem 4.6.3/1]{RS96}.
It is a little bit surprising that there is no continuation to
$M(B^{0,b}_{1,\fz}(\rn))$ with $b$ sufficiently large.
\end{remark}


\subsubsection*{Continuous functions}


Here, we consider sufficient conditions in terms of differences.
For this, we first introduce the related spaces. Let $m \in \nn$.
For any function $f \in L^p_{\loc}(\rn)$, let
\begin{equation}\label{eq-def-w}
\omega_m (f,t)_p := \sup_{h\in\rn:\,|h|<t} \lf\|  \Delta_h^m f \r\|_{L^p (\rn)} , \ \forall\, t\in(0,\fz) ,
\end{equation}
be the $m$-th order modulus of smoothness of $f$, where $\Delta_h^m$ is the same as in \eqref{eq-def-difference}.

\begin{definition}
Let $s,b,d\in \rr$ and  $p,q\in[1,\infty]$. Let $m \in \nn$ satisfy $m>s$.
The space $\mathbf{B}^{s,b,d}_{p,q}(\rn)$ is defined as the collection of all the $f \in L^p(\rn)$
such that
\begin{align*}
 \| \, f\, \|_{\mathbf{B}^{s,b,d}_{p,q}(\rn)}& :=
 \|\, f\, \|_{L^p(\rn)}\\
 &\quad + \lf[\int_0^1 \lf\{t^{-s}
 (1-\log t)^b \lf[1+ \log (1-\log t)\r]^d  \omega_m (f,t)_p\r\}^q \frac{dt}{t}\r]^{\f1 q}
\end{align*}
is finite.
\end{definition}

If $d=0$, we write $\mathbf{B}^{s,b}_{p,q}(\rn)$ instead of
$\mathbf{B}^{s,b,0}_{p,q}(\rn)$.
For more details and further references with respect to the above defined  scale of function spaces, we refer to Cobos and Dom\'{\i}nguez \cite{CD16}.
Recall that a function $f$ is said to be \emph{Dini continuous} if
\begin{equation}\label{eq-def-Dini}
\|f\|_{C_D (\rn)}:=\int_{0}^{1/2}\omega_1 (f,t)_\fz\,\f{dt}{t}<\fz.
\end{equation}
Observe that $\mathbf{B}^{0,0}_{\infty,1}(\rn)$ is just the class $C_D (\rn)$ of all Dini continuous functions,
which is widely used in many branches of mathematics, for instance, functional analysis (see, for instance, \cite{Ga07}) and
PDEs (see, for instance,  \cite{GT01}).
As an almost obvious consequence of Corollary \ref{approx}, we establish the following embeddings.

\begin{corollary}\label{approx3}
\begin{enumerate}
\item[\rm (i)]
Let $b\in(1,\fz)$.
Then  $$ \mathbf{B}^{0,b}_{\infty,\infty}(\rn) \hookrightarrow
M\lf(B^{0,b}_{1,\fz}(\rn)\r).$$

\item[\rm (ii)]
Let $b\in(0,1)$. Then
$$
C_D (\rn) \hookrightarrow \mathbf{B}^{0,1}_{\infty,\infty}(\rn) \hookrightarrow
M\lf(B^{0,b}_{1,\fz}(\rn)\r).
$$
\end{enumerate}
\end{corollary}


\subsubsection*{Exponentials}


Now, we turn back to the examples $\{f_k(x):= e^{ik\cdot x},\ \forall\,x\in\rn\}_{k \in \zn}$,
which we already investigated in Lemma \ref{fj1}.
Using the characterization of $M(B^s_{1,\infty}(\rn))$ by
Netrusov (see Remark \ref{netrusov}), it is immediate that, when $s\in(0,n)$,
$$
\lf\| e^{ik\cdot x} \r\|_{M(B^s_{1,\infty}(\rn))} \sim (1+|k|)^s , \ \forall\,k \in \zn,
$$
where the positive equivalence constants are independent of $k$.
In addition, by the proof of Lemma \ref{fj1},
we also know that $\{\| e^{ik\cdot x} \|_{M(B^0_{1,\infty}(\rn))}\}_{k\in\zn}$ is unbounded.
Thus, it is quite natural to expect logarithmic growth of
$\| e^{ik\cdot x} \|_{M(B^{0,b}_{1,\infty}(\rn))}$ for any $k\in\zn$.
Indeed, we have the following conclusion.

\begin{theorem}\label{expo3}
Let $b\in \rr$.
Then, for any $k\in \zn \setminus \{\mathbf{0}\}$, $e^{ik\cdot x}$ is a pointwise multiplier of $B^{0,b}_{1,\infty}(\rn)$;
furthermore,
\begin{equation*}
\lf\| e^{ik\cdot x} \r\|_{M(B^{0,b}_{1,\infty}(\rn))} \sim
\begin{cases}
\ (1+ \ln |k|)^b &\  \mbox{if}\   b\in(1,\fz),
\\
\ (1+\ln |k|) \ln (1+\ln (|k|+1)) &\ \mbox{if}\   b=1,
\\
\ (1+ \ln |k|) &\  \mbox{if}\   b\in[-1,1),
\\
\ (1+ \ln |k|)^{|b|} &\  \mbox{if}\    b\in(-\fz,-1)
\end{cases}
\end{equation*}
with the positive equivalence constants depending only on both $n$ and $b$.
\end{theorem}


\subsection{Pointwise Multiplier Space of $B^{0,b}_{\infty,\infty}(\rn)$}\label{sec-p=fz}


In this subsection, we deal with the limit case $p=\fz$. The main result is as follows.

\begin{theorem}\label{p=infty}
Let $b\in\rr$. A function $f$ is a pointwise multiplier of $B^{0,b}_{\fz,\fz}(\rn)$
if and only if $f\in L^{\fz}(\rn)$ and
$\|f\|^{(\fz)}_{2,b}+\|f\|^{(\fz)}_{3,b} <\infty$,
where
$$
\|f\|^{(\fz)}_{2,b}:=\sup_{l\in\zz_+}  (1+l)^b  \sup_{P\in\cq:~l(P)=2^{-l}}\fint_P\sum_{k=l}^\infty (1+k)^{-b} \lf|S_kf(y)\r|\,dy
$$
and
\begin{align*}
\|f\|^{(\fz)}_{3,b}:=
\begin{cases}
\ \dsup_{k\ge2} (1+k)^b  \lf\|S_kf\r\|_{L^{\fz}(\rn)}
\  &\mbox{if}\  b\in(1,\fz),\\
\ \dsup_{k\ge2} (1+k)  \ln t  \lf\|S_kf\r\|_{L^{\fz}(\rn)}
\  &\mbox{if}\   b =1,\\
\ \dsup_{k\ge2} (1+k)  \lf\|S_kf\r\|_{L^{\fz}(\rn)}
\  &\mbox{if}\   b\in(-\fz, 1).
\end{cases}
\end{align*}
Furthermore, $\| f\|_{M(B^{0,b}_{\fz,\fz}(\rn))}$ is equivalent to
$\lf\|f\r\|_{L^{\fz}(\rn)}+
\|f\|^{(\fz)}_{2,b}+\|f\|^{(\fz)}_{3,b} $.
\end{theorem}

\begin{remark}
 \rm
Let us consider the case $b=0$. Then our characterization reads as follows:
a function $f\in M(B^{0}_{\fz,\fz}(\rn))$ if and only if  $f\in L^{\fz}(\rn)$ and $\|f\|^{(\fz)}_{2,0}+\|f\|^{(\fz)}_{3,0}  <\infty$,
where
\begin{align*}
\|f\|^{(\fz)}_{2,0}+\|f\|^{(\fz)}_{3,0}
& := \sup_{l\in\zz_+}    \sup_{P\in\cq:~l(P)=2^{-l}}\fint_P\sum_{k=l}^\infty  \lf|S_kf(y)\r|\,dy\\
&\quad+  \sup_{k\ge2} (1+k)  \lf\|S_kf\r\|_{L^{\fz}(\rn)}
\\
&= \|f\|_{F^0_{\infty,1}(\rn)} +   \|f\|_{B^{0,1}_{\infty,\infty}(\rn)}.
\end{align*}
This implies that
\[
 M(B^{0}_{\fz,\fz}(\rn)) = L^{\fz}(\rn) \cap
 F^0_{\infty,1}(\rn) \cap   B^{0,1}_{\infty,\infty}(\rn)
\]
in the sense of equivalent norms.
This coincides with the description of  $M(B^{0}_{\fz,\fz}(\rn))$ obtained in \cite[Theorem 4]{KS02}.
Also the independence of these three conditions is shown therein.
\end{remark}

Especially for the case $b\in(1,\fz)$, the description of $M(B^{0,b}_{\fz,\fz}(\rn))$ can be simplified.

\begin{corollary}\label{bp=infty}
Let $b \in(1,\fz)$. The pointwise multiplier space  $M(B^{0,b}_{\fz,\fz}(\rn))$ is
$B^{0,b}_{\fz,\fz}(\rn)$ itself and
$$
\|f\|_{M(B^{0,b}_{\fz,\fz}(\rn))}  \sim \|f\|_{B^{0,b}_{\fz,\fz}(\rn))}
$$
with the  positive equivalence constants independent of $f$.
\end{corollary}

\begin{remark}
\rm
Recall that, when $s\in(0,\fz)$, we have $M(B^{s}_{\fz,\fz}(\rn)) =  B^{s}_{\infty,\infty}(\rn)$
(in the sense of equivalent norms); see \cite[Theorem 1.7]{NS18}.
Thus, Corollary \ref{bp=infty} describes very well the passage from $s\in(0,\fz)$ to $s=0$.
\end{remark}


Next, applying the characterization of $M(B^{0,b}_{\fz,\fz}(\rn))$ obtained in Theorem \ref{p=infty}, we continue considering the same three  examples  as  studied in the previous subsection.

\subsection*{Characteristic functions of open sets}
For characteristic functions of open sets, as in the case $p=1$, the result is negative.

\begin{theorem}\label{rn++}
Let $b\in \rr$ and  $E$ be a nontrivial measurable set in $\rn$.
Then $\mathbf{1}_{E}\notin M(B^{0,b}_{\infty,\fz}(\rn))$.
\end{theorem}

\begin{remark}
 \rm
 \begin{enumerate}
 \item[(i)] For the case $b=0$ of Theorem \ref{rn++}, we refer to \cite[Proposition 18]{KS02}.
 \item[(ii)]  Recall that, by \cite[Theorem 2.8.7]{Tr83} or \cite[Theorem 4.6.3/1]{RS96}, $\mathbf{1}_{\rn_+} \in M(B^{s}_{\fz,\fz}(\rn))$
 for any $s\in(-1,0)$; on the other hand, Theorem \ref{rn++} implies that
 $\mathbf{1}_{\rn_+} \notin M(B^{0,b}_{\fz,\fz}(\rn))$ for any $b\in(-\fz,0)$.
 We observe that there exists a gap between the case $s=0,b<0$ and the case $s<0,b=0$.
 \end{enumerate}
\end{remark}


\subsubsection*{Continuous functions}


As a direct consequence of both Corollary \ref{bp=infty} and Theorem \ref{p=infty},
we obtain the following embeddings.

\begin{corollary}\label{approx4}
\begin{enumerate}
\item[{\rm (i)}] Let $b\in(0,1)\cup(1,\fz)$.
Then $ \mathbf{B}^{0,\max\{1,b\}}_{\infty,\infty}(\rn) \hookrightarrow M(B^{0,b}_{\fz,\fz}(\rn))$.

\item[{\rm (ii)}] Let $b\in [0,1)$. Then
Dini continuous functions [see \eqref{eq-def-Dini}] belong to $M(B^{0,b}_{\fz,\fz}(\rn))$.
More exactly, $C_D (\rn) \hookrightarrow M(B^{0,b}_{\fz,\fz}(\rn))$.
\end{enumerate}
\end{corollary}

\begin{remark}
The case $b=0$ of Corollary \ref{approx4}(ii) has been proved in \cite[Lemma 20(i)]{KS02}.
\end{remark}


\subsubsection*{Exponentials}


For the exponential functions, we have the following conclusion.

\begin{theorem}\label{expo4}
Let $b\in \rr$.
Then, for any $k\in \zn \setminus \{\mathbf{0}\}$, $e^{ik\cdot x}$ is a pointwise multiplier of $B^{0,b}_{\fz,\infty}(\rn)$;
furthermore,
\begin{equation*}
\lf\|e^{ik\cdot x}\r\|_{M(B^{0,b}_{\infty,\infty}(\rn))} \sim \,
\begin{cases}
\ (1+ \ln |k|)^b &\  \mbox{if}\  b\in(1,\fz),
\\
\ (1+\ln |k|)\ln\lf (1+\ln \lf(|k|+1\r)\r) &\  \mbox{if}\  b=1,
\\
\ (1+ \ln |k|) &\ \mbox{if}\  b\in[-1,1),
\\
\ (1+ \ln |k|)^{|b|} &\  \mbox{if}\    b\in(-\fz,-1)
\end{cases}
\end{equation*}
with the positive equivalence constants depending only on both $n$ and $b$.
\end{theorem}

\begin{remark}
 \rm
Comparing Theorem \ref{expo3} and Theorem \ref{expo4},
one can observe that
$$
\lf\| e^{ik\cdot x}\r\|_{M(B^{0,b}_{\infty,\infty}(\rn))} \sim
\lf\| e^{ik\cdot x} \r\|_{M(B^{0,b}_{1,\infty}(\rn))}, \ \forall\,k\in\zz^n,
$$
with the positive equivalence constants independent of $k$.
\end{remark}


\subsection{Some Results Regarding $M(B^{0,b}_{p,\infty}(\rn))$ for
$p\in(1,\infty)$}\label{sec-1<p<fz}


Now, we turn to the  case $p\in(1,\infty)$.
Here, our results are less complete.

\begin{theorem}\label{thm-suff-p}
Let $p\in (1,\fz)$ and $b\in\rr$.
\begin{enumerate}
\item[\rm(i)] {\rm (Sufficiency)} A function $f$ is a pointwise multiplier of $B^{0,b}_{p,\fz}(\rn)$
if $f\in L^{\fz}(\rn)$ and $$\|f\|_{2,b}^{(p)}+\|f\|_{3,b}^{(p)}<\fz,$$
where
\begin{align*}
\lf\|f\r\|_{2,b}^{(p)}:=\sup_{l\in\zz_+}\sum_{k = l}^\infty \lf(\f{1+l}{1+k}\r)^b
\sup_{l(P)=2^{-l}}\lf[\fint_P\lf|S_kf(y)\r|^{p'}\,dy\r]^{\f{1}{p'}}
\end{align*}
and
\begin{align*}
\lf\|f\r\|_{3,b}^{(p)}:=
\sup_{k\ge2}\sum_{j=0}^{k-2}\lf(\f{1+k}{1+j}\r)^{b}
\sup_{l(P)=2^{-j}}\lf[\fint_P\lf|S_kf(y)\r|^p\,dy\r]^{\f1 p}.
\end{align*}
Furthermore, there exists a positive constant $C_1$ independent of $f$ such that
\begin{align}\label{eq-p-s}
\lf\|f\r\|_{M(B^{0,b}_{p,\fz}(\rn))}&\le C_1\lf[\|f\|_{L^\fz(\rn)}+
\|f\|_{2,b}^{(p)}+\|f\|_{3,b}^{(p)}\r];
\end{align}
\item[\rm(ii)] {\rm (Necessity)}
If $f\in M(B^{0,b}_{p,\fz}(\rn))$, then
$$\wz{\|f\|}_{2,b}^{(p)}+\wz{\|f\|}_{3,b}^{(p)}<\fz,$$
where
\begin{align*}
\wz{\|f\|}_{2,b}^{(p)}:=\sup_{l\in\zz_+}\sup_{l(P)=2^{-l}}\sum_{k=l}^\infty \lf(\f{1+l}{1+k}\r)^{b}
\lf[\fint_{P}\lf|S_kf(z)\r|^{p'}dz\r]^{\f{1}{p'}}
\end{align*}
and
\begin{align*}
\wz{\|f\|}_{3,b}^{(p)}:=\sup_{k\ge2}\lf\{\sum_{j=0}^{k-2}\lf(\f{1+k}{1+j}\r)^{bp}
\sup_{l(P)=2^{-j}}\fint_P\lf|S_kf(y)\r|^p\,dy\r\}^{\f{1}{p}}.
\end{align*}
Furthermore, there exists a positive constant $C_2$ independent of $f$ such that
\begin{align}\label{eq-p-n}
\wz{\|f\|}_{2,b}^{(p)}+\wz{\|f\|}_{3,b}^{(p)}\le C_2\lf\|f\r\|_{M(B^{0,b}_{p,\fz}(\rn))}.
\end{align}
\end{enumerate}
\end{theorem}

\begin{remark}
 \begin{enumerate}
\item[(i)] Clearly, comparing
$\|f\|_{2,b}^{(p)}+\|f\|_{3,b}^{(p)}$ and $\wz{\|f\|}_{2,b}^{(p)}+\wz{\|f\|}_{3,b}^{(p)}$
in Theorem \ref{thm-suff-p},
we find that there exists a gap between these two sums, caused by the different  positions of both
the supremum $\sup_{l(P)=2^{-l}}$ and the sum $\sum_{l=0}^{t-2}$.

\item[(ii)]
Let $p\in (1,\fz)$ and $b=0$. Then both \eqref{eq-p-s} and \eqref{eq-p-n} in Theorem \ref{thm-suff-p} are simplified to
\begin{align*}
\lf\|f\r\|_{M(B^{0}_{p,\fz}(\rn))}&\le C_1\lf\{\|f\|_{L^\fz(\rn)}+
\sup_{l \in\zz_+} \sum_{k = l}^{\fz}
\sup_{l(P)=2^{-l}}\lf[\fint_P\lf|S_kf(y)\r|^{p'}\,dy\r]^{\f{1}{p'}}\r.\\
&\quad+\lf.\sup_{k\ge2}\sum_{j=0}^{k-2}
\sup_{l(P)=2^{-j}}\lf[\fint_P\lf|S_kf(y)\r|^p\,dy\r]^{\f1 p}\r\}
\end{align*}
and
\begin{align*}
\sup_{l\in\zz_+}\sup_{l(P)=2^{-l}}\sum_{k=l}^\infty
\lf[\fint_{P}\lf|S_kf(z)\r|^{p'}dz\r]^{\f{1}{p'}} &
+ \sup_{k\ge2}\lf\{\sum_{j=0}^{k-2}
\sup_{l(P)=2^{-j}}\fint_P\lf|S_kf(y)\r|^p\,dy\r\}^{\f{1}{p}}
\\
& \le C_2\lf\|f\r\|_{M(B^{0,b}_{p,\fz}(\rn))}.
\end{align*}
\end{enumerate}
\end{remark}

Again, below we consider the three concrete examples in this case.


\subsection*{Characteristic functions of open sets}


This time, there exist characteristic functions belonging to $M(B^{s}_{p,\infty}(\rn))$
for some $s$ including $0$.
We refer to Gulisashvili \cite{Gu84,Gu85}, Triebel \cite[2.8.7]{Tr83},  Frazier and Jawerth \cite{FJ90},
and \cite[4.6.3]{RS96}.

We first concentrate on the characteristic function
of the upper half-space in the case $b=0$.
Let $p,q\in[1,\infty]$.
Then we would like to mention that
$\mathbf{1}_{\rn_+} \in M(B^s_{p,q}(\rn))$
if and only if
$$
\frac 1p -1 < s< \frac 1p;
$$
see, for instance, \cite[Theorem 2.8.7]{Tr83} or \cite[Theorem 4.6.3/1]{RS96}.
In particular, this means that, for any $p\in(1,\fz)$,
$\mathbf{1}_{\rn_+} \in M(B^0_{p,\infty}(\rn))$.
For later use, we will be more general than necessary.

\begin{theorem}\label{thm-Q1:binR}
Let $p,q\in[1,\fz]$ and $s\in\rr$.
If $1/p -1< s< 1/p$,
then, for any $b\in\rr$, $\mathbf{1}_{\rn_+} \in M(B^{s,b}_{p,q}(\rn))$.
\end{theorem}

An interesting question is whether or not the condition $1/p -1< s< 1/p$ in Theorem \ref{thm-Q1:binR} is sharp, which is still unknown.

For the special case $s=0$ and $q=\fz$ we are interested in, the result is given as follows.

\begin{corollary}\label{1infty}
Let $p\in(1,\fz)$ and $b\in\rr$. Then
$\mathbf{1}_{\rn_+} \in M(B^{0,b}_{p,\fz}(\rn))$.
\end{corollary}

\begin{remark}
 \rm
 Characteristic functions of more general sets $E$
are investigated in Gulisashvili \cite{Gu84,Gu85},
\cite[Section 4.6.3]{RS96}, Triebel \cite{Tr03,Tr06}, and
\cite{S99a,S99b,S18}. Partly these results carry over (just by interpolation),
but this will not lead to final assertions.
We do not go into detail.
\end{remark}


\subsubsection*{Exponentials}


We continue considering the functions $f(x)= e^{ik\cdot x}$ with $x\in \rn$ and $k \in \zz^n$. Here, since we have no complete characterization of $M(B^{0,b}_{p,\fz}(\rn))$ when $p\in(1,\fz)$,
we need some arguments which are different from those used in
the cases $p=1$ and $p=\fz$.
We split our considerations into two cases: $p\in(1,2]$ and $p\in(2,\fz]$.
This time the sufficient conditions stated in Theorem \ref{thm-suff-p} are not good enough, and we give some appropriate modifications
in Lemmas \ref{lem-suff-4} and \ref{lem-suff-5}.
The proof of the necessity is still quite constructive.

\begin{theorem}\label{expo5}
Let $p\in(1,2]$ and  $b\in \rr$.
Then, for any $k\in \zn \setminus \{\mathbf{0}\}$,
$e^{ik\cdot x}$ is a pointwise multiplier of $B^{0,b}_{p,\infty}(\rn)$;
furthermore,
\begin{equation*}
\lf\|e^{ik\cdot x}\r\|_{M(B^{0,b}_{p,\infty}(\rn))} \sim
\begin{cases}
\ (1+ \ln |k|)^b &\ \mbox{if}\  b\in(\f1 p,\fz),
\\
\ \lf[(1+\ln |k|) \ln \lf(1+\ln (|k|+1)\r)\r]^{\f1 p} &\ \mbox{if}\  b=\f1 p,
\\
\ (1+ \ln |k|)^{\f1 p} &\  \mbox{if}\  b\in[-\f1 p,\f1 p),
\\
\ (1+ \ln |k|)^{|b|} &\  \mbox{if}\    b\in(-\fz,-\f1 p)
\end{cases}
\end{equation*}
with the positive equivalence constants independent of $k$.
\end{theorem}

\begin{theorem}\label{expo7}
Let $p\in(2,\fz)$ and  $b\in \rr$.
Then, for any $k\in \zn \setminus \{\mathbf{0}\}$, $e^{ik\cdot x}$ is a pointwise multiplier of $B^{0,b}_{p,\infty}(\rn)$;
furthermore,
\begin{equation*}
\lf\|e^{ik\cdot x}\r\|_{M(B^{0,b}_{p,\infty}(\rn))}\sim
\begin{cases}
\ (1+ \ln |k|)^b &\  \mbox{if}\  b\in(\f1 2,\fz),
\\
\ \lf[(1+\ln |k|) \ln (1+\ln (|k|+1))\r]^{\f1 2} &\ \mbox{if}\  b=\f1 2,
\\
\ (1+ \ln |k|)^{\f1 2} &\  \mbox{if}\  b\in[-\f1 2,\f1 2),
\\
\ (1+ \ln |k|)^{|b|} &\  \mbox{if}\  b\in(-\fz,-\f1 2)
\end{cases}
\end{equation*}
with the positive equivalence constants independent of $k$.
\end{theorem}


\section{Proofs of Main Results in Section \ref{section3}}\label{section4}

 In this section, we present  all the related proofs about all the results in the last section.
Precisely, Subsections \ref{sec-suff-p} and \ref{sec-nece-p} respectively show the sufficient and the necessary conditions for $M(B^{0,b}_{p,\fz}(\rn))$ with $p\in[1,\fz]$ and $b\in \rr$,
which are the fundamental steps to acquire the main results of this article.
In Subsection \ref{sec-MR}, the characterizations of both
$M(B^{0,b}_{1,\fz}(\rn))$ and $M(B^{0,b}_{\fz,\fz}(\rn))$ are completed.
Finally, the proofs with respect to the aforementioned examples are put in Subsection \ref{sec-eg}.


\subsection{Sufficient Conditions with Respect to $M(B^{0,b}_{p,\fz}(\rn))$ when
$p\in[1,\fz]$}\label{sec-suff-p}


In this subsection, we study the sufficient condition for  $M(B^{0,b}_{p,\fz}(\rn))$.
For this, we will make use of the decomposition of
$fg$ into $\iif(f,g)$, $\iis(f,g)$, and $\iit(f,g)$.
Thus, we will concentrate on the estimates of
$\|\iif(f,g)\|_{B^{0,b}_{p,\fz}(\rn)}$, $\|\iis(f,g)\|_{B^{0,b}_{p,\fz}(\rn)}$, and
$\|\iit(f,g)\|_{B^{0,b}_{p,\fz}(\rn)}$, respectively.
By Lemma \ref{lem-M(B)}, we may always start with $f\in L^\infty (\rn)$.

\begin{lemma}\label{lem-suff-I1}
Let $p,q\in[1,\fz]$ and $s,b\in\rr$.
Then there exists a positive constant $C$ such that,
for any $f\in L^{\fz}(\rn)$ and $g\in B^{s,b}_{p,q}(\rn)$,
\begin{equation}\label{eq-I1-suff}
\lf\|\iif(f,g)\r\|_{B^{s,b}_{p,q}(\rn)}
\le C\lf\|f\r\|_{L^{\fz}(\rn)}\lf\|g\r\|_{B^{s,b}_{p,q}(\rn)},
\end{equation}
where $\iif(f,g)$ is in \eqref{eq-decompose}.
\end{lemma}

\begin{proof}
Applying \eqref{eq-supp-4} and Lemma \ref{lem-I1-1} with $\sigma=1$ and
$u_k := (S^{k-2}f)S_kg$, we find that
\begin{align*}
\lf\|\iif(f,g)\r\|_{B^{s,b}_{p,q}(\rn)}&\ls
\lf\{\sum_{k\in\zz_+}2^{ksq}(1+k)^{bq}\lf\|\lf(S^{k-2}f\r)S_kg\r\|^q_{L^p(\rn)}\r\}^{\f1q}\\
&\ls\sup_{k\in\zz_+}\lf\|S^{k}f\r\|_{L^\fz(\rn)}\lf\|g\r\|_{B^{s,b}_{p,q}(\rn)},
\end{align*}
which, together with the trivial estimate
$\sup_{k\in\zz_+}\|S^{k}f\|_{L^\fz(\rn)}\ls \lf\|f\r\|_{L^\fz(\rn)}$,
implies the desired inequality \eqref{eq-I1-suff}.
This finishes the proof of Lemma \ref{lem-suff-I1}.
\end{proof}

Now, we turn to the estimate of $\iis(f,g)$.

\begin{lemma}\label{lem-suff-I2}
Let $b\in\rr$ and $p\in[1,\fz]$.
\begin{enumerate}
\item[\rm(i)] If  $p\in[1,\fz)$,
then there exists a positive constant $C$ such that,
for any $f\in L^{\fz}(\rn)$ and $g\in B^{0,b}_{p,\fz}(\rn)$,
\begin{align*}
\lf\|\iis(f,g)\r\|_{B^{0,b}_{p,\fz}(\rn)}
&\le C\lf\|g\r\|_{B^{0,b}_{p,\fz}(\rn)}
\sup_{l\in\zz_+}\sum_{k= l-3}^{\fz}\lf(\f{1+l}{1+k}\r)^b\\
&\quad\times\sup_{l(P)=2^{-l}}
\lf[\fint_P\lf|S_{k}f(y)\r|^{p'}\,dy\r]^{\f{1}{p'}}.
\end{align*}

\item[\rm(ii)] If $p=\fz$,
then there exists a positive constant $C$ such that,
for any $f\in L^{\fz}(\rn)$ and $g\in B^{0,b}_{\infty,\fz}(\rn)$,
\begin{align*}
\lf\|\iis(f,g)\r\|_{B^{0,b}_{\fz,\fz}(\rn)}
&\le C\lf\|g\r\|_{B^{0,b}_{\fz,\fz}(\rn)}
\sup_{l\in\zz_+}\sup_{l(P)=2^{-l}}\sum_{k=l-3}^{\fz}\lf(\f{1+l}{1+k}\r)^b\\
&\quad\times\fint_P\lf|S_{k}f(y)\r|\,dy.
\end{align*}
\end{enumerate}
\end{lemma}

\begin{proof}
Let  $l\in\zz_+$ be fixed.
From \eqref{eq-supp-2} and \eqref{eq-supp-6}, we deduce that
\begin{equation}\label{eq-SlI2}
S_l(\iis(f,g))=\sum_{i=-1}^{1}\sum_{k = l-3}^\infty S_l\lf(\lf[S_{k+i}f\r] S_kg\r).
\end{equation}

We first prove (i).
Observe that, for any positive integer $m$, there exists a positive constant $C_m$ such that,
for any $k\in\zz_+$ and  $x\in\rn$,
\begin{equation}\label{eq-|phik|<}
\lf|\varphi_k(x)\r|\le\f{ C_m 2^{kn}}{(1+2^k|x|)^m}.
\end{equation}

Let $Q\in\cq$ with $l(Q)=2^{-l}$ and let $x\in Q$.
Obviously we have  $1+2^l|x|\sim 1+2^l|x_Q|$,
where $x_Q$ denotes the lower-left corner of $Q$.
Now, from \eqref{eq-|phik|<}, the Minkowski inequality, and the H\"{o}lder inequality, we deduce that,
for any $i\in\{-1,0,1\}$ and $l\in\zz_+$,
\begin{align*}
&\lf\|\sum_{k = l-3}^\infty S_l\lf(\lf[S_{k+i}f\r] S_kg\r)\r\|_{L^p(\rn)}\\
&\quad\ls\lf\{\int_{\rn}\lf[\sum_{k = l-3}^\infty \int_{\rn}\f{2^{ln}}{(1+2^l|x-y|)^m}
\lf|S_{k+i}f(y)\r|\lf|S_kg(y)\r|\,dy\r]^p\,dx\r\}^{\f1p}\\
&\quad\ls\lf\{\int_{\rn}\lf[\sum_{k = l-3}^\infty
\lf\{\int_{\rn}\wz{f}_{k+i,l}(x,y)\,dy\r\}^{\f{1}{p'}}
\lf\{\int_{\rn}\wz{g}_{k,l}(x,y)\,dy\r\}^{\f1p}\r]^p\,dx\r\}^{\f1p}\\
&\quad\ls\sum_{k = l-3}^\infty \lf\{\int_{\rn}
\lf[\int_{\rn}\wz{f}_{k+i,l}(x,y)\,dy\r]^{\f{p}{p'}}
\lf[\int_{\rn}\wz{g}_{k,l}(x,y)\,dy\r]\,dx\r\}^{\f1p}\\
&\quad\ls\sum_{k = l-3}^\infty \sup_{x\in\rn}\lf\{\int_{\rn}
\wz{f}_{k+i,l}(x,y)\,dy\r\}^{\f{1}{p'}}
\lf\{\int_{\rn}\int_{\rn}\wz{g}_{k,l}(x,y)\,dy\,dx\r\}^{\f1p},
\end{align*}
where
$$\wz{f}_{k+i,l}(x,y):=\f{2^{ln}\lf|S_{k+i}f(y)\r|^{p'}}{(1+2^l|x-y|)^m}\
\mathrm{and}\ \wz{g}_{k,l}(x,y):=\f{2^{ln}\lf|S_kg(y)\r|^p}{(1+2^l|x-y|)^m}.$$
Next, we use that, if $x\in Q_{l,\mu}$ with $\mu\in\zz^n$, then
\begin{eqnarray*}
 \lf\{\int_{\rn}\wz{f}_{k+i,l}(x,y)\,dy\r\}^{\f{1}{p'}}  & = &
\lf\{\sum_{\nu \in \zz^n} \int_{Q_{l,\nu}}
\wz{f}_{k+i,l}(x,y)\,dy\r\}^{\f{1}{p'}}
\\
& \ls &  \lf\{\lf[\sum_{\nu\in\zz^n}\f{1}{(1+|\nu-\mu|)^m}\r]\sup_{\nu\in\zz^n}
\fint_{Q_{l,\nu}}\lf|S_{k+i}f(y)\r|^{p'}\,dy\r\}^{\f{1}{p'}}\\
& \ls &  \lf\{\sup_{\nu\in\zz^n}
\fint_{Q_{l,\nu}}\lf|S_{k+i}f(y)\r|^{p'}\,dy\r\}^{\f{1}{p'}},
\end{eqnarray*}
where the implicit positive constants are independent of $\mu$ and hence of $x \in \rn$.
We choose an $m>n$. Then, with
\begin{align*}
&\lf\{\int_{\rn}\int_{\rn}\wz{g}_{k,l}(x,y)\,dy\,dx\r\}^{\f1p}\\
&\quad=
\lf\{\int_{\rn}\f{2^{ln}}{(1+2^l|x|)^m}\,dx \int_{\rn}\lf|S_kg(y)\r|^p\,dy\r\}^{\f1p}
\ls \lf\|S_kg\r\|_{L^p(\rn)},
\end{align*}
we obtain
$$
\lf\|\sum_{k = l-3}^\infty S_l\lf(\lf[S_{k+i}f\r]S_kg\r)\r\|_{L^p(\rn)}
\ls \sum_{k=l-3}^\infty \lf\{\sup_{\nu\in\zz^n}
\fint_{Q_{l,\nu}}\lf|S_{k+i}f(y)\r|^{p'}\,dy\r\}^{\f{1}{p'}} \lf\|S_kg\r\|_{L^p(\rn)}.
$$
Thus, we conclude that
\begin{align*}
&\lf\|\iis(f,g)\r\|_{B^{0,b}_{p,\fz}(\rn)}\\
&\quad\ls\sup_{l\in\zz_+}(1+l)^b\sum_{i=-1}^1\lf\|\sum_{k= l-3}^\infty S_l\lf(\lf[S_{k+i}f\r] S_kg\r)\r\|_{L^p(\rn)}\\
&\quad\ls\lf\|g\r\|_{B^{0,b}_{p,\fz}(\rn)}\sum_{i=-1}^1\sup_{l\in\zz_+}\sum_{k = l-3}^\infty \lf(\f{1+l}{1+k}\r)^b\sup_{l(P)=2^{-l}}
\lf[\fint_P\lf|S_{k+i}f(y)\r|^{p'}\,dy\r]^{\f{1}{p'}}\\
&\quad\ls\lf\|g\r\|_{B^{0,b}_{p,\fz}(\rn)}\sup_{l\in\zz_+}\sum_{k = l-3}^\infty \lf(\f{1+l}{1+k}\r)^b\sup_{l(P)=2^{-l}}
\lf[\fint_P\lf|S_{k}f(y)\r|^{p'}\,dy\r]^{\f{1}{p'}}.
\end{align*}
This finishes the proof of (i).

Next, we prove (ii).
Using \eqref{eq-|phik|<} and the H\"{o}lder inequality again, we find that,
for any $i\in\{-1,0,1\}$, $l\in\zz_+$, and $x\in Q_{l,\mu}$ with $\mu\in\zz^n$,
\begin{align*}
&\lf|\sum_{k = l-3}^\infty S_l\lf(\lf[S_{k+i}f\r]S_kg\r)(x)\r|\\
&\quad\ls\sum_{k = l-3}^\infty \int_{\rn}\f{2^{ln}}{(1+2^l|x-y|)^m}
\lf|S_{k+i}f(y)\r|\lf|S_kg(y)\r|\,dy\\
&\quad\ls\sum_{\nu\in\zz^n}\fint_{Q_{l,\nu}}\sum_{k = l-3}^\infty \f{1}{(1+|\nu-\mu|)^m}
\lf|S_{k+i}f(y)\r|\|S_kg\|_{L^{\fz}(\rn)}\,dy
\\
&\quad\ls\|g\|_{B^{0,b}_{\fz,\fz}(\rn)}\sum_{\nu\in\zz^n}\f{1}{(1+|\nu-\mu|)^m}
\fint_{Q_{l,\nu}}\sum_{k = l-3}^\infty (1+k)^{-b}\lf|S_{k+i}f(y)\r|\,dy
\\
&\quad\ls\|g\|_{B^{0,b}_{\fz,\fz}(\rn)}\sup_{\nu\in\zz^n}
\fint_{Q_{l,\nu}}\sum_{k = l-3}^\infty (1+k)^{-b}\lf|S_{k+i}f(y)\r|\,dy
\end{align*}
because  $m>n$.
Here the implicit positive constants are independent of $x \in \rn$.
Combined with \eqref{eq-SlI2}, this  implies that
\begin{align*}
&\lf\|\iis(f,g)\r\|_{B^{0,b}_{\fz,\fz}(\rn)}\\
&\quad\ls\sup_{l\in\zz_+}(1+l)^b\sum_{i=-1}^1
\lf\|\sum_{k = l-3}^\infty S_l\lf(\lf[S_{k+i}f\r]S_kg\r)\r\|_{L^{\fz}(\rn)}
\\
&\quad\ls\lf\|g\r\|_{B^{0,b}_{\fz,\fz}(\rn)}\sum_{i=-1}^1\sup_{l\in\zz_+}\sup_{\nu\in\zz^n}
\fint_{Q_{l,\nu}}\sum_{k = l-3}^\infty \lf(\f{1+l}{1+k}\r)^b\lf|S_{k+i}f(y)\r|\,dy
\\
&\quad\ls\lf\|g\r\|_{B^{0,b}_{\fz,\fz}(\rn)}\sup_{l\in\zz_+}\sup_{l(P)=2^{-l}}
\fint_{P}\sum_{k = l-3}^\infty \lf(\f{1+l}{1+k}\r)^b\lf|S_{k}f(y)\r|\,dy.
\end{align*}
This finishes the proof of (ii) and hence Lemma \ref{lem-suff-I2}.
\end{proof}

Finally we deal with the estimate for $\iit(f,g)$.

\begin{lemma}\label{lem-suff-I3}
Let $p\in[1,\fz]$ and $b\in\rr$.
Then there exists a positive constant $C$ such that,
for any $f\in L^{\fz}(\rn)$ and $g\in B^{0,b}_{p,\fz}(\rn)$,
\begin{align*}
\lf\|\iit(f,g)\r\|_{B^{0,b}_{p,\fz}(\rn)}
&\le C\lf\|g\r\|_{B^{0,b}_{p,\fz}(\rn)}\sup_{k\ge 2}\sum_{l=0}^{k-2}\lf(\f{1+k}{1+l}\r)^{b}\\
&\quad\times\sup_{l(P)=2^{-l}}\lf\{\fint_P\lf|S_kf(y)\r|^p\,dy\r\}^{\f1p}.
\end{align*}
\end{lemma}

\begin{proof}
From \eqref{eq-supp-2} and \eqref{eq-supp-5}, we deduce that, for any $k\in\zz_+$,
\begin{equation*}
S_k(\iit(f,g))=\sum_{i=-2}^{1}S_k\lf(\lf[S_{k+i}f\r]S^{k+i-2}g\r).
\end{equation*}
For any given $l\in\zz_+$, let $S_l^{*}f$ be defined the same as in \eqref{eq-Peetre} with $a=2n$.
The definition of $S_l^{*}f$ implies that, for any $x\in \rn$,
$|S_lf(x)|\le S_l^{*}f(x)$ and, for any $l\in\zz_+$ and $\nu\in\zz^n$,
$$
\min_{z\in Q_{l,\nu}}S_l^*f(z)\ge
\min_{z\in Q_{l,\nu}}\max_{y\in Q_{l,\nu}}\f{|S_l f(y)|}{(1+2^l|z-y|)^{2n}}
\gtrsim\max_{y\in Q_{l,\nu}}S_lf(y)
$$
with  the implicit positive constants independent of $f$, $l$, and $\nu$.
From this and the Minkowski inequality, we deduce
that, for any $k\in\zz_+$,
\begin{align*}
\lf\|\lf(S_kf\r)S^{k-2}g\r\|_{L^p(\rn)}&\le\lf\{\int_{\rn}
\lf[\sum_{l=0}^{k-2}\lf|S_kf(x)\r|\lf|S_lg(x)\r|\r]^p\,dx\r\}^{\f1 p}\\
&\le\sum_{l=0}^{k-2}\lf[\int_{\rn}
\lf|S_kf(x)\r|^p\lf|S_lg(x)\r|^p\,dx\r]^{\f1 p}\\
&\le\sum_{l=0}^{k-2}\lf[\sum_{\nu\in\zz^n}\max_{z\in Q_{l,\nu}}
\lf|S_lg(z)\r|^p\int_{Q_{l,\nu}}\lf|S_kf(y)\r|^p\,dy\r]^{\f1 p}\\
&\ls \sum_{l=0}^{k-2}\lf[\sum_{\nu\in\zz^n}\fint_{Q_{l,\nu}}
\lf|S^*_lg(x)\r|^p\,dx\int_{Q_{l,\nu}}\lf|S_kf(y)\r|^p\,dy\r]^{\f1 p}\\
&\ls \sum_{l=0}^{k-2}\lf[\int_{\rn}\lf|S^*_lg(x)\r|^p\,dx\r]^{\f1 p}\sup_{\nu\in\zz^n}
\lf[\fint_{Q_{l,\nu}}\lf|S_kf(y)\r|^p\,dy\r]^{\f1 p},
\end{align*}
which, together with Lemma \ref{lem-peeter-Lp}, implies that, for any $k\in\zz_+$,
\begin{align*}
&\lf\|\lf(S_kf\r)S^{k-2}g\r\|_{L^p(\rn)} \\
&\quad\ls
\sup_{l\in\zz_+}(1+l)^b\lf\|S_lg\r\|_{L^p(\rn)}\sum_{l=0}^{k-2}(1+l)^{-b}\sup_{\nu\in\zz^n}
\lf[\fint_{Q_{l,\nu}}\lf|S_kf(y)\r|^p\,dy\r]^{\f1 p}
\\
&\quad\sim  \lf\|g\r\|_{B^{0,b}_{p,\fz}(\rn)}\sum_{l=0}^{k-2}(1+l)^{-b}\sup_{\nu\in\zz^n}
\lf[\fint_{Q_{l,\nu}}\lf|S_kf(y)\r|^p\,dy\r]^{\f1 p}.
\end{align*}
By this and the fact that, for any $h\in L^p(\rn)$ and $k\in\zz_+$,
$\|S_k h\|_{L^p(\rn)} \ls \|h\|_{L^p(\rn)}$,
we conclude that
\begin{align*}
\lf\|\iit(f,g)\r\|_{B^{0,b}_{p,\fz}(\rn)}&\ls
\sum_{i=-2}^{1}\sup_{j\in\zz_+}(1+j)^b \lf\|\lf(S_{j+i}f\r)S^{j+i-2}g\r\|_{L^p(\rn)}\\
&\ls\lf\|g\r\|_{B^{0,b}_{p,\fz}(\rn)}\sup_{k\ge 2}\sum_{l=0}^{k-2}\lf(\f{1+k}{1+l}\r)^{b}
\sup_{l(P)=2^{-l}}\lf\{\fint_P\lf|S_kf(y)\r|^p\,dy\r\}^{\f1 p},
\end{align*}
which completes the proof of Lemma \ref{lem-suff-I3}.
\end{proof}

Checking the proof of Lemma \ref{lem-suff-I3}, it becomes clear that the early application
of the Minkowski inequality leads to a possible loss, at least when
$p\in(1,\infty)$.
Thus, we will deal with some variants in what follows.
In a first lemma,  we consider a very simple consequence of the following chain of inequalities.
Let $v\le \min (p,q)$ and $\max (p,q) \le u\le \fz$. Then, for any sequence $\{f_j\}_{j\in\zz_+}$ of measurable functions,
one has
\begin{equation}\label{chain}
\lf[\sum_{j=0}^\infty \|f_j \|^u_{L^p(\rn)} \r]^{\f1u}\le \lf\|
\lf(\sum_{j=0}^\infty |f_j|^q \r)^{\f1q}\r\|_{L^p(\rn)}
\le
\lf[\sum_{j=0}^\infty \|f_j \|^v_{L^p(\rn)} \r]^{\f1v}.
\end{equation}

\begin{lemma}\label{lem-suff-4}
Let $p\in [2,\fz)$, $b\in\rr$,
$$
\alpha(b):= \begin{cases}
\ b &\ \mbox{if}\  b\in[\f12,\fz),\\
\ \f12 &\  \mbox{if}\  b\in(-\fz,\f12),
\end{cases}
$$
and
$$
\beta (b) : = \begin{cases}
\ 0 &\  \mbox{if} \  b\in(\f12,\fz),
\\
\ \f12 &\  \mbox{if} \  b=\f12,
\\
\ 0&\  \mbox{if} \  b\in(-\fz,\f12).
\end{cases}
$$
Then there exists a positive constant $C$ such that,
for any $f\in L^{\fz}(\rn)$ and $g\in B^{0,b}_{p,\fz}(\rn)$,
\begin{equation*}
\lf\|\iit(f,g)\r\|_{B^{0,b}_{p,\fz}(\rn)}
\le C\lf\|g\r\|_{B^{0,b}_{p,\fz}(\rn)}  \sup_{j \in\nn}  (1+j)^{\alpha (b)}
\lf[\ln (1+j)\r]^{\beta (b)} \| S_j f\|_{L^\infty (\rn)}.
\end{equation*}
\end{lemma}

\begin{proof}
Applying the Littlewood--Paley characterization of $L^p (\rn)$,
we find that, for any integer $j\ge0$,
\begin{align*}
\lf\|\lf(S_j f\r)S^{j-2} g\r\|_{L^p (\rn)}
& \le  \lf\|S_j f \r\|_{L^\infty (\rn)}\lf\| S^{j-2} g\r\|_{L^p (\rn)}\\
& \ls  \lf\|S_j f \r\|_{L^\infty (\rn)}
\lf\|\lf(\sum_{l=0}^{j-2} |S_l g|^2\r)^{\f1 2}\r\|_{L^p (\rn)}\\
&\ls  \lf\|S_j f \r\|_{L^\infty (\rn)}  \lf[\sum_{l=0}^{j-2} \|S_l g\|_{L^p(\rn)}^2\r]^{\f1 2}\\
&\ls  \lf\|S_j f \r\|_{L^\infty (\rn)} \| g\|_{B^{0,b}_{p,\infty}(\rn)}
\lf[\sum_{l=0}^{j-2} (1+l)^{-2b} \r]^{\f1 2},
\end{align*}
where we used \eqref{chain} with $v= \min(p,2) = 2$ in the second step.
This implies that
\begin{align*}
&\sup_{j \ge 2} (1+j)^b
 \lf\|\lf(S_j f\r)S^{j-2} g\r\|_{L^p (\rn)}\\
&\quad\ls \sup_{j \ge 2}  (1+j)^b
 \lf[\sum_{l=0}^{j-2} (1+l)^{-2b} \r]^{\f1 2} \lf\|S_j f \r\|_{L^\infty (\rn)} \| g\|_{B^{0,b}_{p,\infty}(\rn)},
\end{align*}
which completes the proof of Lemma \ref{lem-suff-4}.
\end{proof}

The next estimate uses duality.

\begin{lemma}\label{lem-suff-5}
Let $p\in (1,2]$, $b\in\rr$,
$$
\gamma (b):=  \begin{cases}
\ b &\  \mbox{if}\  b \in[\f1p,\fz)\\
\ \f1p &\  \mbox{if}\  b\in(-\fz,\f1p),
\end{cases}
$$
and
$$
\delta (b) : = \begin{cases}
\ 0 &\  \mbox{if} \  b\in(\f1p,\fz),
\\
\ \f1p &\  \mbox{if} \  b=\f1p,
\\
\ 0 &\  \mbox{if} \  b\in(-\fz,\f1p).
 \end{cases}
$$
Then there exists a positive constant $C$ such that,
for any $f\in L^{\fz}(\rn)$ and $g\in B^{0,b}_{p,\fz}(\rn)$,
\begin{equation*}
\lf\|\iit(f,g)\r\|_{B^{0,b}_{p,\fz}(\rn)}
\le C \lf\|g\r\|_{B^{0,b}_{p,\fz}(\rn)} \sup_{j \in\nn} (1+j)^{\gamma (b)}
\lf[\ln (1+j)\r]^{\delta (b)}\| S_j f\|_{L^\infty (\rn)}.
\end{equation*}
\end{lemma}

\begin{proof}
We will use the duality of $L^p(\rn)$ and $L^{p'} (\rn)$,
where $p'$ denotes the conjugate index of $p$.
Let $\wz{\phi}$ be a function in $\cs(\rn)$ such that
$\wz{\phi} = 1$ on $\{x\in\rn:\ 1/4 \le |x|\le 4\}$ and
$$\supp \wz{\phi} \subset \{x \in \rn:\  1/8 \le |x| \le 8\}.$$
For any $j\ge0$,
define $\wz{S}_j f := \cf^{-1} [\wz{\phi}(2^{-j+1} \cdot )\cf f]$.
Then
\begin{eqnarray}\label{ws-00}
\lf\|\lf(S_j f\r)S^{j-2}g \r\|_{L^p(\rn)}
& = & \sup_{\{h_j \in \cs (\rn):\ \|h_j\|_{L_{p'}(\rn)}\le 1\}}
\int_{\rn} S_j f (x)S^{j-2}g (x) \overline{h_j}(x)\, dx
\\
& = & \sup_{\{h_j \in \cs (\rn):\ \|h_j\|_{L^{p'}(\rn)}\le 1\}}
\int_{\rn} S_j f (x)  S^{j-2}g(x) \wz{S}_j \overline{h_j}(x)\, dx
\nonumber
\\
&=& \sup_{\{h_j \in \cs(\rn):\ \|h_j\|_{L^{p'}(\rn)}\le 1\}}\,
\Big|S^{j-2}g(x) \lf(S_j f(x)  \wz{S}_j \overline{h_j}(x)\r)\Big| .\nonumber
\end{eqnarray}

Next, we use $(B_{p',1}^{0,-b}(\rn))' = B^{0,b}_{p,\fz}(\rn)$ (see Lemma \ref{lem-duality}).
As in Step 1 of the proof of \cite[Theorem 2.11.2]{Tr83}, we have, for any $j\in\zz_+$,
\begin{equation}\label{ws-02}
 \lf|S^{j-2}g \lf(\lf[S_j f\r]\wz{S}_j \overline{h_j}\r)\r|
 \ls  \lf\|S^{j-2}g\r\|_{B^{0,b}_{p,\fz}(\rn)}\,
 \lf\|\lf(S_j f\r)\wz{S}_j \overline{h_j}\r\|_{B_{p',1}^{0,-b}(\rn)}.
\end{equation}
By
$\supp \cf ([S_j f]\wz{S}_j \overline{h_j}) \subset \{\xi\in\rn:\ |\xi|\le
3\cdot2^{j-1}+ 2^{j+2}\}$, we find that
\begin{align*}
 &\lf\|\lf(S_j f\r)\wz{S}_j \overline{h_j}\r\|_{B_{p',1}^{0,-b}(\rn)}\\
 &\quad=
 \sum_{l=0}^{j+4}  (1+l)^{-b} \lf\|S_l\lf(\lf[S_j f\r] \wz{S}_j \overline{h_j}\r)
 \r\|_{L^{p'}(\rn)}
\\
&\quad \le   \lf[\sum_{l=0}^{j+4}  (1+l)^{-bp}\r]^{\f1p} \lf\{
\sum_{l=0}^{j+4} \lf\|S_l\lf(\lf[S_j f\r] \wz{S}_j \overline{h_j}\r)
\r\|^{p'}_{L^{p'}(\rn)}\r\}^{\f{1}{p'}}
\\
&\quad \le   \lf[\sum_{l=0}^{j+4}  (1+l)^{-bp}\r]^{\f1p}  \lf\{
\int_{\rn} \sum_{l=0}^{j+4} \lf |S_l\lf(\lf[S_j f\r]\wz{S}_j \overline{h_j}\r) (x)\r|^{p'} dx \r\}^{\f{1}{p'}},
\end{align*}
which, together with the embedding $\ell_2 \hookrightarrow \ell_{p'}$, $p \in (1,2]$,
implies that
\begin{eqnarray*}
\lf\|\lf(S_j f\r)\wz{S}_j \overline{h_j}\r\|_{B_{p',1}^{0,-b}(\rn)} & \ls &
 \lf[\sum_{l=0}^{j+4}  (1+l)^{-bp}\r]^{\f1p}\lf\{
\int_{\rn} \lf[\sum_{l=0}^{j+4}  \lf|S_l\lf(\lf[S_j f\r]\wz{S}_j
\overline{h_j}\r) (x)\r|^{2}\r]^{\f{p'}{2}} dx \r\}^{\f{1}{p'}}
\\
&\ls & \lf[\sum_{l=0}^{j+4}  (1+l)^{-bp}\r]^{\f1p}  \lf\|S^{j+4}\lf(\lf[S_j f\r]\wz{S}_j \overline{h_j}\r)
\r\|_{L^{p'}(\rn)},
\end{eqnarray*}
where, in the last step, we have used the Littlewood-Paley characterization of $L^{p'}(\rn)$.
Because of
\[
\sup_{\ell=0,1,\ldots } \sup_{\|\varphi\|_{L^{p'}(\rn)} \le 1}\|S^\ell \varphi\|_{L^{p'}(\rn)} <\infty,
\]
we find
\begin{eqnarray*}
\lf\|\lf(S_j f\r)\wz{S}_j \overline{h_j}\r\|_{B_{p',1}^{0,-b}(\rn)}
&\ls &  \lf[\sum_{l=0}^{j+4}  (1+l)^{-bp}\r]^{\f1p} \lf\| S_j f \r\|_{L^\infty (\rn)}\lf\| \wz{S}_j \overline{h_j} \r\|_{L^{p'}(\rn)}
\\
&\ls & \lf[\sum_{l=0}^{j+4} (1+l)^{-bp}\r]^{\f1p} \lf\|S_j f \r\|_{L^\infty (\rn)}\lf\| h_j \r\|_{L^{p'}(\rn)}.
\end{eqnarray*}
Inserting this estimate into \eqref{ws-02} and  taking care of \eqref{ws-00}, we conclude that
\begin{align*}
\lf\|\iit(f,g)\r\|_{B^{0,b}_{p,\fz}(\rn)} &\ls
\sup_{j\ge2} (1+j)^b \lf\|\lf(S_j f\r)S^{j-2}g \r\|_{L_p(\rn)}\\
&\ls \|g\|_{B^{0,b}_{p,\infty}(\rn)}
\lf\{\sup_{j\ge2} (1+j)^b \lf[\sum_{l=0}^{j+4}  (1+l)^{-bp}\r]^{\f1p}
\lf\|S_j f \r\|_{L^\infty (\rn)}\r\} .
\end{align*}
Notice that
$$
 (1+j)^b\lf[\sum_{l=0}^{j+4}  (1+l)^{-bp}\r]^{\f1p} \ls
 \begin{cases}
\ (1+j)^b &\  \mbox{if} \  b\in(\f1p,\fz),
\\
\ (1+j)^b  \lf[\ln (2+j)\r]^{\f1p} &\  \mbox{if} \  b=\f1p,
\\
\ (1+j)^{\f1p} &\  \mbox{if} \  b \in(-\fz,\f1p).
\end{cases}
$$
This finishes the proof of Lemma \ref{lem-suff-5}.
\end{proof}

\begin{remark}
We summarize our estimates of $\lf\|\iit(f,g)\r\|_{B^{0,b}_{p,\fz}(\rn)}$ for $b=0$.
We shall concentrate on those estimates using
 $\|S_j f\|_{L^\infty(\rn)}$.
 As consequences of Lemmas \ref{lem-suff-I2} (with $p=1$ and $p=\infty$), \ref{lem-suff-4},
 and \ref{lem-suff-5}, we obtain
$$
 \lf\|\iit(f,g)\r\|_{B^{0}_{p,\fz}(\rn)}\ls\lf\| g\r\|_{B^{0}_{p,\fz}(\rn)}
 \begin{cases}
\ \dsup_{j \ge 2} j\| S_j f\|_{L^\infty (\rn)} &\  \mbox{if} \  p=1,\\
\   \dsup_{j \ge 2} j^{\f1p}\| S_j f\|_{L^\infty (\rn)}
&\ \mbox{if}\  p\in(1,2], \\
 \  \dsup_{j \ge 2} j^{\f12} \| S_j f\|_{L^\infty (\rn)}
&\  \mbox{if}\  p\in[2,\fz),\\
\ \dsup_{j \ge 2} j\| S_j f\|_{L^\infty (\rn)} &\ \mbox{if} \  p=\infty.
   \end{cases}
$$
These estimates show a jump at infinity as many others in harmonic analysis.
However, only in the case $p=\infty$, these estimates are unimprovable.
For the case $p\in(1,\infty)$, the estimates given in Lemma \ref{lem-suff-I2}
are not sharp as well, but they are not comparable with those in
Lemmas \ref{lem-suff-4} and \ref{lem-suff-5}.
Finally, in the case $p=1$, Lemma \ref{lem-suff-I2} is the best possible.
\end{remark}


\subsection{Necessary Conditions for $M(B^{0,b}_{p,\fz}(\rn))$ when $p\in[1,\fz]$}\label{sec-nece-p}


Next, we study necessary conditions for $f$ to belong to
$M(B^{0,b}_{p,\fz}(\rn))$. As in the case of the sufficient conditions,
there are two different types of those restrictions.

\begin{theorem}\label{lem-nece-I2-0}
Let $b\in\rr$ and $p\in[1,\fz]$.
Then there exists a positive constant $C$ such that, for any $f\in M(B^{0,b}_{p,\fz}(\rn))$,
\begin{align*}
&\sup_{l\in\zz_+}\sup_{\{\{P_j\}_{j=l}^{\fz}:\ P_j\in\cq,\ l(P_j)=2^{-l}\}}
\lf[\int_{\rn}2^{ln}\lf\{\sum_{j= l}^{\fz}\lf(\f{1+l}{1+j}\r)^{b}
\lf[\fint_{P_j}\lf|S_jf(z)\r|^{p'}\,dz\r]^{\f{1}{p'}}
\mathbf{1}_{P_j}(x)\r\}^p\,dx\r]^{\f1p}\\
&\quad\le C\lf\|f\r\|_{M(B^{0,b}_{p,\fz}(\rn))},
\end{align*}
where $1/p+1/p'=1$ and the second supremum is taken over all the possible sequences of dyadic cubes with edge length $2^{-l}$.
\end{theorem}

\begin{proof}
Let $f\in M(B^{0,b}_{p,\fz}(\rn))$. Then, by Lemma \ref{lem-M(B)}(ii), we know $f\in L^\fz(\rn)$.
Let $\varphi_1$ be the radial Schwartz function in \eqref{eq-S_k}.
We can find a suitable positive constant $\lambda$, a fixed positive integer $\sigma$
depending only on both $\varphi_1$ and $\lambda$, and
a $\nu_0\in\zz^n$ with $|\nu_0|\in(2^{\sigma},3\cdot2^{\sigma})$ such that there exists a dyadic cube
$Q_{\sigma,\nu_0}:=2^{-\sigma}(\nu_0+[0,1)^n)$ satisfying that
\begin{equation}\label{eq-phi1>lambda}
\lf[2^{-\sigma}(\nu_0\pm[0,1)^n)\r]
\subset\lf\{x:=(x',x_n)\in\rn:\ x'\in\rr^{n-1},\ x_n\ge0,\ \varphi_1(x)\ge \lambda\r\}.
\end{equation}
Define
$
\eta:=\mathbf{1}_{Q_{\sigma,\nu_0}}
$
and, for any $k\in\zz_+$, $\eta_k(\cdot):=\eta(2^k\cdot)$.
Then $\supp \eta_k=2^{-k-\sigma}(\nu_0+[0,1)^n)$ for any $k\in\zz_+$.

 For any given $k\in\zz_+$, $j\ge k$, and any given dyadic cube
 $Q_k^{(j)}:=Q_{k+\sigma,\nu_{k,j}}\in\cq_{k+\sigma}$
[with its edge length $l(Q_k^{(j)})=2^{-k-\sigma}$ but its position depending on $j$],
define $g_k$ by setting, for any $x\in\rn$,
\begin{align}\label{eq-construct-g}
g_k(x):=&\sum_{j= k+N}^{\fz}(1+j)^{-b}\lf\|S_jf\r\|^{1-p'}_{L^{p'}(\wz{Q}_k^{(j)})}\\
&\quad \times S_j\lf(\eta_k(\cdot-x_{Q_k^{(j)}})\sgn(S_jf)\lf|S_jf\r|^{p'-1}\r)(x)\nonumber,
\end{align}
where $1/p+1/p'=1$, $N$ is a sufficiently large positive integer which will be chosen later,
$x_{Q_k^{(j)}}$ denotes the lower-left corner of the dyadic cube $Q_k^{(j)}$, and
$$
\wz{Q}_k^{(j)}:=\lf\{x_{Q_k^{(j)}}+z:\ z\in \supp \eta_k\r\}.
$$

We first claim that $\{g_k\}_{k\in\zz_+}$ is uniformly bounded in $B^{0,b}_{p,\fz}(\rn)$
 although the definition itself
depends on $k$.
Indeed, by the H\"{o}lder inequality and $\supp \eta_k(\cdot-x_{Q_k^{(j)}})=\wz{Q}_k^{(j)}$,
we have, for any $k,l\in \zz_+$,
\begin{align*}
&(1+l)^b\lf\|S_l g_k\r\|_{L^p(\rn)}\\
&\quad\ls\sup_{l-1\le j\le l+1}\lf(\f{1+l}{1+j}\r)^{b}\lf\|S_jf\r\|^{1-p'}_{L^{p'}(\wz{Q}_k^{(j)})}\\
&\qquad\times\lf\|S_j\lf(\eta_k(\cdot-x_{Q_k^{(j)}})\sgn(S_jf)\lf|S_jf\r|^{p'-1}\r)\r\|_{L^p(\rn)}\\
&\quad\ls\sup_{l-1\le j\le l+1}
\|\varphi_1\|_{L^1(\rn)}\|\eta\|_{L^\fz(\rn)}\lf\|S_jf\r\|^{1-p'}_{L^{p'}(\wz{Q}_k^{(j)})}\\
&\qquad\times
\lf\|\lf(S_jf\r)^{p'-1}\mathbf{1}_{\supp \eta_k(\cdot-x_{Q_k^{(j)}})}\r\|_{L^{p}(\rn)}\\
&\quad\ls\sup_{l-1\le j\le l+1}\|\eta\|_{L^\fz(\rn)}\lf\|S_jf\r\|^{1-p'}_{L^{p'}(\wz{Q}_k^{(j)})}
\lf\|S_jf\r\|_{L^{p'}(\wz{Q}_k^{(j)})}^{p'/p}\\
&\quad\ls\|\eta\|_{L^\fz(\rn)}
\end{align*}
with the implicit positive constants independent of both $k$ and $l$.
Thus,  we conclude that, for any $k\in \zz_+$,
$$
\|g_k\|_{B^{0,b}_{p,\fz}(\rn)}\ls\|\eta\|_{L^\fz(\rn)},
$$
which proves the above claim.

In the following, for any $k\in \zz_+$ and $j\ge k+N$,
for simplicity, we write
$$
L_{k,j}:=\lf\|S_jf\r\|^{1-p'}_{L^{p'}(\wz{Q}_k^{(j)})}
$$
and, for any $x\in\rn$,
$$
A_{k,j}(x):=\int_{\rn} S_j\lf(\varphi_k(x-\cdot)f\r)(z)\eta_k(z-x_{Q_k^{(j)}})
\sgn(S_jf)(z)\lf|S_jf(z)\r|^{p'-1}\,dz
$$
and
\begin{align*}
B_{k,j}(x):&=\int_{\rn} S_jf(z)\varphi_k(x-z)\eta_k(z-x_{Q_k^{(j)}})\sgn(S_jf)(z)\lf|S_jf(z)\r|^{p'-1}\,dz\\
&=\int_{\wz{Q}_k^{(j)}}\varphi_k(x-z)\eta_k(z-x_{Q_k^{(j)}})\lf|S_jf(z)\r|^{p'}\,dz.
\end{align*}
By this and the definition of $g_k$, we have
\begin{align}\label{eq-|Sk(fgk)|Lp(Q)}
&\lf\|S_k(fg_k)\r\|_{L^p(\rn)}\\
&\quad=\lf\{\int_{\rn} \lf|\sum_{j= k+N}^{\fz}(1+j)^{-b}L_{k,j}\int_{\rn}\varphi_k(x-y)f(y)\r.\r.\nonumber\\
&\quad\quad\times\lf.S_j\lf(\eta_k(\cdot-x_{Q_k^{(j)}})\sgn(S_jf)
\lf|S_jf\r|^{p'-1}\r)(y)\,dy\Bigg|^p\,dx\r\}^{\f1p}\nonumber\\
&\quad=\lf\{\int_{\rn} \lf|\sum_{j= k+N}^{\fz}
(1+j)^{-b}L_{k,j}\int_{\rn}\varphi_j(y-z)\varphi_k(x-y)f(y)\,dy\r.\r.\nonumber\\
&\quad\quad\times\lf.\int_{\rn}\eta_k(z-x_{Q_k^{(j)}})\sgn(S_jf)(z)
\lf|S_jf(z)\r|^{p'-1}\,dz\Bigg|^p\,dx\r\}^{\f1p}\nonumber\\
&\quad=\lf[\int_{\rn} \lf|\sum_{j= k+N}^{\fz}(1+j)^{-b}
L_{k,j}A_{k,j}(x)\r|^p\,dx\r]^{\f1p}\nonumber.
\end{align}

Next, we hope to study the last quantity in \eqref{eq-|Sk(fgk)|Lp(Q)} with $A_{k,j}(x)$ replaced by $B_{k,j}(x)$.
For this purpose, we first estimate $L_{k,j}\|A_{k,j}-B_{k,j}\|_{L^p(Q_k^{(j)})}$ for any $k\in \zz_+$ and $j\ge k+N$.
By the Minkowski inequality with $p\ge1$, the H\"{o}lder inequality,
the differential mean value theorem, and the definition of $L_{k,j}$,
we find that, for any $k\in \zz_+$ and $j\ge k+N$,
\begin{align}\label{eq-|Aj-Bj|p}
&L_{k,j}\lf\|A_{k,j}-B_{k,j}\r\|_{L^p(Q_k^{(j)})}\\
&\quad\le L_{k,j}\lf\{\int_{Q_k^{(j)}}\lf[\int_{\rn}\int_{\rn}
\lf|\varphi_k(y)-\varphi_k(z)\r|\lf|\varphi_j(y-z)\r|\lf|f(x-y)\r|\r.\r.\nonumber\\
&\quad\quad\times\lf.\lf|\eta_k(x-x_{Q_k^{(j)}}-z)\r|\lf|S_jf(x-z)\r|^{p'-1}
\,dy\,dz\Bigg]^p\,dx\r\}^{\f1p}\nonumber\\
&\quad\le L_{k,j}\lf\|f\r\|_{L^{\fz}(\rn)}\int_{|z|\le 2^{-k-\sigma}(|\nu_0|+\sqrt{n})}\int_{\rn}
\lf|\varphi_k(y)-\varphi_k(z)\r|\lf|\varphi_j(y-z)\r|\,dy\,dz\nonumber\\
&\quad\quad\times\lf[\int_{\rn}\lf|\eta_k(x-x_{Q_k^{(j)}})\r|^p\lf|S_jf(x)\r|^{p(p'-1)}\,dx\r]^{\f1p}\nonumber\\
&\quad\ls L_{k,j}\lf\|f\r\|_{L^{\fz}(\rn)}\lf\|S_jf\r\|_{L^{p'}(\wz{Q}_k^{(j)})}^{\f{p'}{p}}\nonumber\\
&\quad\quad\times\int_{|z|\le 2^{-k-\sigma}(|\nu_0|+\sqrt{n})}
\int_{\rn}\f{2^{jn}|\varphi_k(y)-\varphi_k(z)|}{(1+2^j|y-z|)^M}\,dy\,dz\nonumber\\
&\quad\ls L_{k,j}\lf\|f\r\|_{L^{\fz}(\rn)}\lf\|S_jf\r\|_{L^{p'}(\wz{Q}_k^{(j)})}^{\f{p'}{p}}\nonumber\\
&\quad\quad\times\int_{|z|\le 2^{-\sigma}(|\nu_0|+\sqrt{n})}
\int_{\rn}\f{|\varphi_1(z+2^{k-j}y)-\varphi_1(z)|}{(1+|y|)^M}\,dy\,dz\nonumber\\
&\quad\ls L_{k,j}\lf\|f\r\|_{L^{\fz}(\rn)}\lf\|S_jf\r\|_{L^{p'}(\wz{Q}_k^{(j)})}^{\f{p'}{p}}
2^{k-j}\lf\|\nabla\varphi_1\r\|_{L^{\fz}(\rn)}\int_{\rn}\f{|y|}{(1+|y|)^M}\,dy\nonumber\\
&\quad\ls2^{k-j}L_{k,j}\lf\|f\r\|_{L^{\fz}(\rn)}\lf\|S_jf\r\|_{L^{p'}(\wz{Q}_k^{(j)})}^{\f{p'}{p}}
\sim2^{k-j}\lf\|f\r\|_{L^{\fz}(\rn)}\nonumber,
\end{align}
where $M$ is a fixed sufficiently large positive integer and the implicit positive constants depend only on $n,\ \sigma,\ \nu_0$, and $M$.
Here, we also used the fact that
$\supp \eta_k(\cdot-x_{Q_k^{(j)}})=x_{Q_k^{(j)}}+\supp \eta_k$, and,
if $x\in {Q_k^{(j)}}$ and $x-z\in x_{Q_k^{(j)}}+\supp \eta_k$, then $|z|\le 2^{-k-\sigma}(|\nu_0|+\sqrt{n})$.
Thus, by both the elementary  inequality that, for any $p\in[1,\fz]$,
\[
\Big|\| F \|_{L^p(\rn)} - \| G \|_{L^p(\rn)}\Big| \le \| F-G \|_{L^p(\rn)}
\]
valid for all $F,G \in L^p(\rn)$ and \eqref{eq-|Aj-Bj|p}, we conclude that
\begin{align*}
&\lf|(1+k)^b\lf\|\sum_{j= k+N}^{\fz}(1+j)^{-b}L_{k,j}A_{k,j}\mathbf{1}_{Q_k^{(j)}}\r\|_{L^p(\rn)}\r.\\
&\qquad-\lf.
(1+k)^b\lf\|\sum_{j= k+N}^{\fz}(1+j)^{-b}L_{k}B_{k,j}\mathbf{1}_{Q_k^{(j)}}\r\|_{L^p(\rn)}\r|\\
&\quad\le(1+k)^b\lf\|\sum_{j= k+N}^{\fz}(1+j)^{-b}L_{k,j}(A_{k,j}-B_{k,j})\mathbf{1}_{Q_k^{(j)}}\r\|_{L^p(\rn)}\\
&\quad\le\sum_{j= k+N}^{\fz}\lf(\f{1+k}{1+j}\r)^b L_{k,j}\lf\|A_{k,j}-B_{k,j}\r\|_{L^p(Q_k^{(j)})}\\
&\quad\ls\sum_{j= k+N}^{\fz}\lf(\f{1+k}{1+j}\r)^b 2^{k-j}\lf\|f\r\|_{L^{\fz}(\rn)}
\ls\lf(\f{1+k}{1+k+N}\r)^b 2^{-N}\lf\|f\r\|_{L^{\fz}(\rn)}.
\end{align*}
We denote by $C_1$ the implicit positive constant in the last quantity, which
depends only on $n,\ p,\ b,$ and $M$.
Using this and choosing $N$ big enough such that $C_1(1+k)^b(1+k+N)^{-b}2^{-N}<1$, we further obtain
\begin{align*}
&(1+k)^b\lf\|\sum_{j= k+N}^{\fz}(1+j)^{-b}L_{k,j}A_{k,j}\mathbf{1}_{Q_k^{(j)}}\r\|_{L^p(\rn)}\\
&\quad\ge(1+k)^b\lf\|\sum_{j= k+N}^{\fz}(1+j)^{-b}L_{k,j}B_{k,j}\mathbf{1}_{Q_k^{(j)}}\r\|_{L^p(\rn)}-
\lf\|f\r\|_{L^{\fz}(\rn)},\nonumber
\end{align*}
which, combined with \eqref{eq-|Sk(fgk)|Lp(Q)}, further implies that
\begin{align}\label{eq-|Skfgk|p+|f|fz}
&(1+k)^b\lf\|S_k(fg_k)\r\|_{L^p(\rn)}+\lf\|f\r\|_{L^{\fz}(\rn)}\\
&\quad\ge(1+k)^b\lf\|\sum_{j= k+N}^{\fz}(1+j)^{-b}L_{k,j}B_{k}\mathbf{1}_{Q_k^{(j)}}\r\|_{L^p(\rn)}\nonumber.
\end{align}
Notice that, by \eqref{eq-phi1>lambda}, we have, for any $x\in Q_k^{(j)}$ and
$z\in x_{Q_k^{(j)}}+\supp \eta_k=:\wz{Q}_k^{(j)}$,
\begin{equation*}
x-z\in 2^{-k-\sigma}(\nu_0\pm[0,1)^n)\subset\lf\{x\in\rn:\ \varphi_k(x)\ge 2^{kn}\lambda\r\}.
\end{equation*}
Thus, from this and the definition of $L_{k,j}$, we deduce that
\begin{align}\label{eq-|sumCjBj|p}
&\lf\|\sum_{j= k+N}^{\fz}(1+j)^{-b}L_{k,j}B_{k,j}\mathbf{1}_{Q_k^{(j)}}\r\|_{L^p(\rn)}\\
&\quad=\lf\{\int_{\rn}\lf|\sum_{j= k+N}^{\fz}(1+j)^{-b}L_{k,j}\int_{\wz{Q}_k^{(j)}}
\varphi_k(x-z)\lf|S_jf(z)\r|^{p'}\,dz\r|^p\r.\nonumber\\
&\qquad\times\mathbf{1}_{Q_k^{(j)}}(x)\,dx\Bigg\}^{\f1p}\nonumber\\
&\quad\ge\lambda\lf[\int_{\rn}2^{kn}\lf\{\sum_{j= k+N}^{\fz}(1+j)^{-b}
\lf[\fint_{\wz{Q}_k^{(j)}}\lf|S_jf(z)\r|^{p'}\,dz\r]^{\f{1}{p'}}\nonumber\r.\r.\\
&\qquad\times\lf.\mathbf{1}_{Q_k^{(j)}}(x)\Bigg\}^p\,dx\r]^{\f1p}\nonumber\\
&\quad=\lambda\lf[\int_{\rn}2^{kn}\bigg\{\sum_{j= k+N}^{\fz}(1+j)^{-b}
\lf[\fint_{\wz{Q}_k^{(j)}}\lf|S_jf(z)\r|^{p'}\,dz\r]^{\f{1}{p'}}\r.\nonumber\\
&\qquad\times\lf.\mathbf{1}_{\wz{Q}_k^{(j)}}(x)\Bigg\}^p\,dx\r]^{\f1p}\nonumber,
\end{align}
where, in the last quantity, we used the translation and the fact that
$$\wz{Q}_k^{(j)}:=x_{Q_k^{(j)}}+\supp \eta_k =Q_k^{(j)}+2^{-k-\sigma}\nu_0.$$
Thus, combining this, \eqref{eq-|Skfgk|p+|f|fz}, \eqref{eq-|sumCjBj|p},
and the arbitrariness of $Q_k^{(j)}$ in the above argument,
we conclude that, for any $k\in\zz_+$,
\begin{align}\label{eq-|fgk|B0,b+|f|}
&\lf\|fg_k\r\|_{B^{0,b}_{p,\fz}(\rn)}+\lf\|f\r\|_{L^{\fz}(\rn)}\\
&\quad\ge(1+k)^b\sup_{\{\{P_j\}_{j=k+N}^{\fz}:\ P_j\in\cq,\,l(P_j)=2^{-k-\sigma}\}}\lf\|\sum_{j= k+N}^{\fz}
(1+j)^{-b}L_{k,j}B_{k,j}\mathbf{1}_{P_j}\r\|_{L^p(\rn)}\nonumber\\
&\quad\gtrsim\sup_{\{\{P_j\}_{j=k+N}^{\fz}:\ P_j\in\cq,\,l(P_j)=2^{-k-\sigma}\}}
\lf[\int_{\rn}2^{kn}\lf\{\sum_{j= k+N}^{\fz}\lf(\f{1+k}{1+j}\r)^{b}\r.\r.\nonumber\\
&\qquad\lf.\lf.\times
\lf[\fint_{P_j}\lf|S_jf(z)\r|^{p'}\,dz\r]^{\f{1}{p'}}
\mathbf{1}_{P_j}(x)\r\}^p\,dx\r]^{\f1p}\nonumber.
\end{align}
Finally, applying Lemma \ref{lem-M(B)}(ii) and the uniform boundedness of
$\{g_k\}_{k\in\zz_+}$ in $B^{0,b}_{p,\fz}(\rn)$ to \eqref{eq-|fgk|B0,b+|f|},
we obtain, for any $k\in\zz_+$,
\begin{align*}
\lf\|f\r\|_{M(B^{0,b}_{p,\fz}(\rn))}&\gtrsim\sup_{k\in\zz_+}
\sup_{\{\{P_j\}_{j=k+N}^{\fz}:\ P_j\in\cq,\,l(P_j)=2^{-k-\sigma}\}}\lf[\int_{\rn}2^{kn}\r.\\
&\quad\times\lf.\lf\{\sum_{j= k+N}^{\fz}\lf(\f{1+k}{1+j}\r)^{b}
\lf[\fint_{P_j}\lf|S_jf(z)\r|^{p'}\,dz\r]^{\f{1}{p'}}
\mathbf{1}_{P_j}(x)\r\}^p\,dx\r]^{\f1p}\\
&\gtrsim\sup_{l\in\zz_+}\sup_{\{\{P_j\}_{j=l}^{\fz}:\ P_j\in\cq,\,l(P_j)=2^{-l}\}}\lf[\int_{\rn}2^{ln}\r.\\
&\quad\times\lf.\lf\{\sum_{j= l}^{\fz}\lf(\f{1+l}{1+j}\r)^{b}
\lf[\fint_{P_j}\lf|S_jf(z)\r|^{p'}\,dz\r]^{\f{1}{p'}}
\mathbf{1}_{P_j}(x)\r\}^p\,dx\r]^{\f1p}
\end{align*}
with the implicit positive constants depending on $n,\ p,\ N$, and $\varphi_1$.
This finishes the proof of Theorem \ref{lem-nece-I2-0}.
\end{proof}

As a straight corollary of Theorem \ref{lem-nece-I2-0},
we obtain a much simpler result for $p=1$ and $p\in(1,\fz]$ in
the following theorem.

\begin{theorem}\label{lem-nece-I2}
Let $b\in\rr$ and $p\in[1,\fz]$.
Then there exists a positive constant $C$ such that, for any $f\in M(B^{0,b}_{p,\fz}(\rn))$,
\begin{enumerate}
\item[\rm(i)] when $p=1$,
\begin{align*}
\sup_{l\in\zz_+}\sum_{j= l}^{\fz}\lf(\f{1+l}{1+j}\r)^{b}
\lf\|S_jf\r\|_{L^{\fz}(\rn)}\le C\lf\|f\r\|_{M(B^{0,b}_{p,\fz}(\rn))};
\end{align*}

\item[\rm(ii)] when $p\in(1,\fz]$,
\begin{align*}
\sup_{l\in\zz_+}\sup_{\gfz{Q\in\cq}{l(Q)=2^{-l}}}\sum_{j= l}^{\fz}\lf(\f{1+l}{1+j}\r)^{b}
\lf[\fint_{Q}\lf|S_jf(z)\r|^{p'}\,dz\r]^{1/p'}\le C\lf\|f\r\|_{M(B^{0,b}_{p,\fz}(\rn))}.
\end{align*}
\end{enumerate}
\end{theorem}

We also have the following necessary conditions of pointwise multipliers for
$B^{0,b}_{p,\fz}(\rn)$, which are different from those in Theorem \ref{lem-nece-I2}.

\begin{theorem}\label{lem-nece-I3}
Let $b\in\rr$ and $p\in[1,\fz]$.
Then there exists a positive constant $C=C_{(n,p,b)}$ such that, for any $f\in M(B^{0,b}_{p,\fz}(\rn))$,
\begin{enumerate}
\item[\rm(i)] when $p\in[1,\fz)$,
\begin{align*}
\sup_{k\ge2}\lf\{\sum_{l=0}^{k-2}\lf(\f{1+k}{1+l}\r)^{bp}
\sup_{\gfz{P\in\cq}{l(P)=2^{-l}}}\fint_P\lf|S_kf(y)\r|^p\,dy\r\}^{1/p}
\le C\lf\|f\r\|_{M(B^{0,b}_{p,\fz}(\rn))};
\end{align*}

\item[\rm(ii)] when $p=\fz$,
\begin{align*}
\sup_{k\ge2}\sum_{l=0}^{k-2}\lf(\f{1+k}{1+l}\r)^{b}
\lf\|S_kf\r\|_{L^{\fz}(\rn)}
\le C\lf\|f\r\|_{M(B^{0,b}_{\fz,\fz}(\rn))}.
\end{align*}
\end{enumerate}
\end{theorem}

To prove Theorem \ref{lem-nece-I3}(i), we begin with the construction of an auxiliary function with
a bounded support.

Let $Q$ be the unit cube centered at $(3/8,\ldots,3/8)$ with edges parallel to axes,
$Q^+$ the dyadic cube with $l(Q^+)=2^{-2}$ and the left-lower corner at the origin,
and $Q^-$ the dyadic cube with $l(Q^-)=2^{-2}$ and the left-lower corner at $(1/2,\ldots,1/2)$.
Obviously, $Q^+\subsetneqq Q$ and $Q^-\subsetneqq Q$.
Furthermore, let $h\in C_{\mathrm{c}}^{\fz}(\rn)$ be such that $|h|\le1$, $\supp h\subset Q$,
\begin{align}\label{eq-def-h}
h(x):=\begin{cases}
\ 1,\ &x\in Q^+,\\
\ -1,\ &x\in Q^-,\\
\ 0,\ &x\in \rn\setminus Q,\\
\ \text{smooth},\ &\text{otherwise},
\end{cases}
\end{align}
and $\int_{\rn}h(x)\,dx=0$.

We also need several technical lemmas.
The following is a more detailed version of  \cite[(6.21)]{KS02}.
For the convenience of the reader, we give a complete proof.
\begin{lemma}\label{lem-|Sjhl|p}
Let $p\in[1,\fz]$, $b\in\rr$, and $h$ be the same as in \eqref{eq-def-h}.
For any $l\in\zz_+$, define
\begin{equation}\label{eq-def-hl}
h_l(\cdot):=h(2^{l-2}(\cdot-x_l)),
\end{equation}
where $x_l$ is any fixed point.
Then there exists a positive constant $C$ such that, for any $j,l\in\zz_+$ with $j<l$,
$$
\lf\|2^{\f{ln}{p}}(1+l)^{-b}S_jh_l\r\|_{L^p(\rn)}\le C2^{(j-l)(n-\f np +1)}(1+l)^{-b},
$$
and,  for any $j,l\in\zz_+$ with $j\ge l$,
$$
\lf\|2^{\f{ln}{p}}(1+l)^{-b}S_jh_l\r\|_{L^p(\rn)}\le C2^{l-j}(1+l)^{-b}.
$$
\end{lemma}

\begin{proof}
Obviously, for any $l\ge0$, $\|h_l\|_{L^{\fz}(\rn)}=1$ and $\int_{\rn}h_l(x)\,dx=0$.

We first let $j,l\in\zz_+$ with $j<l$.
Notice that, for any $x\in\rn$ and $z\in B(x,2^{j-l+2})$,
$$
|x|\le |z|+|z-x|\le|z|+2,
$$
which further implies that
$$
1+|x|\le\inf_{z\in B(x,2^{j-l})}3\lf(1+|z|\r).
$$
Using this, $\int_{\rn}h_l(y+x_l)\,dy=0$, and the Minkowski integral inequality, changing variables,
and applying $\supp h\subset Q$, $\|h_l\|_{L^{\fz}(\rn)}=1$, and $\varphi_1\in\cs(\rn)$,
 we find that
\begin{align*}
&\lf\|2^{\f{ln}{p}}(1+l)^{-b}S_jh_l\r\|_{L^p(\rn)}\\
&\quad=\lf\|2^{\f{ln}{p}}(1+l)^{-b}S_jh_l(\cdot+x_l)\r\|_{L^p(\rn)}\\
&\quad\le\lf\{\int_{\rn}\lf[2^{\f{ln}{p}}(1+l)^{-b}\int_{\rn}\lf|\varphi_j(x-y)-\varphi_j(x)\r|
\lf|h_l(y+x_l)\r|\,dy\r]^p\,dx\r\}^{\f1 p}\nonumber\\
&\quad\le 2^{\f{ln}{p}}(1+l)^{-b}\int_{\rn}\lf[\int_{\rn}\lf|\varphi_j(x-y)-\varphi_j(x)\r|^p
\lf|h_l(y+x_l)\r|^p\,dx\r]^{\f1 p}\,dy\nonumber\\
&\quad= 2^{\frac{(j-l)n}{p'}}(1+l)^{-b}\int_{\rn}\lf[\int_{\rn}\lf|\varphi_1(x-2^{j-l}y)-\varphi_1(x)\r|^p
\lf|h(2^{-2}y)\r|^p\,dx\r]^{\f1 p}\,dy\nonumber\\
&\quad\le 2^{\frac{(j-l)n}{p'}}(1+l)^{-b}\int_{4Q}\lf[\int_{\rn}\lf|\varphi_1(x-2^{j-l}y)-\varphi_1(x)\r|^p
\,dx\r]^{\f1 p}\,dy\nonumber\\
&\quad\le 2^{(j-l)(\frac n{p'}+1)}(1+l)^{-b}\int_{4Q}\lf[\int_{\rn}\sup_{z\in B(x,2^{j-l}|y|)}\lf|\nabla\varphi_1(z)\r|^p
\,dx\r]^{\f1 p}|y|\,dy\nonumber\\
&\quad\le 2^{(j-l)(\frac n{p'}+1)}(1+l)^{-b}\lf[\int_{\rn}\sup_{z\in B(x,2^{j-l+2})}\lf|\nabla\varphi_1(z)\r|^p
\,dx\r]^{\f1 p}\nonumber\\
&\quad\ls 2^{(j-l)(\f{n}{p'}+1)}(1+l)^{-b}\lf[\int_{\rn}\f{1}{(1+|x|)^{2n}}
\,dx\r]^{\f1 p}\lf\|(1+|\cdot|)^{\frac{2n}p}\lf|\nabla\varphi_1(\cdot)\r|\r\|_{L^{\fz}(\rn)}\nonumber\\
&\quad\ls 2^{(j-l)(\f{n}{p'}+1)}(1+l)^{-b},\nonumber
\end{align*}
where $1/p+1/p'=1$.

Now, we let $j,l\in\zz_+$ with $j\ge l$.
Observe that, for any $x,y\in\rn$ such that $B(x,2^{l-j-2}|y|)\cap B(\mathbf{0},\sqrt{n})\neq\emptyset$,
$$
|x|\le\inf_{z\in B(x,2^{l-j}|y|)\cap B(\mathbf{0},\sqrt{n})}\lf(|z|+|z-x|\r)\le \sqrt{n}+2^{l-j-2}|y|\le\sqrt{n}(1+|y|),
$$
which means that
$1+|x|\ls 1+|y|$.
From this, $\int_{\rn}\varphi_j(x)\,dx=0$, the Minkowski integral inequality, and
$\supp |\nabla h|\subset B(\mathbf{0},\sqrt{n})$,
we deduce that, for any $k\in\nn$,
 \begin{align*}
&\lf\|2^{\f{ln}{p}}(1+l)^{-b}S_jh_l\r\|_{L^p(\rn)}\\
&\quad=\lf\|2^{\f{ln}{p}}(1+l)^{-b}S_jh_l(\cdot+x_l)\r\|_{L^p(\rn)}\\
&\quad= 2^{\f{ln}{p}}(1+l)^{-b}\lf\|\int_{\rn}\varphi_j(y)h_l(\cdot+x_l-y)\,dy-
\int_{\rn}\varphi_j(y)h_l(\cdot+x_l)\,dy\r\|_{L^p(\rn)}\\
&\quad\le 2^{\f{ln}{p}}(1+l)^{-b}\lf\|\int_{\rn}\lf|\varphi_j(y)\r|\lf|h_l(\cdot+x_l-y)-
h_l(\cdot+x_l)\r|\,dy\r\|_{L^p(\rn)}\\
&\quad\le 2^{\f{ln}{p}}(1+l)^{-b}\int_{\rn}\lf[\int_{\rn}\lf|\varphi_j(y)\r|^p
\lf|h_l(x+x_l-y)-h_l(x+x_l)\r|^p\,dx\r]^{\f1 p}\,dy\\
&\quad= 2^{\f{ln}{p}}(1+l)^{-b}\int_{\rn}2^{jn}\lf|\varphi_1(2^jy)\r|
\lf[\int_{\rn}\lf|h\lf(2^{l-2}(x-y)\r)-h\lf(2^{l-2}x\r)\r|^p
\,dx\r]^{\f1 p}\,dy\\
&\quad\le (1+l)^{-b}\int_{\rn}\lf|\varphi_1(y)\r|\lf[\int_{\rn}\lf|h(x-2^{l-j-2}y)-h(x)\r|^p
\,dx\r]^{\f1 p}\,dy\\
&\quad\ls2^{l-j} (1+l)^{-b}\int_{\rn}\f{|y|}{(1+|y|)^M}\lf[\int_{\rn}
\sup_{z\in [B(x,2^{l-j-2}|y|)\cap B(\mathbf{0},\sqrt{n})]}\lf|\nabla h(z)\r|^p
\,dx\r]^{\f1 p}\,dy\\
&\quad\ls\lf\|\nabla h(z)\r\|_{L^{\fz}(\rn)} 2^{l-j} (1+l)^{-b}\int_{\rn}\f{|y|(1+|y|)^{2n}}{(1+|y|)^M}\lf[\int_{\rn}
\f{1}{(1+|x|)^{2np}}\,dx\r]^{\f1 p}\,dy\\
&\quad\ls 2^{l-j} (1+l)^{-b},
\end{align*}
where $M$ is a positive integer which is big enough.
This finishes the proof of Lemma \ref{lem-|Sjhl|p}.
\end{proof}

Based on Lemma \ref{lem-|Sjhl|p}, we have a further conclusion.

\begin{lemma}\label{lem-g-property}
Let $b\in\rr$ and $p\in[1,\fz]$. Then there exist a positive integer $m$ which depends only on $b,\ n,$ and $p$,
and a positive constant $C=C_{(m)}$ such that,
for any given $N\in\zz_+$, any given sequence $\{Q_{i,\nu_i}\}_{i=0}^{N}$ of cubes,
and any of its sub-sequences,
$$
E_{N,N_0}:=\lf\{Q_{lm+N_0,\nu_{lm+N_0}}:\  l\in \mathbb{Z}_+\  \mbox{with}\ 0\le lm+N_0\le N\r\}
$$
with $N_0\in\{0,\ldots,m-1\}$,
there exists a $g_{N,N_0}\in B^{0,b}_{p,\fz}(\rn)$ satisfying that
$$
\lf\|g_{N,N_0}\r\|_{B^{0,b}_{p,\fz}(\rn)}\le C,
$$
\begin{equation}\label{eq-sum-k>t-1|Skg|p}
\sum_{k=N+1}^{\fz}\lf\|S_kg_{N,N_0}\r\|_{L^p(\rn)}\le C(1+N)^{-b},
\end{equation}
and, for any $x\in\rn$,
\begin{align*}
\lf|g_{N,N_0}(x)\r|\le\sum_{l=0}^{\lfloor \f Nm\rfloor}2^{\f{(lm+N_0)n}{p}}(1+lm+N_0)^{-b}\mathbf{1}_{Q_{lm+N_0}^0}(x)
\end{align*}
and
\begin{align}\label{eq-1<g(x)<1}
\lf|g_{N,N_0}(x)\mathbf{1}_{Q_{N,N_0}}(x)\r|\ge
\f{1}{C}\sum_{l=0}^{\lfloor \f Nm\rfloor}2^{\f{(lm+N_0)n}{p}}(1+lm+N_0)^{-b}\mathbf{1}_{Q_{lm+N_0,\nu_{lm+N_0}}}(x),
\end{align}
where $\lfloor N/m \rfloor$ denotes the maximum integer not bigger than $N/m$,
$Q_{N,N_0}:=\bigcup_{Q\in E_{N,N_0}}Q$, and,
for any $l\in\zz_+$, $Q_{lm+N_0}^0$ denotes the cube with the center
at the left-lower corner of $Q_{lm+N_0,\nu_{lm+N_0}}$,
but $8$-times edge length of $Q_{lm+N_0,\nu_{lm+N_0}}$.
\end{lemma}

\begin{proof}
Let $C_1$, depending only on $n,\ p,$ and $b$, be the smallest positive constant such that, for any $L\in\nn$,
$$
\sum_{l=0}^{L}\f{2^{l\max(1,\f n p)}}{(1+l)^{b}}\le C_1\f{2^{L\max(1,\f n p)}}{(1+L)^{b}}
\ \text{and}\
\sum_{l=L}^{\fz}\f{2^{-l(n-\f n p+1)}}{(1+l)^{b}}\le C_1\f{2^{-L(n-\f n p+1)}}{(1+L)^{b}}
$$
This also implies that, for any $\tau\in\zz_+$ and $\kappa\in\nn$,
\begin{equation}\label{eq-def-C1-<}
\sum_{l=0}^{L}\f{2^{(l\kappa+\tau)\max(1,\f n p)}}{(1+l\kappa+\tau)^{b}}
\le\sum_{l=0}^{L\kappa+\tau}\f{2^{l\max(1,\f n p)}}{(1+l)^{b}}
\le C_1\f{2^{(L\kappa+\tau)\max(1,\f n p)}}{(1+L\kappa+\tau)^{b}}
\end{equation}
and
\begin{equation}\label{eq-def-C1->}
\sum_{l=L}^{\fz}\f{2^{-(l\kappa+\tau)(n-\f n p+1)}}{(1+l\kappa+\tau)^{b}}
\le\sum_{l=L\kappa+\tau}^{\fz}\f{2^{-l(n-\f n p+1)}}{(1+l)^{b}}
\le C_1\f{2^{-(L\kappa+\tau)(n-\f n p+1)}}{(1+L\kappa+\tau)^{b}}.
\end{equation}
We choose an $m_1$ as the  smallest positive constant such that
$2^{-m_1n/p}(1+2m_1)^b\le 1$ and an
$m_2$ the smallest positive constant such that $2^{-m_2n/p}\le 1/(2C_1),$
and define $m:=\max\{m_1,m_2\}.$

Notice that, for any given $N_0\in\{0,\ldots,m-1\}$, the sub-sequence
$E_{N_0}$ is just $\{Q_{lm+N_0,\nu_{lm+N_0}}\}_{l=0}^{\lfloor N/m\rfloor}$.
Define $g_{N,N_0}$ by setting, for any $x\in\rn$,
\begin{align}\label{eq-def-gt}
g_{N,N_0}(x):=\sum_{l=0}^{\lfloor \f N m \rfloor} i^{l}2^{\f{(lm+N_0)n}{p}}(1+lm+N_0)^{-b}h_{lm+N_0}(x),
\end{align}
where $i$ denotes the imaginary unit and, for any $l\in\{1,\ldots,\lfloor N/m\rfloor\}$,
$h_{lm+N_0}$ is the same as in \eqref{eq-def-hl} with $l$ and $x_l$ replaced,
respectively, by $lm+N_0$ and $x_{lm+N_0,\nu_{lm+N_0}}$ (the left-lower corner of $Q_{lm+N_0,\nu_{lm+N_0}}$).

We first prove that $\{g_{N,N_0}\}$ is uniformly bounded in $B^{0,b}_{p,\fz}(\rn)$.
Applying Lemma \ref{lem-|Sjhl|p} to each term of the sum in \eqref{eq-def-gt}
and using both \eqref{eq-def-C1-<} and \eqref{eq-def-C1->} with $\kappa=m$ and $\tau=N_0$, we find that
\begin{align*}
&\sup_{k\in\zz_+}\lf(1+k\r)^b\lf\|S_kg_{N,N_0}\r\|_{L^p(\rn)}\\
&\quad\le \sup_{k\in\zz_+}\lf(1+k\r)^b\sum_{l=0}^{\lfloor \f N m\rfloor}
 \lf\|2^{\f{(lm+N_0)n}{p}}(1+lm+N_0)^{-b}S_k h_{lm}\r\|_{L^p(\rn)}\\
&\quad\ls\sup_{k\in\zz_+}\lf(1+k\r)^b
\lf[\sum_{l=0}^{\lfloor \f{k-N_0}{m}\rfloor} \f{2^{lm+N_0-k}}{(1+lm+N_0)^{b}}
+\sum_{l=\lfloor \f{k-N_0}{m}\rfloor+1}^{\lfloor \f Nm\rfloor} \f{2^{(k-lm-N_0)
(\f n{p'}+1)}}{(1+lm+N_0)^{b}}\r]\\
&\quad\ls\sup_{k\in\zz_+}\lf(1+k\r)^b
\lf\{\f{2^{m\lfloor \f{k-N_0}{m}\rfloor+N_0-k}}{(1+m\lfloor \f{k-N_0}{m}\rfloor+N_0)^{b}}
+\f{2^{(k-m\lfloor \f{k-N_0}{m}\rfloor-m-N_0)(\f n{p'}+1)}}{(1+m\lfloor \f{k-N_0}{m}\rfloor+m)^{b}}\r\}\\
&\quad\ls1,
\end{align*}
where $1/p+1/p'=1$, the implicit positive constants depend on $m$ but are independent of $N$, $N_0$,
and the choice of $x_{lm+N_0,\nu_{lm+N_0}}$.
This proves that, for any $N\in\zz_+$ and $N_0\in\{0,\ldots,m-1\}$,
$\|g_{N,N_0}\|_{B^{0,b}_{p,\fz}(\rn)}\ls 1$.

Applying Lemma \ref{lem-|Sjhl|p} again, we conclude that, for any  $k\ge N+1$,
\begin{align*}
\lf\|S_kg_{N,N_0}\r\|_{L^p(\rn)}&\ls\sum_{l=0}^{\lfloor \f Nm\rfloor}
2^{lm+N_0-k}(1+lm+N_0)^{-b}\\
&\ls2^{N-k}(1+N)^{-b}\ls(1+N)^{-b},
\end{align*}
which proves \eqref{eq-sum-k>t-1|Skg|p}.

For any $l\in\zz_+$, let $Q^0_{lm+N_0}$ be the cube with the center $x_{lm+N_0,\nu_{lm+N_0}}$
and $l(Q^0_{lm+N_0})=2^{-lm-N_0+3}$. Obviously, we have $\supp h_{lm+N_0}\subset Q^0_{lm+N_0}$.
Thus, we conclude that, for any $x\in\rn$,
\begin{align*}
|g_{N,N_0}(x)|&\le\sum_{l=0}^{\lfloor \f Nm\rfloor} 2^{\f{(lm+N_0)n}{p}}(1+lm+N_0)^{-b}
\lf|h(2^{lm+N_0-2}[x-x_{lm+N_0,\nu_{lm+N_0}}])\r|\\
&\le\sum_{l=0}^{\lfloor \f Nm\rfloor} 2^{\f{(lm+N_0)n}{p}}(1+lm+N_0)^{-b}\mathbf{1}_{Q^0_{lm+N_0}}(x).
\end{align*}

Next, we prove \eqref{eq-1<g(x)<1}.
We first observe that,
for any $l\in\zz_+$, $-1\le h_{lm+N_0}\le1$, and,
for any $x\in Q_{lm+N_0,\nu_{lm+N_0}}$, $h_{lm+N_0}(x)=1$.
Thus, when $K=0$ or $K=1$, we find that,
for any $x\in Q_{Km,\nu_{Km}}$,
\begin{align*}
\lf|\sum_{l=0}^{K} i^{l}\f{2^{\f{(lm+N_0)n}{p}}}{(1+lm+N_0)^{b}}h_{lm+N_0}(x)\r|
&\ge\f{2^{\f{(Km+N_0)n}{p}}}{(1+Km+N_0)^{b}}\lf|h_{m}(x)\r|\\
&\ge \f{2^{\f{(Km+N_0)n}{p}}}{(1+Km+N_0)^{b}}.
\end{align*}
By the choice of $m$, we also notice that, for any $K\in\nn$ with $K\ge 2$,
\begin{align}\label{eq-N-sum}
&2^{\f{(Km+N_0)n}{p}}(1+Km+N_0)^{-b}-\sum_{l=0}^{K-2}2^{\f{(lm+N_0)n}{p}}(1+lm+N_0)^{-b}\\
&\quad\ge2^{\f{(Km+N_0)n}{p}}(1+Km+N_0)^{-b}-C_12^{\f{[(K-2)m+N_0]n}{p}}(1+(K-2)m+N_0)^{-b}\nonumber\\
&\quad\ge2^{\f{(Km+N_0)n}{p}}(1+Km+N_0)^{-b}
\lf\{1-C_12^{-\f{2mn}{p}}\lf[\f{1+Km+N_0}{1+(K-2)m+N_0}\r]^{b}\r\}\nonumber\\
&\quad\ge\f{1}{2}2^{\f{(Km+N_0)n}{p}}(1+Km+N_0)^{-b}\nonumber.
\end{align}
Thus, from both \eqref{eq-N-sum} and \eqref{eq-def-C1-<}, we deduce that,
for any positive integer $K$ with $K\in[3,\lfloor N/m\rfloor]$ and for any $x\in Q_{Km,\nu_{Km}}$,
\begin{align*}
&\lf|\sum_{l=0}^K i^{l}\f{2^{\f{(lm+N_0)n}{p}}}{(1+lm+N_0)^{b}}h_{lm+N_0}(x)\r|\\
&\quad\ge \f{2^{\f{(Km+N_0)n}{p}}}{(1+Km+N_0)^{b}}-
\sum_{l=0}^{K-2}\f{2^{\f{(lm+N_0)n}{p}}}{(1+lm+N_0)^{b}}\\
&\quad\ge\f{1}{2}\f{2^{\f{(Km+N_0)n}{p}}}{(1+Km+N_0)^{b}}
\ge\f{1}{2C_1}\sum_{l=0}^K \f{2^{\f{(lm+N_0)n}{p}}}{(1+lm+N_0)^{b}}.
\end{align*}
Let $Q_{N,N_0}:=\bigcup_{l=0}^{\lfloor N/m\rfloor}Q_{lm+N_0,\nu_{lm+N_0}}$.
Altogether, we find that, for any $x\in Q_{N,N_0}$, there exists a
$$K_0:=\max\lf\{K\in\lf\{0,\ldots,\lf\lfloor \f N m\r\rfloor\r\}:\ x\in Q_{Km,\nu_{Km}}\r\}$$
such that
\begin{align*}
\lf|g_{N,N_0}(x)\r|
\gtrsim\sum_{l=0}^{K_0} 2^{\f{(lm+N_0)n}{p}}(1+lm+N_0)^{-b}.
\end{align*}
Thus,
\begin{align*}
\lf|g_{N,N_0}(x)\mathbf{1}_{Q_{N,N_0}}(x)\r|
&\gtrsim\sum_{l=0}^{K_0} 2^{\f{(lm+N_0)n}{p}}(1+lm+N_0)^{-b}\mathbf{1}_{Q_{lm+N_0,\nu_{lm+N_0}}}(x)\\
&\sim\sum_{l=0}^{\lfloor \f N m\rfloor} 2^{\f{(lm+N_0)n}{p}}(1+lm+N_0)^{-b}\mathbf{1}_{Q_{lm+N_0,\nu_{lm+N_0}}}(x),
\end{align*}
which implies \eqref{eq-1<g(x)<1}.
Altogether, we conclude that $g_{N,N_0}$ is the desired function.
This finishes the proof of Lemma \ref{lem-g-property}.
\end{proof}

Now, we show Theorem \ref{lem-nece-I3}(i).

\begin{proof}[Proof of Theorem \ref{lem-nece-I3}(i)]
Let $Q$ be the unit cube centered at $(3/8,\ldots,3/8)$ with edges parallel to axes
and let $f\in M(B^{0,b}_{p,\fz}(\rn))$.
For any given integer $k\ge 2$ and any $l\in\{0,\ldots,k-2\}$, let $Q_{l,\nu_l}\subset Q$
be the dyadic cube such that
$$
\fint_{Q_{l,\nu_l}}\lf|S_kf(y)\r|^p\,dy\ge\f{1}{2}
\sup_{l(P)=2^{-l}}\fint_P\lf|S_kf(y)\r|^p\,dy
$$
and
$$
g:=\sum_{l=0}^{k-2}i^{l}2^{ln/p}(1+l)^{-b}h_{l}\,,
$$
where $h_l$ is the same as in \eqref{eq-def-hl} with $x_l:=x_{l,\nu_l}$ being the left-lower corner of $Q_{l,\nu_l}$.
By this and \eqref{eq-triangle}, we find that
\begin{align}\label{eq-neceproofi-1}
&\lf\{\sum_{l=0}^{k-2}\lf(\f{1+k}{1+l}\r)^{bp}
\sup_{l(P)=2^{-l}}\fint_P\lf|S_kf(y)\r|^p\,dy\r\}^{\f1p}\\
&\quad\le2^{\f1p}\lf\{\sum_{l=0}^{k-2}\lf(\f{1+k}{1+l}\r)^{bp}
\fint_{Q_{l,\nu_l}}\lf|S_kf(y)\r|^p\,dy\r\}^{\f1p}\nonumber\\
&\quad=2^{\f1p}(1+k)^b\lf\{\int_{\rn}\lf|S_kf(y)\r|^p\sum_{l=0}^{k-2}\f{2^{ln}}{(1+l)^{bp}}
\mathbf{1}_{Q_{l,\nu_l}}(y)\,dy\r\}^{\f1p}\nonumber\\
&\quad\le2^{\f1p}(1+k)^b\lf\{\int_{\rn}\lf|S_kf(y)\r|^p\lf[\sum_{l=0}^{k-2}\f{2^{\f{ln}{p}}}{(1+l)^{b}}
\mathbf{1}_{Q_{l,\nu_l}}(y)\r]^p\,dy\r\}^{\f1p}\nonumber.
\end{align}

Let $m$ be the same positive integer as in Lemma \ref{lem-g-property},
which depends only on $b$, $n$, and $p$.
Observe that, for any $l\in\{0,\ldots,k-2\}$, there exists a unique $k_l\in\zz_+$
such that $l=mk_l+\sigma$ for some $\sigma\in\{0,\ldots,m-1\}$.
Then we divide $\{Q_{l,\nu_l}\}_{l=0}^{k-2}$ into $m$ groups:
\begin{align*}
&W_0:=\lf\{Q_{l,\nu_l}:\ 0\le l\le k-2, \ l=mk_l\r\},\\
&\quad\vdots\\
&W_\sigma:=\lf\{Q_{l,\nu_l}:\ 0\le l\le k-2, \ l=mk_l+\sigma\r\},\\
&\quad\vdots\\
&W_{m-1}:=\lf\{Q_{l,\nu_l}:\ 0\le l\le k-2, \ l=mk_l+(m-1)\r\}.
\end{align*}
For each $\sigma\in\{0,\ldots,m-1\}$, let
$$
g_\sigma:=\sum_{l\in\cj_\sigma}i^{l}2^{ln/p}(1+l)^{-b}h_{l}\,,
$$
where $\cj_\sigma:=\lf\{l:\ Q_{l,\nu_l}\in W_\sigma\r\},$
and let
$
g:=\sum_{\sigma=0}^{m-1}g_\sigma.
$
Then, for any $\sigma\in\{0,\ldots,m-1\}$, applying Lemma \ref{lem-g-property}
with $N$, $N_0$, $E_{N,N_0}$, and $Q_{N,N_0}$ therein
replaced, respectively, by $k-2$,
$\sigma$, $W_\sigma$, and $\bigcup_{P\in W_\sigma}P$, we conclude that
there exists a positive constant $C_0$, independent of $k$, such that,
for any $\sigma\in\{0,\ldots,m-1\}$,
$\|g_\sigma\|_{B^{0,b}_{p,\fz}(\rn)}\le C_0$
and, for any $x\in\rn$,
\begin{align*}
\lf|g_\sigma(x)\mathbf{1}_{\bigcup_{P\in W_\sigma}P}(x)\r|&\ge
\f{1}{C_0}\sum_{l\in\cj_\sigma}2^{\f{ln}{p}}(1+l)^{-b}\mathbf{1}_{Q_{l,\nu_{l}}}(x).
\end{align*}
Applying this argument to \eqref{eq-neceproofi-1}, we further obtain
\begin{align}\label{eq-sum-Stfgi}
&\lf\{\sum_{l=0}^{k-2}\lf(\f{1+k}{1+l}\r)^{bp}
\sup_{l(P)=2^{-l}}\fint_P\lf|S_kf(y)\r|^p\,dy\r\}^{\f1p}\\
&\quad\ls \sum_{\sigma=0}^{m-1} (1+k)^b\lf\{\int_{\rn}\lf|S_kf(y)\r|^p
\lf[\sum_{l\in\cj_{\sigma}}\f{2^{\f{ln}{p}}}{(1+l)^{b}}
\mathbf{1}_{Q_{l,\nu_l}}(y)\r]^p\,dy\r\}^{\f1p}\nonumber\\
&\quad\ls\sum_{\sigma=0}^{m-1}(1+k)^b\lf\|(S_kf) g_\sigma\r\|_{L^p(\rn)}\nonumber.
\end{align}

For any $\sigma\in\{0,\ldots,m-1\}$, since $g_\sigma=\sum_{j=0}^{\fz}S_jg_\sigma$ in $L^p(\rn)$,
then, from the H\"{o}lder inequality, \eqref{eq-sum-k>t-1|Skg|p}
with $N+1$ therein replaced by $k-1$, and from Lemma \ref{lem-M(B)}(ii),
we deduce that, for the given integer $k\ge 2$,
\begin{align}\label{eq-Stfgi}
\lf\|\lf(S_kf\r)g_\sigma\r\|_{L^p(\rn)}
&\le\lf\|S_kf\r\|_{L^{\fz}(\rn)}\lf[\sum_{j= k-1}^{\fz}\lf\|S_jg_\sigma\r\|_{L^{p}(\rn)}\r]
+\lf\|\lf(S_kf\r)S^{k-2}g_\sigma\r\|_{L^p(\rn)}\\
&\ls \lf\|S_kf\r\|_{L^{\fz}(\rn)}(-1+k)^{-b}\nonumber\\
&\quad+(1+k)^{-b}\sum_{s=-1}^{1}(1+k)^{b}
\lf\|S_{N+s}\lf(\lf[S_kf\r] S^{k-2}g_\sigma\r)\r\|_{L^p(\rn)}\nonumber\\
&\ls (1+k)^{-b}\lf\|S_kf\r\|_{M(B^{0,b}_{p,\fz}(\rn))}+(1+k)^{-b}
\lf\|\lf(S_kf\r) S^{k-2}g_\sigma\r\|_{B^{0,b}_{p,\fz}(\rn)}\nonumber\\
&\ls (1+k)^{-b}\lf\|S_kf\r\|_{M(B^{0,b}_{p,\fz}(\rn))}
\lf(1+\lf\|S^{k-2}g_\sigma\r\|_{B^{0,b}_{p,\fz}(\rn)}\r)\nonumber.
\end{align}
Here, we also used the fact that, for any $k\ge 2$,
\begin{equation}\label{S+S+S}
\lf(S_kf\r) S^{k-2}g_\sigma=\lf(S_{k-1}+S_k+S_{k+1}\r)\lf(\lf[S_kf\r]S^{k-2}g_\sigma\r)
\end{equation}
which can be obtained by observing the support of their Fourier transforms.

Notice that, for any given $p\in[1,\fz)$ and any $k\in\zz_+$, $S_k$ is bounded on $L^p(\rn)$.
By this, we conclude that, for any $\sigma\in\{0,\ldots,m-1\}$,
\begin{align*}
\lf\|S^{k-2}g_\sigma\r\|_{B^{0,b}_{p,\fz}(\rn)}
&\ls\sum_{i=-1}^{1}\sup_{j\in\zz_+}(1+j)^b\lf\|S_jS_{j+i}g_\sigma\r\|_{L^p(\rn)}\\
&\ls\sup_{j\in\zz_+}(1+j)^b\lf\|S_jg_\sigma\r\|_{L^p(\rn)}
\sim\lf\|g_\sigma\r\|_{B^{0,b}_{p,\fz}(\rn)}.
\end{align*}
Thus, from this, \eqref{eq-sum-Stfgi}, \eqref{eq-Stfgi}, Lemma \ref{lem-M(B)}(iii),
and the proved estimate that $\|g_\sigma\|_{B^{0,b}_{p,\fz}(\rn)}\le C_0$ for any $\sigma\in\{0,\ldots,m-1\}$, we further deduce that
\begin{align*}
&\lf\{\sum_{l=0}^{k-2}\lf(\f{1+k}{1+l}\r)^{bp}
\sup_{l(P)=2^{-l}}\fint_P\lf|S_kf(y)\r|^p\,dy\r\}^{1/p}\\
&\quad\ls\sum_{\sigma=0}^{m-1}\lf\|S_kf\r\|_{M(B^{0,b}_{p,\fz}(\rn))}\lf(1+
\lf\|g_\sigma\r\|_{B^{0,b}_{p,\fz}(\rn)}\r)\\
&\quad\ls\sum_{\sigma=0}^{m-1}\lf\|\varphi_1\r\|_{L^1(\rn)}\lf\|f\r\|_{M(B^{0,b}_{p,\fz}(\rn))}
\ls m\lf\|f\r\|_{M(B^{0,b}_{p,\fz}(\rn))},
\end{align*}
which completes the proof of Theorem \ref{lem-nece-I3}(i).
\end{proof}

Next, we show Theorem \ref{lem-nece-I3}(ii).

\begin{proof}[Proof of Theorem \ref{lem-nece-I3}(ii)]
For any  $l\in\nn$, define $h_l$ by setting, for any $x:=(x_1,\ldots,x_n)\in\rn$,
$h_l(x):=e^{i2^lx_1}$, where $i$ denotes the imaginary unit.

Let $k\in\nn$ and $\varphi_k$ be the same as in \eqref{eq-S_k}.
We first claim that there exists a positive constant $C$ such that, for any $k\in\nn$,
\begin{equation}\label{eq-||Skhl||}
\lf\|S_k h_l\r\|_{L^{\fz}(\rn)}\le C2^{-|k-l|}.
\end{equation}

Indeed, on the one hand, when $l\le k$, we use both the vanishing property of $\varphi_k$ and the differential mean value theorem to find that, for any $x\in \rn$,
\begin{align*}
\lf|S_kh_l(x)\r|&\le\int_{\rn}\lf|\varphi_k(x-y)\r|\lf|e^{i2^lx_1}-e^{i2^ly_1}\r|\,dy\\
&\ls2^l\int_{\rn}\f{2^{kn}|x_1-y_1|}{(1+2^k|x-y|)^m}\,dy\ls2^{l-k},
\end{align*}
where we took $m>k+1.$
On the other hand, when $l>k$, we use the divergence theorem to conclude that, for any $x\in\rn$,
$$
\lf|S_kh_l(x)\r|\le2^{-l}\int_{\rn}\lf|\nabla\varphi_k(x-y)\r|\lf|e^{i2^ly_1}\r|\,dy
\ls2^{k-l},
$$
which proves the above claim.
Furthermore, by a similar argument, we know that \eqref{eq-||Skhl||} still holds true with $h_l$ replaced by any function
$\wz{h}_l(\cdot):=h_l(\cdot-z_0)$, where $z_0$ is any fixed point.

Next, for any given $k\in\zz_+$, let $z_k$ be the point such that $2S_kf(z_k)\ge\|S_kf\|_{L^{\fz}(\rn)}$
and
$$
g_k(\cdot):=\sum_{l=0}^{k}(1+l)^{-b}h_l(\cdot-z_k).
$$
Then, by the above claim, we find that $\{g_k\}_{k\in\zz_+}$ is uniformly bounded in $B^{0,b}_{\fz,\fz}(\rn)$
and
\begin{equation}\label{eq-sum||Sjgt||}
\sum_{j=k-1}^{\fz}\lf\|S_jg_k\r\|_{L^{\fz}(\rn)}\ls
\sum_{j= k-1}^{\fz}\sum_{l=0}^{k}(1+l)^{-b}2^{l-j}\ls(1+k)^{-b}.
\end{equation}
Thus, from the definition of $g_k$, an argument similar to that used in the estimations of \eqref{eq-Stfgi},
\eqref{eq-sum||Sjgt||}, and \eqref{S+S+S}, the uniform boundedness of $\{g_k\}_{k\in\zz_+}$ in $B^{0,b}_{\fz,\fz}(\rn)$,
and Lemma \ref{lem-M(B)}(iii), we deduce that, for any $k\in\zz_+$,
\begin{align*}
&\sum_{l=0}^{k-2}\lf(\f{1+k}{1+l}\r)^{b}\lf\|S_kf\r\|_{L^{\fz}(\rn)}\\
&\quad\ls\sum_{l=0}^{k-2}\lf(\f{1+k}{1+l}\r)^{b}S_kf(z_k)\\
&\quad\sim(1+k)^bS_kf(z_k)g_k(z_k)\ls(1+k)^b\lf\|\lf(S_kf\r)g_k\r\|_{L^{\fz}(\rn)}\\
&\quad\ls(1+k)^b\lf[\lf\|S_kf\r\|_{L^{\fz}(\rn)}\lf\{\sum_{j=k-1}^{\fz}\lf\|S_jg_k\r\|_{L^{\fz}(\rn)}\r\}
+\lf\|\lf(S_kf\r)S^{k-2}g_k\r\|_{L^{\fz}(\rn)}\r]\\
&\quad\ls\lf\|S_kf\r\|_{L^{\fz}(\rn)}
+\sum_{i=-1}^1(1+k)^b\lf\|S_{k+i}\lf(\lf[S_kf\r]S^{k-2}g_k\r)\r\|_{L^{\fz}(\rn)}\\
&\quad\ls\lf\|S_kf\r\|_{M(B^{0,b}_{\fz,\fz}(\rn))}
+\lf\|\lf(S_kf\r)S^{k-2}g_k\r\|_{B^{0,b}_{\fz,\fz}(\rn)}\\
&\quad\ls\lf\|S_kf\r\|_{M(B^{0,b}_{\fz,\fz}(\rn))}\lf[1+\lf\|S^{k-2}g_k\r\|_{B^{0,b}_{\fz,\fz}(\rn)}\r]\\
&\quad\ls\lf\|S_kf\r\|_{M(B^{0,b}_{\fz,\fz}(\rn))}
\lf[1+\sum_{i=-1}^{1}\sup_{j\in\zz_+}(1+j)^b\lf\|S_jS_{j+i}g_k\r\|_{L^{\fz}(\rn)}\r]\\
&\quad\ls\lf\|S_kf\r\|_{M(B^{0,b}_{\fz,\fz}(\rn))}\lf[1+\lf\|g_k\r\|_{B^{0,b}_{\fz,\fz}(\rn)}\r]\\
&\quad\ls\lf\|S_kf\r\|_{M(B^{0,b}_{\fz,\fz}(\rn))}\ls\lf\|f\r\|_{M(B^{0,b}_{\fz,\fz}(\rn))},
\end{align*}
which completes the proof of Theorem \ref{lem-nece-I3}(ii).
\end{proof}


\subsection{Proofs of Characterizations of Both
$M(B^{0,b}_{1,\fz}(\rn))$ and $M(B^{0,b}_{\fz,\fz}(\rn))$}\label{sec-MR}


 In this section, we present the proofs of characterizations of both
the spaces
$M(B^{0,b}_{1,\fz}(\rn))$ and $M(B^{0,b}_{\fz,\fz}(\rn))$,
respectively, in Subsections \ref{sec-subsub-p=1} and \ref{sec-subsub-p=fz}.


\subsubsection{Proofs of Theorem \ref{p=1} and Corollary \ref{approx}}\label{sec-subsub-p=1}


We begin with  the proof of the characterization for
$M(B^{0,b}_{1,\fz}(\rn))$.

\begin{proof}[Proof of Theorem \ref{p=1}]
Notice that, for any given $f\in L^{\fz}(\rn)$ and any $g\in B^{s,b}_{p,q}(\rn)$,
$fg$ has the decomposition same as in
\eqref{eq-decompose}.
Thus, by Lemmas \ref{lem-suff-I1}, \ref{lem-suff-I2}(i),
and \ref{lem-suff-I3} with $p=1$, we conclude that
$$
\| f\|_{M(B^{0,b}_{1,\fz}(\rn))}\ls\lf\|f\r\|_{L^{\fz}(\rn)} + \|f\|^{(1)}_{2,b} +\|f\|^{(1)}_{3,b}\,,
$$
which proves the sufficiency.

To prove the necessity, we apply
Lemma \ref{lem-M(B)}(ii) and Theorems \ref{lem-nece-I2}(i)
and \ref{lem-nece-I3}(i) with $p=1$,
and obtain
$$\|f\|_{L^{\fz}(\rn)} + \|f\|^{(1)}_{2,b} + \|f\|^{(1)}_{3,b}\ls\| f\|_{M(B^{0,b}_{1,\fz}(\rn))}\,.$$
This finishes the proof of the necessity and hence Theorem \ref{p=1}.
\end{proof}

\begin{proof}[Proof of Corollary \ref{approx}]
First, we deal with $\|f\|^{(1)}_{3,b}$.
Obviously, we have
\begin{align*}
\|f\|^{(1)}_{3,b}&\le\dsup_{k\ge2} \sum_{l=0}^{k-2}\lf(\f{1+k}{1+l}\r)^{b} \|S_kf\|_{L^\infty (\rn)}\\
&\ls
\left\{ \begin{array}{lll}
\dsup_{k\ge2}  (1+k)^b\|S_kf\|_{L^\infty (\rn)} &\qquad & \mbox{if}\  b\in(1,\fz),
 \\
\dsup_{k\ge2} (1+k) \ln (1+k)\, \|S_kf\|_{L^\infty (\rn)} &\qquad & \mbox{if}\  b=1,
 \\
\dsup_{k\ge2} (1+k)\|S_kf\|_{L^\infty (\rn)} &\qquad & \mbox{if}\  b\in(-\fz,1).
\end{array}\right.
\end{align*}
Next, we use the H\"{o}lder inequality and Lemma \ref{lem-sum-log}(i) to estimate $\|f\|^{(1)}_{2,b}$.
We consider several cases on $b$.\\
{\em Case 1)} Let $b\in(1,\fz)$. Then
we have
\begin{align}\label{eq-2,b,1,b>1}
\|f\|^{(1)}_{2,b}&\sim
\sup_{l\in\zz_+}(1+l)^b \sum_{k = l}^\infty \f{1}{(1+k)^{2b}}
(1+k)^b \lf\|S_kf\r\|_{L^{\fz}(\rn)}\\
&\ls\sup_{l\in\zz_+}(1+l)^{1-b} \sup_{k\in\zz_+} (1+k)^b\|S_kf\|_{L^{\fz}(\rn)}
\sim\|f\|_{B^{0,b}_{\fz,\fz}(\rn)}.\nonumber
\end{align}
\\
{\em Case 2)}
Let $b=1$.
Similarly, we find that
\begin{align*}
\|f\|^{(1)}_{2,b}&\sim \sup_{l\in\zz_+}(1+l)  \sum_{k = l}^\infty \f{1}{(1+k)^{2}\ln (1+k)}
(1+k)\ln (1+k) \lf\|S_kf\r\|_{L^{\fz}(\rn)}\\
&\ls \sup_{k\in\zz_+ } (1+k)\ln (1+k) \|S_kf\|_{L^{\fz}(\rn)}.
\end{align*}
\\
{\em Case 3)}
Let $b\in(0,1)$.
We have
\begin{align*}
\|f\|^{(1)}_{2,b}&\sim \sup_{l\in\zz_+}
(1+l)^b \sum_{k = l}^\infty \f{1}{(1+k)^{b+1}}
(1+k)\lf\|S_kf\r\|_{L^{\fz}(\rn)}\\
& \ls
\sup_{k\in\zz_+} (1+k)\|S_kf\|_{L^{\fz}(\rn)}.
\end{align*}
\\
{\em Case 4)}\
Let $b=0$. Obviously, $\|f\|^{(1)}_{2,b}\sim\|f\|_{B^{0}_{\fz,1}(\rn)}$.
\\
{\em Case 5)}
Let $b\in(-\fz,0)$ and $|b|+1< \alpha$. Then
\begin{align*}
\|f\|^{(1)}_{2,b}&\sim \sup_{l\in\zz_+}
(1+l)^{-|b|} \sum_{k = l}^\infty (1+k)^{|b|-\alpha}
(1+k)^\alpha\lf\|S_kf\r\|_{L^{\fz}(\rn)}\\
& \ls
\sup_{k\in\zz_+ } (1+k)^\alpha \|S_kf\|_{L^{\fz}(\rn)}.
\end{align*}

Notice that, when $b\in(1,\fz)$, we   have
$B_{\fz,\fz}^{0,b}(\rn)\hookrightarrow B_{\fz,1}^{0}(\rn)\hookrightarrow L^{\fz}(\rn)$.
This, together with Theorem \ref{p=1} and the above estimates of
both $\|f\|^{(1)}_{2,b}$ and $\|f\|^{(1)}_{3,b}$, then
finishes the proof of  Corollary \ref{approx}.
\end{proof}


\subsubsection{Proofs of Theorem \ref{p=infty} and Corollary \ref{bp=infty}}\label{sec-subsub-p=fz}


The proof  of the results about the characterization of
$M(B^{0,b}_{\fz,\fz}(\rn))$ is similar to those in Subsection \ref{sec-subsub-p=1},
and  we describe its details  as follows.

\begin{proof}[Proof of Theorem \ref{p=infty}]
Following the same idea as that used in the proof of Theorem \ref{p=1},
the sufficiency follows from Lemmas \ref{lem-suff-I1}, \ref{lem-suff-I2}(ii),
and \ref{lem-suff-I3},
and the necessity follows from Lemma \ref{lem-M(B)}(ii) and Theorems \ref{lem-nece-I2} and
\ref{lem-nece-I3}(ii).
This finishes the proof of Theorem \ref{p=infty}.
\end{proof}

\begin{proof}[Proof of Corollary \ref{bp=infty}]
By Theorem \ref{p=infty}, we know that, for any given $b\in(1,\fz)$,
$\|\cdot\|_{B^{0,b}_{\fz,\fz}(\rn))}\ls\|\cdot\|_{M(B^{0,b}_{\fz,\fz}(\rn))}$
and, for any $f\in B^{0,b}_{\fz,\fz}(\rn)$, $\|f\|^{(\fz)}_{3,b}=\|f\|_{B^{0,b}_{\fz,\fz}(\rn)}$.
From \eqref{eq-2,b,1,b>1}, we also infer that
$$
\|f\|^{(\fz)}_{2,b}\ls
\sup_{l\in\zz_+}(1+l)^b \sum_{k = l}^\infty \f{1}{(1+k)^{b}}
\lf\|S_kf\r\|_{L^{\fz}(\rn)}
\ls\|f\|_{B^{0,b}_{\fz,\fz}(\rn)}.
$$
This, combined with the embedding
$B_{\fz,\fz}^{0,b}(\rn)\hookrightarrow B_{\fz,1}^{0}(\rn)\hookrightarrow L^{\fz}(\rn)$, then
finishes the proof of Corollary \ref{bp=infty}.
\end{proof}


\subsection{Proofs Connected with the Examples}\label{sec-eg}


Now,  we give the proofs of the results  related to the three examples mentioned in Section 3.
We divide these into three subsections.
The proofs for characteristic functions of open sets,
continuous functions defined by differences,
and exponential functions are given, respectively,
in Subsections \ref{subsec-pf.eg.1}, \ref{subsec-pf.eg.2}, and \ref{subsec-pf.eg.3}.


\subsubsection{Proofs Related to Characteristic Functions of Open Sets\, --\, Theorems \ref{rn+}, \ref{rn+-}, \ref{rn++}, and \ref{thm-Q1:binR}}\label{subsec-pf.eg.1}

 We begin with the proof  of Theorem \ref{rn+}.

\begin{proof}[Proof of Theorem \ref{rn+}]
By Theorem \ref{p=1},
a necessary condition for a function $f$ belonging to the multiplier space $M(B^{0,b}_{1,\fz}(\rn))$ is
$$
\sum_{k = 0}^\infty \lf(\f{1}{1+k}\r)^b
\lf\|S_kf\r\|_{L^{\fz}(\rn)}
\le\|f\|^{(1)}_{2,b}
 < \infty\, .
$$
Notice that $b\le 0$ implies
\[
\sum_{k = 0}^\infty \lf(\f{1}{1+k}\r)^b
\lf\|S_kf\r\|_{L^{\fz}(\rn)}\ge \sum_{k = 0}^\infty
\lf\|S_kf\r\|_{L^{\fz}(\rn)} = \|f\|_{B^0_{\infty,1}(\rn)}\, .
\]
Since $B^0_{\infty,1}(\rn)$ contains bounded and uniformly continuous functions only
(see, for instance, \cite[Proposition 2.5.7]{Tr83}),
then a nontrivial characteristic function is not in $B^0_{\infty,1}(\rn)$ and hence
cannot be a pointwise multiplier for the space
$B^{0,b}_{1,\infty}(\rn)$.
This finishes the proof of Theorem \ref{rn+}.
\end{proof}

To prove Theorem \ref{rn+-}, we need the following conclusion.

 \begin{lemma}\label{hilfe}
 Let $b\in(0,\fz)$ and suppose $f\in L^\infty (\rn)$.
 If, for some $l_0 \in \nn_0$, there exists some constant $C\in(0,\fz)$ such that
 \begin{equation*}
 \inf_{k\ge l_0} \|S_k f\|_{L^\infty (\rn)} \ge C,
 \end{equation*}
then $f \notin M(B^{0,b}_{1,\infty}(\rn))$.
 \end{lemma}

 \begin{proof}
  By Theorem \ref{p=1}, it is enough to observe that
 \[
  \|f\|^{(1)}_{2,b} \ge C
  \sup_{l\ge l_0}  \sum_{k = l}^\infty \lf(\f{1+l}{1+k}\r)^b = \infty .
 \]
This is obvious by subdividing the considerations into
the cases $b\in(0,1]$ and $b\in(1,\fz)$ [see Lemma \ref{lem-sum-log}(i)],
which completes the proof of Lemma \ref{hilfe}.
 \end{proof}

\begin{proof}[Proof of Theorem \ref{rn+-}]
We begin with the proof of (i).

 By applying  Lemma \ref{hilfe}, to confirm $\mathbf{1}_{(-1,1)^n}\notin M(B_{1,\fz}^{0,b}(\rn))$,
 it suffices to show
 \begin{equation*}
 \inf_{k\ge l_0}\|S_k \mathbf{1}_{(-1,1)^n} \|_{L^\infty (\rn)} \ge C>0.
 \end{equation*}
 Our proof will be split into several steps.
 In the first step, we   deal with the one-dimensional situation.\\
 {\textbf{Step 1.}} Let $n=1$.
 For simplicity, we denote $\mathbf{1}_{(-1,1)}$ by $\mathbf{1}$.
 In what follows we shall need some properties of our smooth dyadic decomposition of
 unity $\{\phi_k\}_{k\in\zz_+}$ (see Subsection \ref{sec-HoTri}).
 Here, we suppose in addition that $\phi_0$ is nonincreasing on $[0, \infty)$.
 This, together with \eqref{eq-phi1}, implies $\phi_1\ge 0$ on $\rr$.

 Let $k \in\nn$.
 Thanks to the homogeneity property \eqref{eq-phik}, we find that, for any $x\in\rr$,
 \begin{align*}
  S_k  \mathbf{1} (x) & =   \f{1}{\sqrt{2\pi}}\int_{-1}^1\cf^{-1} \phi_k (x-y)\, dy\\
 &=   \f{2^{k-1}}{\sqrt{2\pi}} \int_{-1}^1 \cf^{-1} \phi_1 (2^{k-1}(x-y))\, dy\\
 & =   \f{1}{\sqrt{2\pi}}\int_{-2^{k-1}}^{2^{k-1}} \cf^{-1} \phi_1 (2^{k-1}x-z)\, dz\\
 &=  \f{1}{\sqrt{2\pi}}\int_{2^{k-1}x -2^{k-1}}^{2^{k-1}x+2^{k-1}} \cf^{-1} \phi_1 (u)\, du.
 \end{align*}
 Let $\delta \in(0,\fz)$ be a fixed positive constant which is chosen later.
By choosing $x:= 1+2^{-k+1}\delta $, we obtain
 \begin{equation}\label{new3}
 \| S_k  \mathbf{1} \|_{L^\infty (\rr)} \ge \frac 1{\sqrt{2\pi}}
 \lf|\int_{\delta}^{\delta+2^k}  \cf^{-1} \phi_1 (u)\, du \r|.
 \end{equation}
 Since the function $\phi_1$ is a Schwartz function, we infer that
 there exists a positive constant $C_1$ such that, for any $k \in \nn$,
$$
\frac 1{\sqrt{2\pi}}\int_{\{z \in \rr:\ |z| > 2^k\}} \lf| \cf^{-1} \phi_1 (u)\r|\, du  \le C_1 2^{-k}
$$
and, furthermore,
\begin{equation}\label{eq-star1}
\frac 1{\sqrt{2\pi}}
\lf| \int_{\delta}^{\delta+2^k}  \cf^{-1} \phi_1 (u)\, du \r|
\ge\frac 1{\sqrt{2\pi}} \lf| \int_{\delta}^\infty  \cf^{-1} \phi_1 (u)\, du \r|- C_1 2^{-k} .
\end{equation}
Notice that $\phi_0$ is even and hence $\cf^{-1}\phi_1$ is even as well.
This symmetry guarantees that
\begin{equation}\label{eq-star2}
 \int_{\{z \in \rr:\ |z| >\delta \}}  \cf^{-1} \phi_1 (u)\, du  =
 2 \int_{\delta}^\infty   \cf^{-1} \phi_1 (u)\, du .
\end{equation}
By using the fact that
\begin{equation}\label{eq-int=0}
\frac 1{\sqrt{2\pi}}\int_{-\infty}^\infty  \cf^{-1} \phi_1 (u)\, du = \phi_1 (0) =0
\end{equation}
[see \eqref{eq-phi0} and \eqref{eq-phi1}], we obviously have the identity
\begin{equation}\label{new5}
 \int_{\{z \in \rr:~ |z| >\delta \}}  \cf^{-1} \phi_1 (u)\, du = -
  \int_{-\delta}^\delta   \cf^{-1} \phi_1 (u)\, du.
\end{equation}
Inserting \eqref{eq-star1}, \eqref{eq-star2}, and \eqref{new5} into \eqref{new3}, we obtain
\begin{equation}\label{eq-star3}
  \lf\| S_k  \mathbf{1} \r\|_{L^\infty (\rr)} \ge\frac 1{2\sqrt{2\pi}}\lf|
 \int_{-\delta}^{\delta}  \cf^{-1} \phi_1 (u)\, du \r|- C_1 2^{-k}.
\end{equation}

Next, we prove that $\|S_k \mathbf{1} \|_{L^\infty (\rn)} \ge c>0$ for some suitable $k\in\nn$.
Since $\phi_1 $ is even,  nonnegative, and nontrivial, it follows that
$\cf^{-1} \phi_1$ is real-valued and
$$
\cf^{-1} \phi_1 (0) =\frac 1{\sqrt{2\pi}}\int_{-\infty}^\infty   \phi_1 (u)\, du =: c >0.
$$
Thus, the continuity of $\cf^{-1} \phi_1$ further implies that there exists a $\delta_0\in(0,\fz)$ such that,
for any $\delta\in(0,\delta_0]$,
\begin{equation}\label{eq-continuity}
 \lf|\int_{-\delta}^{\delta} \cf^{-1} \phi_1 (u)\, du \r| \ge c\delta.
\end{equation}
Thus, by this and \eqref{eq-star3}, we conclude that there exists a $k_0:=k_0(c,\delta)$ such that,
for any $k\ge k_0$,
$$
  \lf\| S_k  \mathbf{1} \r\|_{L^\infty (\rr)} \ge \frac 1{2\sqrt{2\pi}} c \delta
 - C_1  2^{-k} \ge C >0.
$$
In view of Lemma \ref{hilfe}, this proves the present theorem for the case $n=1$.
\\
{\textbf{Step 2.}} Let $n=2$.
This time we will work with a modified smooth dyadic resolution of unity.
The radial property of $\phi_0$ is by no means necessary, only convenient.
We proceed as follows.
First, let $\wz{\phi}_0\in\cs(\rr)$ be an even and real-valued function such that
$\wz{\phi}_0$ is nonincreasing on $[0,\infty)$ and
\[
0\leq \wz{\phi}_0\leq1,\ \wz{\phi}_0\equiv1\ \text{on}\ [-1,1],\ \text{and}\ \
\wz{\phi}_0\equiv0\ \  \text{on}\ \ \{x\in\rr:\ |x|\geq3/2\}.
\]
Then we define $\phi_0:\rr^2 \rar [0,\infty) $ by setting, for any $x:=(x_1,x_2)$,
\[
 \phi_0 (x):= \wz{\phi}_0(x_1) \wz{\phi}_0(x_2)\,.
\]
Based on this generator, we follow the usual procedure to define $\phi_k$ for any $k\in\nn$
[see \eqref{eq-phi1} and \eqref{eq-phik}], that is, for any $x\in\rr^2$,
\begin{equation*}
\phi_1(x):=\phi_0\lf(\f x2\r)-\phi_0(x)
\ \ \mbox{and} \ \
\phi_k(x):=\phi_1(2^{-k+1}x), \ \forall\, k \ge 2 \,.
\end{equation*}
This results in a smooth dyadic resolution of unity,
that is, $\sum_{k=0}^{\fz}\phi_k\equiv1$, which can serve
to define the norm in our logarithmic Besov spaces.
Notice that we have the elementary identity
\begin{align}\label{eq-iden-phi1}
\phi_1(x) &=\wz{\phi}_0\lf(\f{x_1}{2}\r)\wz{\phi}_0\lf(\f{x_2}{2}\r)-
\wz{\phi}_0(x_1) \wz{\phi}_0(x_2)
\\
&= \lf[\wz{\phi}_0\lf(\f{x_1}{2}\r) - \wz{\phi}_0(x_1)\r]\wz{\phi}_0\lf(\f{x_2}{2}\r)
+  \wz{\phi}_0(x_1) \lf[\wz{\phi}_0\lf(\f{x_2}{2}\r)- \wz{\phi}_0(x_2)\r]\nonumber\\
&=\wz{\phi}_1(x_1)\wz{\phi}_0\lf(\f{x_2}{2}\r)+\wz{\phi}_0(x_1)\wz{\phi}_1(x_2)\nonumber,
\end{align}
where $x:=(x_1,x_2)\in\rr^2$ and, for any $x_1\in\rr$,
\begin{align*}
\wz{\phi}_1(x_1):=\wz{\phi}_0\lf(\f{x_1}{2}\r) - \wz{\phi}_0(x_1).
\end{align*}
By scaling and \eqref{eq-iden-phi1}, we also have, for any $k\in\nn$,
\begin{align}\label{eq-iden-phik}
\phi_k(x) &=\phi_1(2^{-k+1}x)
=\wz{\phi}_1(2^{-k+1}x_1)\wz{\phi}_0(2^{-k}x_2)+\wz{\phi}_0(2^{-k+1}x_1)\wz{\phi}_1(2^{-k+1}x_2)\\
&=\wz{\phi}_k(x_1)\sum_{i=0}^k\wz{\phi}_i(x_2)+\sum_{i=0}^{k-1}\wz{\phi}_i(x_1)\wz{\phi}_k(x_2)\nonumber,
\end{align}
where, for any $t\in\rr$,
$$
\wz{\phi}_k(t):=\wz{\phi}_1(2^{-k+1}t).
$$
Let $\mathbf{1}$ be the characteristic function of the interval $[-1,1)$.
Then
\[\mathbf{1}_2 (x) := \mathbf{1}_{(-1,1)^2} (x)=\mathbf{1} (x_1) \mathbf{1}(x_2),\ \forall\, x:=(x_1,x_2) \in \rr^2.
 \]
By \eqref{eq-iden-phik} and an argument via the Fourier transform,
we conclude that, for any $k\in\zz$ and $x:=(x_1,x_2) \in \rr^2$,
\begin{equation}\label{new4}
 S_k  \mathbf{1}_2 (x) =  S_k  \mathbf{1}_1 (x_1)
 S^k \mathbf{1}_1 (x_2) +  S^{k-1}  \mathbf{1}_1 (x_1)   S_k  \mathbf{1}_1 (x_2).
\end{equation}
Here and thereafter, either the symbol $S^k$ or $S_k$, defined the same as in \eqref{eq-S_k} and \eqref{eq-S^k},
is used as an operator with a meaning depending on the dimension of the underlying space on which
those functions the operator is applied to are defined;
for instance, for any $k\in\nn$,
$S_k\mathbf{1}_1:=\wz{\phi}_k\ast\mathbf{1}_1$ and $S_k\mathbf{1}_2:=\phi_k\ast\mathbf{1}_2$.

At this point, we would like to use the results from Step 1.
However, observe that
\begin{equation}\label{new7}
\lim_{k \to \infty} S^k \mathbf{1}_1 (t) =
\begin{cases}
1\ & \mbox{if}\   t\in(-1,1),
\\
0 \  & \mbox{if}\  t\in(-\fz,-1)\cup(1,\fz).
\end{cases}
\end{equation}
This makes clear that the choice
$x_1:=1+ 2^{-k+1}\delta =:x_2$ (as done in Step 1) is useless
and hence we have to modify the arguments from Step 1.

We first go back to the case $n=1$. Let $\{\phi_k\}_{k\in\zz_+}$ and $\mathbf{1}$ be the same as in Step 1.
This time, we let $\delta\in(0,1)$ be a fixed positive constant and choose $x:= 1- 2^{-k+1}\delta$.
Obviously,  $x\in(0,1)$ if $k\ge 2$.
Similarly to the proof of step 1,
We know that, for any $k\in\nn$,
 \begin{equation}\label{new6}
 \| S_k  \mathbf{1} \|_{L^\infty (\rr)} \ge \lf|S_k  \mathbf{1} ( 1- 2^{-k+1}\delta) \r|
 =   \f{1}{\sqrt{2\pi}}\lf|
 \int_{-\delta}^{2^k - \delta}  \cf^{-1} \phi_1 (u)\, du \r|
 \end{equation}
and there exists a positive constant $C_1$ such that, for any $k \in \nn$,
 \[
 \f{1}{\sqrt{2\pi}}\int_{\{z \in \rr:\  |z| > 2^k-1\}} | \cf^{-1} \phi_1 (u)|\, du  \le C_1 2^{-k}.
 \]
Thus, by this, $\delta\in(0,1)$, and \eqref{new6}, we conclude that, for any $k \in \nn$,
\begin{align}\label{eq-star4}
\lf|S_k  \mathbf{1} ( 1- 2^{-k+1}\delta) \r|
&\ge  \f{1}{\sqrt{2\pi}}\lf| \int_{-\delta}^\infty  \cf^{-1} \phi_1 (u)\, du \r|\\
&\quad-
 \f{1}{\sqrt{2\pi}}\lf|\int_{\{z \in \rr:\  |z| > 2^k-1\}} \cf^{-1} \phi_1 (u)\, du\r| \nonumber\\
&\ge \f{1}{\sqrt{2\pi}}\lf| \int_{-\delta}^\infty  \cf^{-1} \phi_1 (u)\, du \r|- C_1 2^{-k}\nonumber.
\end{align}
Clearly,
\[
 \int_{-\delta}^\infty  \cf^{-1} \phi_1 (u)\, du =
  \int_{-\delta}^\delta   \cf^{-1} \phi_1 (u)\, du  +  \int_{\delta}^\infty  \cf^{-1} \phi_1 (u)\, du.
\]
Due to \eqref{eq-int=0} and the symmetry of $\cf^{-1}\phi_1$, we also have
\[
 \int_{\delta}^\infty  \cf^{-1} \phi_1 (u)\, du  = -
 \int_{0}^\delta   \cf^{-1} \phi_1 (u)\, du.
\]
Inserting these two equalities into \eqref{eq-star4}, we find that, for any $k \in \nn$,
\[
\lf|S_k  \mathbf{1} ( 1- 2^{-k+1}\delta) \r| \ge  \f{1}{\sqrt{2\pi}}\lf|
 \int_{-\delta}^{0}  \cf^{-1} \phi_1 (u)\, du \r|- C_12^{-k}.
  \]
Using this and applying \eqref{eq-continuity}, provided by the continuity of $\cf^{-1} \phi_1 $, with $\delta\in(0,1)$,
and the symmetry of $\cf^{-1}\phi_1$,
we obtain
\begin{equation}\label{eq-star5}
 \| S_k  \mathbf{1} \|_{L^\infty (\rr)}
  \ge \lf|S_k  \mathbf{1} ( 1- 2^{-k+1}\delta) \r|
  \ge\frac{{c}\delta}{2\sqrt{2\pi}}  - C_1 2^{-k} \ge C_2 >0,
\end{equation}
where we took $k \ge k_0$ and $k_0$ is a positive integer big enough which depends only on both $c$ and $\delta$.

Now, we turn  to the case $n=2$.
By \eqref{new7}, we find that, for any given $\varepsilon \in(0,\fz)$,
there exists some $k_1(\vp)\in\nn$ such that, for any $k\ge k_1(\vp)$,
$$
S^k \mathbf{1}_1 (1-2^{-k+1}\delta)\ge 1-\f{\vp}{2},
$$
which, together with \eqref{new4}
 with $x_1:=1- 2^{-k+1}\delta =:x_2$, and
 \eqref{eq-star5} with $\mathbf{1}$ therein replaced by $\mathbf{1}_1$,
implies that, as long as $k \ge k_1 (\varepsilon)$,
\begin{align*}
 \lf\| S_k  \mathbf{1}_2 \r\|_{L^\infty (\rr)}
 &\ge\lf| S_k  \mathbf{1}_2 (1-2^{-k+1}\delta ,1-2^{-k+1}\delta)\r|\\
&=\lf|S_k  \mathbf{1}_1 ( 1-2^{-k+1}\delta) \r|
\lf|\lf(S^k \mathbf{1}_1
 +   S^{k-1}  \mathbf{1}_1\r) ( 1-2^{-k+1}\delta)\r|\nonumber\\
&\ge C_2 \lf|\lf(S^k \mathbf{1}_1 + S^{k-1}  \mathbf{1}_1\r) ( 1-\delta 2^{-(k-1)})\r|
\ge C_2 (2-\varepsilon).
\nonumber
\end{align*}
This finishes the proof of the case $n=2$.
\\
\noindent
{\textbf{Step 3.}} Let $n\ge 3$. In this case, we proceed by mathematical induction based on
the following identity [see \eqref{new8}]. We use the symbol $x:=(x',x_n)$ for any $x\in\rn$, where
$x' \in \rr^{n-1}$ and $x_n\in \rr$.
For simplicity, we denote $\mathbf{1}_{(-1,1)^{n-1}}$ by
$\mathbf{1}_{{n-1}}$ and $\mathbf{1}_{(-1,1)^{n}}$ by
$\mathbf{1}_{{n}}$.
Let $\{\wz{\phi}_k\}_{k\in\zz_+}$ be the functions defined on $\rr$ the same as in Step 2,
$$
 \Phi_0 (x'):= \wz{\phi}_0(x_1)\cdots \wz{\phi}_0(x_{n-1}),  \ \ \forall\,x':=(x_1,\ldots,x_{n-1}) \in \rr^{n-1},
$$
and
$$
 \phi_0 (x):= \Phi_0(x') \wz{\phi_0}(x_n), \ \ \forall\,x_n  \in \rr \,.
$$
Then, from an argument similar to that used in the estimation of \eqref{new4}, it follows that,
for any $k\in\nn$ and $x\in\rn$,
\begin{equation}\label{new8}
 S_k  \mathbf{1}_n (x) =  S_k  \mathbf{1}_{n-1} (x')  S^k \mathbf{1}_1 (x_n) +  S^{k-1}  \mathbf{1}_{n-1} (x')  S_k  \mathbf{1}_1 (x_n),
\end{equation}
where
$$
 S_k  \mathbf{1}_{n-1} (x') := (2\pi)^{-(n-1/2)} \int_{\rr^{n-1}}
 {\cf}^{-1} \Psi_k (x'-y') \mathbf{1}_{n-1} (y')\, dy',
$$
$$
 \Psi_1 (x') := \Phi_0\lf(\frac {x'}2\r)- \Phi_0 (x'),
$$
and, for any $x'\in\rr^{n-1}$ and $k\in\{2,3,\ldots\}$,
 $$
\Psi_k (x') := \Psi_1 (2^{-k+1}x').
$$
This proves the case $n\ge3$ and hence (i).\\
{\textbf{Step 4.}}
As we all know, Besov spaces and also logarithmic Besov spaces are invariant under rotations and translations.
To prove (ii), if we assume that $\mathbf{1}_{\rn_+}$ belongs to $M(B^{0,b}_{1,\fz}(\rn))$,
then it follows immediately that  $\mathbf{1}_{(-1,1)^n}$
also belongs to $M(B^{0,b}_{1,\fz}(\rn))$.
Since this is not true, we end up with a contradiction.
This finishes the proof of (ii) and hence Theorem \ref{rn+-}.
\end{proof}

Now, we turn to show Theorem \ref{rn++}.

\begin{proof}[Proof of Theorem \ref{rn++}]
Obviously, by Theorem \ref{p=infty}, functions in $M(B^{0,b}_{\infty,\infty}(\rn))$
must satisfy
$$\sup_{k\ge2} (1+k) \|S_k f \|_{L^\infty (\rn)} <\infty.$$
Thus, by an argument similar to that used in the proof of \cite[Proposition 18]{KS02},
we conclude that $\mathbf{1}_{E}\notin M(B^{0,b}_{1,\infty}(\rn))$.
This finishes the proof of Theorem \ref{rn++}.
\end{proof}

Finally, we show Theorem \ref{thm-Q1:binR}.

\begin{proof}[Proof of Theorem \ref{thm-Q1:binR}]
We apply the real interpolation with a function parameter; see, for instance,
Merucci \cite{Me84}, Cobos and Fernandez \cite{CF88}, Caetano and Moura \cite{CM04}, and
Almeida \cite{Al}.
It is known that, for any $p,q\in(0,\fz]$, $b \in \rr$,  and $s, s_1,s_2 \in \rr$ with
$s_1>s> s_2$,
$$
\lf(B^{s_1}_{p,1}(\rn),B^{s_2}_{p,1}(\rn) \r)_{g,q} = B^{s,b}_{p,q}(\rn),
$$
where $g$ is a function parameter appropriately chosen; see \cite[Proposition 7]{Al}.
Then Theorem \ref{thm-Q1:binR} follows from the interpolation property of this method;
see \cite[Theorem 2.8.7]{Tr83}.
This finishes the proof of Theorem \ref{thm-Q1:binR}.
\end{proof}


\subsubsection{Proofs Related to  Continuous Functions\,--\,Corollaries \ref{approx3} and \ref{approx4}}\label{subsec-pf.eg.2}


Next, we turn to the proofs about continuous functions.

\begin{proof}[Proof of Corollary \ref{approx3}]
To prove (i), by \cite[Theorem 3.9]{CD16}, we know that, for any $b\in(0,\fz)$,
the continuous embedding
$\mathbf{B}^{0,b}_{\infty,\infty}(\rn) \hookrightarrow
{B}^{0,b}_{\infty,\infty}(\rn)$ holds true.
Thus, Corollary \ref{approx3}(i) follows from this, Corollary \ref{approx}(i), and the obvious embedding
$\mathbf{B}^{0,1}_{\infty,\infty}(\rn)\hookrightarrow L^{\fz}(\rn)$.

Now, we prove (ii).
By the monotonicity of the  modulus of smoothness $\omega_1(f,\cdot)_{\fz}$
[see \eqref{eq-def-w}], we have
\begin{align*}
\|f\|_{\mathbf{B}^{0,1}_{\fz,\fz}(\rn)}&\sim\sup_{j\in\zz_+}(1+j)\omega_1(f,2^{-j})_{\fz}
\ls\sup_{j\in\zz_+}\sum_{k=0}^j\omega_1(f,2^{-k})_{\fz}\\
&\sim\|f\|_{\mathbf{B}^{0,0}_{\fz,1}(\rn)}
\sim\|f\|_{C_D(\rn)},
\end{align*}
which implies the embedding
\begin{equation}\label{eq-embed-CD}
C_D(\rn)\hookrightarrow \mathbf{B}^{0,1}_{\fz,\fz}(\rn)\hookrightarrow
\big[L^{\fz}(\rn)\cap B^{0,1}_{\fz,\fz}(\rn)\big].
\end{equation}
This, combined with Corollary \ref{approx}(iii), proves (ii) and hence Corollary \ref{approx3}.
\end{proof}

\begin{proof}[Proof of Corollary \ref{approx4}]
To prove (i), again we use the continuous embedding
$\mathbf{B}^{0,b}_{\infty,\infty}(\rn) \hookrightarrow
{B}^{0,b}_{\infty,\infty}(\rn)$ with $b\in(0,\fz)$. This and Corollary \ref{bp=infty} are enough for the case $b\in(1,\fz)$.
For the case $b\in(0,1)$, we have
\begin{align*}
&\sup_{l\in\zz_+} (1+l)^b  \sup_{\{P:\ l(P)=2^{-l}\}}
 \fint_P\sum_{k=l}^\infty (1+k)^{-b} \lf|S_kf(y)\r|\,dy\\
&\quad\le\sup_{l\in\zz_+} (1+l)^b
 \sum_{k=l}^\infty (1+k)^{-b-1}(1+k)\|S_kf\|_{L^\infty(\rn)}\\
&\quad\ls \sup_{k\in\zz_+}\,(1+k)\|S_kf\|_{L^\infty(\rn)}
\sim\|f\|_{B^{0,1}_{\fz,\fz}(\rn)},
\end{align*}
which, combined with Theorem \ref{p=infty},
implies $L^{\fz}(\rn)\cap{B}^{0,1}_{\infty,\infty}(\rn)\hookrightarrow M(B^{0,b}_{\fz,\fz}(\rn))$
and, therefore, by \eqref{eq-embed-CD},
$\mathbf{B}^{0,1}_{\infty,\infty}(\rn) \hookrightarrow M(B^{0,b}_{\fz,\fz}(\rn))$.
This proves (i).

The item (ii) follows from both (i) and \eqref{eq-embed-CD}.
This finishes the proof of Corollary \ref{approx4}.
\end{proof}


\subsubsection{Proofs Related to Exponentials\, --\, Theorems \ref{expo3}, \ref{expo4}, \ref{expo5}, and \ref{expo7}}\label{subsec-pf.eg.3}


To finish the proof of Theorem \ref{expo3},
first, we consider the slightly simpler situation:
$$f_k(x) := e^{i2^kx_1}, \ \ \forall\, x:= (x_1, \ldots , x_n)\in \rn,\ \forall\,k \in \nn.$$
This will make our proof more transparent and clear.

\begin{lemma}\label{expo}
Let $b\in \rr$.
Then, for any given $k\in \nn$, it holds true that
\begin{align*}
\lf\| e^{i2^kx_1}\r\|_{M(B^{0,b}_{1,\infty}(\rn))} \sim
\begin{cases}
\ (1+k)^b &\ \text{if}\  b\in(-\fz,-1)\cup(1,\fz),\\
\ (1+k)\ln (1+k) &\ \text{if}\  b=1,\\
\ (1+k) &\ \text{if}\  b\in[-1,1).
\end{cases}
\end{align*}
\end{lemma}

\begin{proof}
Let $k\in \nn$ and $\phi_k$ be the same as in \eqref{eq-phi1} and \eqref{eq-phik}.
Then $\phi_k (2^k,0, \ldots, 0) = 1$.
Observe that, for any $j \in \zz_+$ and $k\in \nn$, we have $\phi_j (2^k,0,\ldots,0) = \delta_{j,k}$,
where $\delta_{j,k}$ is the well-known Kronecker symbol.
We conclude that, for any $j \in \zz_+$, $k\in \nn$, and $x:=(x_1,\ldots,x_n)\in\rn$,
$$
S_j f_k (x) =\int_{\rn}\varphi_j(z)e^{i2^k(x_1-z_1)}\,dz = \delta_{j,k}  e^{i2^kx_1}.
$$
From this, we deduce that
\begin{align*}
\sup_{0 \le l \le k}\sum_{j = l}^\infty \lf(\f{1+l}{1+j}\r)^b
\lf\|S_jf_k\r\|_{L^{\fz}(\rn)} &=
\sup_{0 \le l \le k}\lf(\f{1+l}{1+k}\r)^b\\
& = \begin{cases}
\ 1 &\  \mbox{if}\ b\in[0,\fz),
\\
\ (1+k)^{-b} &\  \mbox{if}\  b\in(-\fz,0),
\end{cases}
\end{align*}
and, similarly, for any $k\ge2$,
\begin{align*}
&\sup_{j\ge2}\sum_{l=0}^{j-2}\lf(\f{1+j}{1+l}\r)^{b}
\sup_{\{P:~l(P)=2^{-l}\}}\fint_P\lf|S_jf_k(y)\r|\,dy\\
&\quad =
\sum_{l=0}^{k-2}\lf(\f{1+k}{1+l}\r)^{b} \sim
\begin{cases}
\ (1+k)^b &\ \mbox{if}\ b\in(1,\fz),
\\
\ (1+k)\ln (1+k) &\ \mbox{if}\ b=1,
\\
\ 1+k &\ \mbox{if}\ b\in(-\fz,1),
\end{cases}
\end{align*}
and, for $k=0$ or $k=1$,
$$
\sup_{j\ge2}\sum_{l=0}^{j-2}\lf(\f{1+j}{1+l}\r)^{b}
\sup_{\{P:~l(P)=2^{-l}\}}\fint_P\lf|S_jf_k(y)\r|\,dy=0.
$$
This, combined with Theorem \ref{p=1}, then finishes the proof of Lemma \ref{expo}.
\end{proof}

Based on Lemma \ref{expo} and by using some standard degrees of freedom
in the choice of the decomposition of unity, one can extend the above consideration to the family
$h_k (x) := e^{ikx_1}$ for any $x:= (x_1, \ldots , x_n)\in \rn$ and any given $k \in \nn$; see Lemma \ref{expo2} below.
Indeed, it will be convenient to construct, for any $k\in\nn$,
an adapted decomposition of unity $\{\wz{\phi}_j\}_{j\in\nn}$ such that the associated norms satisfy,
for any $g \in B^{0,b}_{1,\infty}(\rn)$,
\[
\|g \|_{B^{0,b}_{1,\infty}(\rn)}^{\phi} \sim
\| g \|_{B^{0,b}_{1,\infty}(\rn)}^{\wz{\phi}}
\]
with the positive equivalence constants independent of $k$, and, for any $j\in \zz_+$,
$
\phi_{j}(k,0, \ldots, 0) \in \{0,1\}$,
where $\phi_{j}(k,0, \ldots, 0) =1$ is true exactly for one $j$ depending on $k$.
This implies that, for any $x\in \rn$, either $S_j h_k(x) \equiv 0$ or $S_j h_k(x) = h_k (x)$.
 Notice that the rough relation between $k$ and $j$ here will be $j \sim \log k$.  We omit further details.

\begin{lemma}\label{expo2}
Let $b\in \rr$.
Then, for any given $k\in \nn$, it holds true that
\begin{align*}
\lf\| e^{ikx_1}\r\|_{M(B^{0,b}_{1,\infty}(\rn))} \sim
\begin{cases}
\ (1+\ln k)^b &\  \text{if}\  b\in(-\fz,-1)\cup(1,\fz),\\
\ (1+\ln k)\ln \lf(1+\ln (1+k)\r) &\  \text{if}\  b=1,\\
\ (1+\ln k) &\  \text{if}\  b\in[-1,1).
\end{cases}
\end{align*}
\end{lemma}

In a similar way one can further deal with the most general situation of
$H_{(k_1, \ldots ,k_n)} (x):= e^{ik\cdot x},\ \forall\,x\in \rn$ and confirm Theorem \ref{expo3};
 here, the rough relation between $k:=(k_1,\ldots ,k_n)\in\zz_+^n\setminus\{\mathbf{0}\}$ and $j$
will be $j \sim \log |k|$.
We omit its proof.


To prove Theorem \ref{expo4},
basically, we may use the same strategy as that used in the proof of Theorem \ref{expo3}.
We also omit further details.


Now, we concentrate on the case for $p\in(1,\fz)$.

\begin{proof} [Proof of Theorem \ref{expo5}]
To prove its sufficiency,  we employ  Lemmas \ref{lem-suff-I1}, \ref{lem-suff-I2},
and \ref{lem-suff-5}. To prove its necessity, we use Theorem \ref{thm-suff-p}(ii);
the further details are omitted.
This finishes the proof of Theorem \ref{expo5}.
\end{proof}


\begin{proof}[Proof of Theorem \ref{expo7}]

{\textbf{Step 1.}}
The direction '$\ls$' follows from  Lemmas \ref{lem-suff-I1}, \ref{lem-suff-I2},
and \ref{lem-suff-4}.
In the following, to prove the direction '$\gtrsim$',
we will introduce a special family of test functions.
\\
{\textbf{Step 2.}} We shall deal with the simplified situation $k:=(k_1,0 , \ldots \, , 0)$
and $k_1 := 2^m$ for some $m\ge 2$.
\\
{\textbf{Substep 2.1.}} Preparations.
The following family of functions will be studied:
\begin{equation*}
g_{m,\alpha}(x):= \Psi (x)\sum_{j=1}^m \alpha_j  e^{i2^jx_1} , \ \forall\, x\in \rn, \ \forall\,m\in \nn ,
\end{equation*}
where the function $\Psi \in \mathcal{S} (\rn)$ has a compactly supported Fourier transform, more exactly,
$$
\supp \cf {\Psi}\subset \{\xi \in \rn:~ 3/2\le |\xi|\le 2\},
$$
and $\alpha:= \{\alpha_j\}_{j=1}^\infty$ is a sequence belonging to $\ell^\infty$.

For later use we need to specify $\Psi$ further.
We claim that there exists a radial real-valued function $\Psi$ such that
\begin{equation}
 \label{positive} \Psi (x) \ge C >0\qquad \mbox{on}\qquad
[-\pi-1, \pi+1]^n,
\end{equation}
where $C$ is a positive constant.
Indeed,
a possible construction runs as follows.
Let $H \in C_{\rm c}^\infty (\rn)$ be a radial real-valued  function such that, for some $A,B\in(0,\fz)$ with $A < B$,
$\supp H \subset \{\xi \in \rn:~ A\le |\xi|\le B\}$ and
$H >0$ on $\{\xi \in \rn:~ A<  |\xi|< B\}$.
Then $\mathcal{F}^{-1} H$ is real-valued as well and, moreover,
$$
\mathcal{F}^{-1} H (\mathbf{0})= (2\pi)^{-n/2} \int_{\rn}  H(\xi)\, d\xi = C ,
$$
where $C$ is a positive constant.
Since $\mathcal{F}^{-1} H$ is smooth, it follows that there exists some $\varepsilon \in(0,\fz)$ such that,
for any $x\in\rn$ with $|x|\le \varepsilon (\pi + 1)$,
$\mathcal{F}^{-1} H (x)\ge C/2$.
By scaling, we know that, for any $x\in \rn$,
$\varepsilon^{-n} \mathcal{F}^{-1} H(\vp^{-1}\cdot)(x) = (\mathcal{F}^{-1} H) (\varepsilon x)$.
Thus, we conclude that
$$
\supp  H(\varepsilon^{-1} \cdot) \subset \lf\{\xi \in \rn:\ \varepsilon\, A\le |\xi|\le \varepsilon \, B\r\}
$$
and, if $|x|\le \pi + 1$, then
$$
 \varepsilon^{-n} \mathcal{F}^{-1} H(\vp^{-1} \cdot)(x) \ge \frac C 2.
$$
Now, we choose $A:=3/(2 \varepsilon)$ and $B: = 2/\varepsilon$, and define
$\Psi (x) := \varepsilon^{-1} \mathcal{F}^{-1} H(\varepsilon^{-1}\cdot)(x)$
for any $x\in\rn$.
This proves the above claim.
\\
{\textbf{ Substep 2.2.}} Suppose $m\ge 3$. We define, for any $x:=(x_1,\ldots,x_n)\in\rn$,
\begin{equation}\label{eq-def-g}
f(x):= e^{-i 2^m x_1}\ \ \text{and}\ \
g(x):=g_{m-2,\alpha}(x):= \Psi (x) \sum_{j=1}^{m-2} \alpha_j e^{i2^jx_1} ,
\end{equation}
which implies that, for any $x\in\rn$,
$$
f(x) g(x) = \Psi (x) \sum_{j=1}^{m-2} \alpha_j
 e^{i (2^j - 2^m)x_1}.
$$
Next, we compare the norms of these two functions $g$ and $f g$.
On the one hand we have
$$
\cf \Big(S_k\lf(\Psi e^{i 2^jx_1}\r)\Big)(x)
=\phi_k(x)  \lf(\cf {\Psi}\r)(x_1-2^j,x_2,\ldots,x_n)
$$
for all $k\in\zz_+$, $x\in\rn$, and $j\in\{1, \ldots \, , m-2\}$;
on the other hand we derive by using the requirement concerning the
support of $\widehat{\Psi}$
$$
S_kg=\sum_{j=\max\{1,k-2\}}^{\min\{m-2,k+2\}} \alpha_jS_k\lf(\Psi e^{i 2^jx_1}\r)
$$
and hence
\begin{equation}\label{eq-|g|}
 \| g \|_{B^{0,b}_{p,\infty}(\rn)} \ls
 \sup_{1\le j\le m-2}  (1+j)^b  |\alpha_j|\|\Psi \|_{L^{p}(\rn)} .
\end{equation}
From \eqref{eq-supp-2} and
$$
\supp \cf ( f g) \subset \lf[-2^{m+1}, -2^{m-1}\r] \times \lf\{\xi' \in \rr^{n-1}:~ |\xi'|\le 2\r\},
$$
we infer that
$$
\lf\|f g\r\|_{L^p(\rn)}=\lf\|\sum_{k=m-2}^{m+2}S_k(f g)\r\|_{L^p(\rn)},
$$
which implies that
\begin{align*}
 \| f  g\|_{B^{0,b}_{p,\infty}(\rn)}
& = \sup_{m-2\le k\le m+2}(1+k)^b\lf\|S_k\lf(f g\r)\r\|_{L^p(\rn)}
 \sim (1+m)^b\lf\|f g\r\|_{L^p(\rn)}\\
&\sim
(1+m)^b \lf\|\Psi (x) \sum_{j=1}^{m-2} \alpha_j e^{i(2^j- 2^m)x_1}\r\|_{L^{p}(\rn)}.
\end{align*}
We add a short comment to the first $\sim$ in the previous line.
By the Littlewood--Paley characterization of $L^p(\rn)$ we obtain
\[
\lf\|f g\r\|_{L^p(\rn)} \sim
\lf\|\lf[\sum_{k=m-2}^{m+2} |S_k(f g)|^2\r]^{\f12}\r\|_{L^p(\rn)}
\ls \max_{m-2\le k \le m+2} \lf\|S_k (f g) \r\|_{L^p(\rn)}
\]
and
\[
\max_{m-2\le k \le m+2}
\lf\|S_l (f g) \r\|_{L^p(\rn)}  \ls
\lf\|\lf[\sum_{k=m-2}^{m+2} |S_k(f g)|^2\r]^{\f12}\r\|_{L^p(\rn)}
\sim \lf\|f g \r\|_{L^p(\rn)} .
\]
Again, using the Littlewood--Paley characterization of $L^p(\rn)$ supplemented by   \eqref{positive}  we conclude that
\begin{align*}
&\lf\|\Psi (x) \sum_{j=1}^{m-2} \alpha_j e^{i(2^j- 2^m)x_1} \r\|_{L^{p}(\rn)}
\\
&\quad \gtrsim \lf\|\sum_{j=1}^{m-2} \alpha_j e^{i(2^j- 2^m)x_1}\r\|_{L^{p}([-\pi,\pi]^n)}
\sim (2\pi)^{\f{n-1}{p}} \lf\|\sum_{j=1}^{m-2} \alpha_j e^{i(2^j- 2^m)x_1}\r\|_{L^{p}([-\pi,\pi])}
\\
&\quad\sim (2\pi)^{\f{n-1}{p}}
\lf\|\sum_{j=1}^{m-2} \alpha_j e^{i 2^jx_1}\r\|_{L^{p}([-\pi,\pi])}\\
&\quad \sim \lf\|\lf[\sum_{j=1}^{m-2} \lf|\alpha_j e^{i 2^jx_1}\r|^2 \r]^{\f12}\,\r\|_{L^{p}([-\pi,\pi])}
\sim\lf(\sum_{j=1}^{m-2} |\alpha_j|^2 \r)^{\f12} .
\end{align*}
For any given $k\in \zz^n$, we put $I_k:= 2\pi k + [-\pi,\pi]^n$.
Vice versa, using $\Psi \in \mathcal{S}(\rn)$ and periodicity we find that
\begin{align*}
&\lf\|\Psi (x)\sum_{j=1}^{m-2} \alpha_j e^{i(2^j- 2^m)x_1}\r\|_{L^{p}(\rn)}^p\\
&\quad\le C_{(M)}\sum_{k \in \zz^n} (1+|k|)^{-M}
\int_{I_k} \lf| \sum_{j=1}^{m-2} \alpha_j e^{i(2^j- 2^m)x_1}\r|^p dx\\
&\quad\ls\int_{I_0} \lf|  \sum_{j=1}^{m-2} \alpha_j e^{i(2^j- 2^m)x_1}\r|^p dx
 \ls\int_{-\pi}^\pi \lf|  \sum_{j=1}^{m-2} \alpha_j e^{i(2^j- 2^m)x_1}\r|^p dx_1\\
&\quad \sim \int_{-\pi}^\pi \lf|  \sum_{j=1}^{m-2} \alpha_j e^{i2^jx_1}\r|^p dx_1
\sim\int_{-\pi}^\pi \lf(  \sum_{j=1}^{m-2} |\alpha_j e^{i2^jx_1}|^2\r)^{\f p2} dx_1\\
&\quad\sim \lf[\sum_{j=1}^{m-2} |\alpha_j|^2\r]^{\f p2},
\end{align*}
where $M$ is some natural number chosen large enough and $C_{(M)}$ denotes a positive constant depending on $M$.
This implies
\begin{align}\label{eq-|fg|}
 \|f  g\|_{B^{0,b}_{p,\infty}(\rn)}
 &\sim (1+m)^b \lf\|\Psi (x) \sum_{j=1}^{m-2} \alpha_j e^{i(2^j- 2^m)x_1} \r\|_{L^{p}(\rn)}\\
 &\sim   (1+m)^b \lf[ \sum_{j=1}^{m-2} |\alpha_j|^2\r]^{1/2}\nonumber.
\end{align}
{\textbf{Substep 2.3.}}
To complete the proof of this simplified situation, we will use \eqref{eq-|g|}, \eqref{eq-|fg|}, and Lemma \ref{lem-sum-log}.
We need to distinguish four cases on $b$.
\\
{\em Case 1.}
Let $b\in[0,1/2)$.
The particular choice $\alpha_j =1$, $j\in\{1, \ldots \, , m-2\}$, implies that
$$
\lf\|e^{-i 2^m x_1}\r\|_{M(B^{0,b}_{p,\infty}(\rn))} \ge \frac{\| f  g\|_{B^{0,b}_{p,\infty}(\rn)}}{\| g\|_{B^{0,b}_{p,\infty}(\rn)}}
\sim \frac{(1+m)^b \sqrt{1+m}}{(1+m)^b}\sim \sqrt{m}
$$
\\
{\em Case 2.} Let $b = 1/2$.
The particular choice $\alpha_j =(1+j)^{-1/2}$, $j\in\{1, \ldots \, , m-2\}$, implies that
$$
\lf\| e^{-i 2^m x_1}\r \|_{M(B^{0,\f12}_{p,\infty}(\rn))}
\gtrsim \sqrt{1+m}\lf[\sum_{j=1}^{m-2}(1+j)^{-1}\r]^{\f12}
\sim \sqrt{1+m}\sqrt{\ln (1+m)}.
$$
\\
{\em Case 3.} Let $b \in(1/2,\fz)$.
By choosing $\alpha_j =(1+j)^{-b}$, $j\in\{1, \ldots , m-2\}$,
we have
$$
\lf\| e^{-i 2^m x_1}\r \|_{M(B^{0,b}_{p,\infty}(\rn))}
 \gtrsim (1+m)^b\,\lf[\sum_{j=1}^{m-2}\f{1}{(1+j)^{2b}}\r]^{\f12}
\sim (1+m)^b.
$$
\\
{\em Case 4.} Let $b\in[- 1/2,0)$.
Choose  $\alpha_j = (1+j)^{-b}$, $j\in\{1, \ldots , m-2\}$.
Then we obtain
$$
\lf\| e^{-i 2^m x_1}\r  \|_{M(B^{0,b}_{p,\infty}(\rn))}
 \gtrsim (1+m)^b\lf[\sum_{j=1}^{m-2}\f{1}{(1+j)^{2b}}\r]^{\f12}
\sim \sqrt{1+m}.
 $$
\\
{\em Case 5.} Let $b \in(-\fz,- 1/2)$. This time we are forced to choose another test function.
Let $m\in\zz_+$ and
$$
 \wz{g}(x):= e^{i 2^m x_1} \Psi (x) , \ \forall\, x:=(x_1,\ldots,x_n)\in \rn .
$$
Here, $\wz{g}$ is just the function $g_{m,\alpha}$ in \eqref{eq-def-g} with $\{\alpha_j:=\delta_{j,m}\}_{j\in\nn}$.
Then, by an argument similar to that used in the estimation of \eqref{eq-|g|},
we find that
$$
 \|\wz{g}\|_{B^{0,b}_{p,\infty}(\rn)} \ls (1+m)^b \|\Psi\|_{L^{p}(\rn)}.
$$
Notice that
$$
\lf\| f \wz{g}\r\|_{B^{0,b}_{p,\infty}(\rn)}
= \lf\|\Psi\r\|_{B^{0,b}_{p,\infty}(\rn)} \sim 1.
$$
Thus, we obtain
$$
\lf\|  e^{-i 2^m x_1} \r\|_{M(B^{0,b}_{p,\infty}(\rn))} \ge
\frac{\|f \wz{g}\|_{B^{0,b}_{p,\infty}(\rn)}}{\|\wz{g}\|_{B^{0,b}_{p,\infty}(\rn)}}
\sim \frac{1}{(1+m)^b}\sim (1+m)^{|b|},
$$
which completes the proof of this case.
\\
{\textbf{Step 3.}} The general case $h(x):= e^{ik\cdot x}$ for any $x\in \rn$ and any given $k \in \zz^n\setminus \{\mathbf{0}\}$ follows in a similar way. Indeed,
the argument explained in Step 2 is robust under small changes.
For instance, if we replace the test function $g_{m,\alpha}$ by
\[
G_{m,\alpha} (x):=  \Psi (x)  \sum_{j=j_0}^{m-2} \alpha_j e^{i\beta_j\cdot  x} , \ \forall\, x\in \rn ,
\]
where $\beta_j \in \zz^n$,
\[
|\beta_j - (2^j, 0, \ldots , 0)|\le \sqrt{n}, \ \forall\, j \in\{j_0, \ldots , m-2\},
\]
and $j_0$ is an appropriate fixed natural number chosen in dependence of $n$, then
we can apply the above arguments as well and obtain both \eqref{eq-|g|} and
\eqref{eq-|fg|} with $g$ therein replaced by $G$.
Now, we turn to the general case.
Let $k \in \zz^n \setminus \{\mathbf{0}\}$ satisfy
$2^m \le |k| < 2^{m-1}$. Then we choose a sequence $\beta_{j_0}, \ldots , \beta_{m-2} \in \zz^n$ such that
\[
\Big|\beta_j - 2^{j} \frac{k}{|k|}\Big|\le \sqrt{n}, \ \forall\, j \in\{j_0, \ldots , m-2\},
\]
and define $G_{m,\alpha} $ as above.
As mentioned, logarithmic Besov spaces are invariant under rotations.
By means of a rotation (we map the beam $\{\lambda k:\  \lambda \in(0,\fz) \}$
onto $\{(x_1, 0, \ldots, 0):\  x_1 \in(0,\fz)\}$),
we return to the previously discussed situation;
we omit further details.
This finishes the  proof of Theorem \ref{expo7}.
\end{proof}

\begin{remark}
 Functions as defined in \eqref{eq-def-g} have been investigated by Triebel
 in the framework of Triebel-Lizorkin and Besov spaces in the proof of
 \cite[Theorem 2.3.9]{Tr83}.
\end{remark}



\bigskip

\noindent Ziwei Li, Dachun Yang and Wen Yuan

\medskip

\noindent Laboratory of Mathematics and Complex Systems (Ministry of Education of China),
School of Mathematical Sciences, Beijing Normal University, Beijing 100875, People's Republic of China

\smallskip

\noindent{\it E-mails:} \texttt{zwli@mail.bnu.edu.cn} (Z. Li)

\noindent\phantom{{\it E-mails:} }\texttt{dcyang@bnu.edu.cn} (D. Yang)

\noindent\phantom{{\it E-mails:} }\texttt{wenyuan@bnu.edu.cn} (W. Yuan)

\bigskip

\noindent Winfried Sickel

\medskip

\noindent  Institute of Mathematics, Friedrich-Schiller-University Jena,
Ernst-Abbe-Platz 2, Jena 07737, Germany

\smallskip

\noindent{\it E-mail:} \texttt{winfried.sickel@uni-jena.de}

\end{document}